\documentclass[a4paper, 11pt, oneside, reqno]{amsart}

\usepackage[english]{babel}
\usepackage{amsthm}
\usepackage{amsmath}
\usepackage{amssymb}
\usepackage[ansinew]{inputenc}
\usepackage{enumerate}
\usepackage{graphicx}
\usepackage[english,intoc]{nomencl}
\usepackage{fancyhdr}
\usepackage{cancel}
\usepackage{mdwlist}
\usepackage{bbm}
\usepackage[titletoc]{appendix}
\usepackage[Lenny]{fncychap}
\usepackage{xcolor}   
\usepackage{hyperref}
\hypersetup{
    colorlinks=false, 
    linktoc=page,     
    linkbordercolor=blue,  
    pdfborderstyle={/S/U/W 1}    }   
\usepackage{geometry}
\usepackage{stmaryrd}
\usepackage{tikz}
\usetikzlibrary{arrows}
\usepackage{float}

\newcommand{\cF}{\mathcal{F}}
\newcommand{\cC}{\mathcal{C}}
\newcommand{\cI}{\mathcal{I}}

\newcommand{\cH}{\mathcal{H}}

\newcommand{\cG}{\mathcal{G}}

\newcommand{\N}{\mathbb{N}}

\newcommand{\id}{\mathord{\text{\rm id}}}

\newcommand{\C}{\mathbb{C}}

\newcommand{\G}{\mathbb{G}}

\newcommand{\al}{\alpha}
\newcommand{\R}{\mathbb{R}}

\newcommand{\Tr}{\operatorname{Tr}}

\newcommand{\ot}{\otimes}

\newcommand{\F}{\mathbb{F}}

\newcommand{\cB}{\mathcal{B}}

\newcommand{\1}{\mathbbm 1}
\newcommand{\cP}{\mathcal{P}}
\newcommand{\cQ}{\mathcal{Q}}
\newcommand{\be}{\beta}

\DeclareMathOperator{\spann}{span}

\DeclareMathOperator{\Rep}{Rep}
\DeclareMathOperator{\Irr}{Irr}
\DeclareMathOperator{\Obj}{Obj}
\DeclareMathOperator{\Pol}{Pol}
\DeclareMathOperator{\Mor}{Mor}

\DeclareMathOperator{\tr}{tr}
\DeclareMathOperator{\kr}{kr}
\DeclareMathOperator{\TLJ}{TLJ}

\numberwithin{equation}{section}

\title[\resizebox{5.0in}{!}{Free wreath product quantum groups and standard invariants of subfactors}]{Free wreath product quantum groups and standard invariants of subfactors}

\author{Pierre Tarrago}
\address{Pierre Tarrago
\newline Centro de Investigaci\`on en Matematicas
\newline Guanajuato, Mexico}
\email{pierre.tarrago@cimat.mx}

\author{Jonas Wahl}
\thanks{JW is supported by European Research Council Consolidator Grant 614195}
\address{Jonas Wahl
\newline KU Leuven, Department of Mathematics
\newline Celestijnenlaan 200B -- Box 2400, B-3001 Leuven, Belgium}
\email{jonas.wahl@kuleuven.be}

\newtheorem{thmx}{Theorem}
\newtheorem{corx}[thmx]{Corollary}

\newtheorem{thm}{Theorem}[section]
\newtheorem{lem}[thm]{Lemma}
\newtheorem{cor}[thm]{Corollary}
\newtheorem{prop}[thm]{Proposition}

\theoremstyle{definition}

\newtheorem{defi}[thm]{Definition}
\newtheorem{rem}[thm]{Remark}
\newtheorem{ex}[thm]{Example}

\usepackage{calc}
\newcounter{PartitionDepth}
\newcounter{PartitionLength}

\newsavebox{\boxidpartww}
   \savebox{\boxidpartww}
   { \begin{picture}(0.5,1.2)
     \put(0.2,0.15){\line(0,1){0.7}}
     \put(0,-0.2){$\circ$}
     \put(0,0.8){$\circ$}
     \end{picture}}
\newsavebox{\boxidpartbw}
   \savebox{\boxidpartbw}
   { \begin{picture}(0.5,1.2)
     \put(0.2,0.15){\line(0,1){0.7}}
     \put(0,-0.2){$\circ$}
     \put(0,0.8){$\bullet$}
     \end{picture}}
\newsavebox{\boxidpartwb}
   \savebox{\boxidpartwb}
   { \begin{picture}(0.5,1.2)
     \put(0.2,0.15){\line(0,1){0.7}}
     \put(0,-0.2){$\bullet$}
     \put(0,0.8){$\circ$}
     \end{picture}}
\newsavebox{\boxidpartbb}
   \savebox{\boxidpartbb}
   { \begin{picture}(0.5,1.2)
     \put(0.2,0.15){\line(0,1){0.7}}
     \put(0,-0.2){$\bullet$}
     \put(0,0.8){$\bullet$}
     \end{picture}}

\newsavebox{\boxidpartsingletonbb}
   \savebox{\boxidpartsingletonbb}
   { \begin{picture}(0.5,1.2)
     \put(0.2,0.15){\line(0,1){0.2}}
     \put(0.2,0.65){\line(0,1){0.2}}
     \put(0,-0.2){$\bullet$}
     \put(0,0.8){$\bullet$}
     \end{picture}}
\newsavebox{\boxidpartsingletonww}
   \savebox{\boxidpartsingletonww}
   { \begin{picture}(0.5,1.2)
     \put(0.2,0.15){\line(0,1){0.2}}
     \put(0.2,0.65){\line(0,1){0.2}}
     \put(0,-0.2){$\circ$}
     \put(0,0.8){$\circ$}
     \end{picture}}
     
\newsavebox{\boxpaarpartbb}
   \savebox{\boxpaarpartbb}
   { \begin{picture}(1,1)
     \put(0.2,0.2){\line(0,1){0.4}}
     \put(0.7,0.2){\line(0,1){0.4}}
     \put(0.2,0.6){\line(1,0){0.5}}
     \put(0.5,-0.2){$\bullet$}
     \put(0,-0.2){$\bullet$}
     \end{picture}}
\newsavebox{\boxpaarpartww}
   \savebox{\boxpaarpartww}
   { \begin{picture}(1,1)
     \put(0.2,0.2){\line(0,1){0.4}}
     \put(0.7,0.2){\line(0,1){0.4}}
     \put(0.2,0.6){\line(1,0){0.5}}
     \put(0.5,-0.2){$\circ$}
     \put(0,-0.2){$\circ$}
     \end{picture}}
\newsavebox{\boxpaarpartbw}
   \savebox{\boxpaarpartbw}
   { \begin{picture}(1,1)
     \put(0.2,0.2){\line(0,1){0.4}}
     \put(0.7,0.2){\line(0,1){0.4}}
     \put(0.2,0.6){\line(1,0){0.5}}
     \put(0.5,-0.2){$\circ$}
     \put(0,-0.2){$\bullet$}
     \end{picture}}
\newsavebox{\boxpaarpartwb}
   \savebox{\boxpaarpartwb}
   { \begin{picture}(1,1)
     \put(0.2,0.2){\line(0,1){0.4}}
     \put(0.7,0.2){\line(0,1){0.4}}
     \put(0.2,0.6){\line(1,0){0.5}}
     \put(0.5,-0.2){$\bullet$}
     \put(0,-0.2){$\circ$}
     \end{picture}}
\newsavebox{\boxbaarpartbb}
   \savebox{\boxbaarpartbb}
   { \begin{picture}(1,1)
     \put(0.2,-0.1){\line(0,1){0.4}}
     \put(0.7,-0.1){\line(0,1){0.4}}
     \put(0.2,-0.1){\line(1,0){0.5}}
     \put(0.5,0.3){$\bullet$}
     \put(0,0.3){$\bullet$}
     \end{picture}}
\newsavebox{\boxbaarpartww}
   \savebox{\boxbaarpartww}
   { \begin{picture}(1,1)
     \put(0.2,-0.1){\line(0,1){0.4}}
     \put(0.7,-0.1){\line(0,1){0.4}}
     \put(0.2,-0.1){\line(1,0){0.5}}
     \put(0.5,0.3){$\circ$}
     \put(0,0.3){$\circ$}
     \end{picture}}
\newsavebox{\boxbaarpartbw}
   \savebox{\boxbaarpartbw}
   { \begin{picture}(1,1)
     \put(0.2,-0.1){\line(0,1){0.4}}
     \put(0.7,-0.1){\line(0,1){0.4}}
     \put(0.2,-0.1){\line(1,0){0.5}}
     \put(0.5,0.3){$\circ$}
     \put(0,0.3){$\bullet$}
     \end{picture}}
\newsavebox{\boxbaarpartwb}
   \savebox{\boxbaarpartwb}
   { \begin{picture}(1,1)
     \put(0.2,-0.1){\line(0,1){0.4}}
     \put(0.7,-0.1){\line(0,1){0.4}}
     \put(0.2,-0.1){\line(1,0){0.5}}
     \put(0.5,0.3){$\bullet$}
     \put(0,0.3){$\circ$}
     \end{picture}}
\newsavebox{\boxAspace}
   \savebox{\boxAspace}
   { \begin{picture}(0.1,1.5)
     \end{picture}}
\newsavebox{\boxBspace}
   \savebox{\boxBspace}
   { \begin{picture}(0,1.85)
     \end{picture}}
\newsavebox{\boxcutpaarpartbb}
   \savebox{\boxcutpaarpartbb}
   { \begin{picture}(1,1)
     \put(0,0.2){\line(0,1){0.4}}
     \put(0.8,0.2){\line(0,1){0.4}}
     \put(0,0.6){\line(1,0){0.2}}
     \put(0.6,0.6){\line(1,0){0.2}}
     \put(0.6,-0.2){$\bullet$}
     \put(-0.2,-0.2){$\bullet$}
     \end{picture}}
\newsavebox{\boxcutpaarpartww}
   \savebox{\boxcutpaarpartww}
   { \begin{picture}(1,1)
     \put(0,0.2){\line(0,1){0.4}}
     \put(0.8,0.2){\line(0,1){0.4}}
     \put(0,0.6){\line(1,0){0.2}}
     \put(0.6,0.6){\line(1,0){0.2}}
     \put(0.6,-0.2){$\circ$}
     \put(-0.2,-0.2){$\circ$}
     \end{picture}}
\newsavebox{\boxcutpaarpartbw}
   \savebox{\boxcutpaarpartbw}
   { \begin{picture}(1,1)
     \put(0,0.2){\line(0,1){0.4}}
     \put(0.8,0.2){\line(0,1){0.4}}
     \put(0,0.6){\line(1,0){0.2}}
     \put(0.6,0.6){\line(1,0){0.2}}
     \put(0.6,-0.2){$\circ$}
     \put(-0.2,-0.2){$\bullet$}
     \end{picture}}
\newsavebox{\boxcutpaarpartwb}
   \savebox{\boxcutpaarpartwb}
   { \begin{picture}(1,1)
     \put(0,0.2){\line(0,1){0.4}}
     \put(0.8,0.2){\line(0,1){0.4}}
     \put(0,0.6){\line(1,0){0.2}}
     \put(0.6,0.6){\line(1,0){0.2}}
     \put(0.6,-0.2){$\bullet$}
     \put(-0.2,-0.2){$\circ$}
     \end{picture}}

\newsavebox{\boxsingletonw}
   \savebox{\boxsingletonw}
   { \begin{picture}(0.5, 1)
     \put(0,0.3){$\uparrow$}
     \put(0,-0.2){$\circ$}
     \end{picture}}
\newsavebox{\boxsingletonb}
   \savebox{\boxsingletonb}
   { \begin{picture}(0.5, 1)
     \put(0,0.3){$\uparrow$}
     \put(0,-0.2){$\bullet$}
     \end{picture}}
\newsavebox{\boxdownsingletonw}
   \savebox{\boxdownsingletonw}
   { \begin{picture}(0.5, 1)
     \put(0,0){$\downarrow$}
     \put(0,0.6){$\circ$}
     \end{picture}}
\newsavebox{\boxdownsingletonb}
   \savebox{\boxdownsingletonb}
   { \begin{picture}(0.5, 1)
     \put(0,0){$\downarrow$}
     \put(0,0.6){$\bullet$}
     \end{picture}}

\newsavebox{\boxvierpartwbwb}
   \savebox{\boxvierpartwbwb}
   { \begin{picture}(1.7,1)
     \put(0.2,0.2){\line(0,1){0.4}}
     \put(0.6,0.2){\line(0,1){0.4}}
     \put(1,0.2){\line(0,1){0.4}}
     \put(1.4,0.2){\line(0,1){0.4}}
     \put(0.2,0.6){\line(1,0){1.2}}
     \put(0,-0.2){$\circ$}
     \put(0.4,-0.2){$\bullet$}
     \put(0.8,-0.2){$\circ$}
     \put(1.2,-0.2){$\bullet$}
     \end{picture}}
\newsavebox{\boxvierpartbwbw}
   \savebox{\boxvierpartbwbw}
   { \begin{picture}(1.7,1)
     \put(0.2,0.2){\line(0,1){0.4}}
     \put(0.6,0.2){\line(0,1){0.4}}
     \put(1,0.2){\line(0,1){0.4}}
     \put(1.4,0.2){\line(0,1){0.4}}
     \put(0.2,0.6){\line(1,0){1.2}}
     \put(0,-0.2){$\bullet$}
     \put(0.4,-0.2){$\circ$}
     \put(0.8,-0.2){$\bullet$}
     \put(1.2,-0.2){$\circ$}
     \end{picture}}
\newsavebox{\boxvierpartwwbb}
   \savebox{\boxvierpartwwbb}
   { \begin{picture}(1.7,1)
     \put(0.2,0.2){\line(0,1){0.4}}
     \put(0.6,0.2){\line(0,1){0.4}}
     \put(1,0.2){\line(0,1){0.4}}
     \put(1.4,0.2){\line(0,1){0.4}}
     \put(0.2,0.6){\line(1,0){1.2}}
     \put(0,-0.2){$\circ$}
     \put(0.4,-0.2){$\circ$}
     \put(0.8,-0.2){$\bullet$}
     \put(1.2,-0.2){$\bullet$}
     \end{picture}}
\newsavebox{\boxvierpartwbbw}
   \savebox{\boxvierpartwbbw}
   { \begin{picture}(1.7,1)
     \put(0.2,0.2){\line(0,1){0.4}}
     \put(0.6,0.2){\line(0,1){0.4}}
     \put(1,0.2){\line(0,1){0.4}}
     \put(1.4,0.2){\line(0,1){0.4}}
     \put(0.2,0.6){\line(1,0){1.2}}
     \put(0,-0.2){$\circ$}
     \put(0.4,-0.2){$\bullet$}
     \put(0.8,-0.2){$\bullet$}
     \put(1.2,-0.2){$\circ$}
     \end{picture}}
\newsavebox{\boxvierpartbwwb}
   \savebox{\boxvierpartbwwb}
   { \begin{picture}(1.7,1)
     \put(0.2,0.2){\line(0,1){0.4}}
     \put(0.6,0.2){\line(0,1){0.4}}
     \put(1,0.2){\line(0,1){0.4}}
     \put(1.4,0.2){\line(0,1){0.4}}
     \put(0.2,0.6){\line(1,0){1.2}}
     \put(0,-0.2){$\bullet$}
     \put(0.4,-0.2){$\circ$}
     \put(0.8,-0.2){$\circ$}
     \put(1.2,-0.2){$\bullet$}
     \end{picture}}
\newsavebox{\boxvierpartrotwbwb}
   \savebox{\boxvierpartrotwbwb}
   { \begin{picture}(1,1.2)
     \put(0.2,0.15){\line(0,1){0.2}}
     \put(0.2,0.65){\line(0,1){0.2}}
     \put(0.2,0.65){\line(1,0){0.5}}
     \put(0.2,0.35){\line(1,0){0.5}}
     \put(0.45,0.35){\line(0,1){0.3}}
     \put(0.7,0.15){\line(0,1){0.2}}
     \put(0.7,0.65){\line(0,1){0.2}}
     \put(0,0.8){$\circ$}
     \put(0.5,0.8){$\bullet$}
     \put(0,-0.2){$\circ$}
     \put(0.5,-0.2){$\bullet$}
     \end{picture}}
\newsavebox{\boxvierpartrotbwwb}
   \savebox{\boxvierpartrotbwwb}
   { \begin{picture}(1,1.2)
     \put(0.2,0.15){\line(0,1){0.2}}
     \put(0.2,0.65){\line(0,1){0.2}}
     \put(0.2,0.65){\line(1,0){0.5}}
     \put(0.2,0.35){\line(1,0){0.5}}
     \put(0.45,0.35){\line(0,1){0.3}}
     \put(0.7,0.15){\line(0,1){0.2}}
     \put(0.7,0.65){\line(0,1){0.2}}
     \put(0,0.8){$\bullet$}
     \put(0.5,0.8){$\circ$}
     \put(0,-0.2){$\circ$}
     \put(0.5,-0.2){$\bullet$}
     \end{picture}}
\newsavebox{\boxvierpartrotbwbw}
   \savebox{\boxvierpartrotbwbw}
   { \begin{picture}(1,1.2)
     \put(0.2,0.15){\line(0,1){0.2}}
     \put(0.2,0.65){\line(0,1){0.2}}
     \put(0.2,0.65){\line(1,0){0.5}}
     \put(0.2,0.35){\line(1,0){0.5}}
     \put(0.45,0.35){\line(0,1){0.3}}
     \put(0.7,0.15){\line(0,1){0.2}}
     \put(0.7,0.65){\line(0,1){0.2}}
     \put(0,0.8){$\bullet$}
     \put(0.5,0.8){$\circ$}
     \put(0,-0.2){$\bullet$}
     \put(0.5,-0.2){$\circ$}
     \end{picture}}
\newsavebox{\boxvierpartrotwwww}
   \savebox{\boxvierpartrotwwww}
   { \begin{picture}(1,1.2)
     \put(0.2,0.15){\line(0,1){0.2}}
     \put(0.2,0.65){\line(0,1){0.2}}
     \put(0.2,0.65){\line(1,0){0.5}}
     \put(0.2,0.35){\line(1,0){0.5}}
     \put(0.45,0.35){\line(0,1){0.3}}
     \put(0.7,0.15){\line(0,1){0.2}}
     \put(0.7,0.65){\line(0,1){0.2}}
     \put(0,0.8){$\circ$}
     \put(0.5,0.8){$\circ$}
     \put(0,-0.2){$\circ$}
     \put(0.5,-0.2){$\circ$}
     \end{picture}}
\newsavebox{\boxvierpartrotbbbb}
   \savebox{\boxvierpartrotbbbb}
   { \begin{picture}(1,1.2)
     \put(0.2,0.15){\line(0,1){0.2}}
     \put(0.2,0.65){\line(0,1){0.2}}
     \put(0.2,0.65){\line(1,0){0.5}}
     \put(0.2,0.35){\line(1,0){0.5}}
     \put(0.45,0.35){\line(0,1){0.3}}
     \put(0.7,0.15){\line(0,1){0.2}}
     \put(0.7,0.65){\line(0,1){0.2}}
     \put(0,0.8){$\bullet$}
     \put(0.5,0.8){$\bullet$}
     \put(0,-0.2){$\bullet$}
     \put(0.5,-0.2){$\bullet$}
     \end{picture}}

\newsavebox{\boxdreipartwww}
   \savebox{\boxdreipartwww}
   { \begin{picture}(1.2,1)
     \put(0.2,0.2){\line(0,1){0.4}}
     \put(0.6,0.2){\line(0,1){0.4}}
     \put(1,0.2){\line(0,1){0.4}}
     \put(0.2,0.6){\line(1,0){0.8}}
     \put(0,-0.2){$\circ$}
     \put(0.4,-0.2){$\circ$}
     \put(0.8,-0.2){$\circ$}
     \end{picture}}

\newsavebox{\boxsechspartwbwbwb}
   \savebox{\boxsechspartwbwbwb}
   { \begin{picture}(2.5,1)
     \put(0.2,0.2){\line(0,1){0.4}}
     \put(0.6,0.2){\line(0,1){0.4}}
     \put(1,0.2){\line(0,1){0.4}}
     \put(1.4,0.2){\line(0,1){0.4}}
     \put(1.8,0.2){\line(0,1){0.4}}
     \put(2.2,0.2){\line(0,1){0.4}}
     \put(0.2,0.6){\line(1,0){2}}
     \put(0,-0.2){$\circ$}
     \put(0.4,-0.2){$\bullet$}
     \put(0.8,-0.2){$\circ$}
     \put(1.2,-0.2){$\bullet$}
     \put(1.6,-0.2){$\circ$}
     \put(2.0,-0.2){$\bullet$}
     \end{picture}}
 
\newsavebox{\boxcrosspartwbbw}
   \savebox{\boxcrosspartwbbw}
   { \begin{picture}(1,1.4)
     \put(0.25,0.15){\line(1,2){0.4}}
     \put(0.65,0.15){\line(-1,2){0.4}}
     \put(0,0.9){$\circ$}
     \put(0.5,0.9){$\bullet$}
     \put(0,-0.2){$\bullet$}
     \put(0.5,-0.2){$\circ$}
     \end{picture}}
\newsavebox{\boxcrosspartbwwb}
   \savebox{\boxcrosspartbwwb}
   { \begin{picture}(1,1.4)
     \put(0.25,0.15){\line(1,2){0.4}}
     \put(0.65,0.15){\line(-1,2){0.4}}
     \put(0,0.9){$\bullet$}
     \put(0.5,0.9){$\circ$}
     \put(0,-0.2){$\circ$}
     \put(0.5,-0.2){$\bullet$}
     \end{picture}}
\newsavebox{\boxcrosspartwwww}
   \savebox{\boxcrosspartwwww}
   { \begin{picture}(1,1.4)
     \put(0.25,0.15){\line(1,2){0.4}}
     \put(0.65,0.15){\line(-1,2){0.4}}
     \put(0,0.9){$\circ$}
     \put(0.5,0.9){$\circ$}
     \put(0,-0.2){$\circ$}
     \put(0.5,-0.2){$\circ$}
     \end{picture}}
\newsavebox{\boxcrosspartbbbb}
   \savebox{\boxcrosspartbbbb}
   { \begin{picture}(1,1.4)
     \put(0.25,0.15){\line(1,2){0.4}}
     \put(0.65,0.15){\line(-1,2){0.4}}
     \put(0,0.9){$\bullet$}
     \put(0.5,0.9){$\bullet$}
     \put(0,-0.2){$\bullet$}
     \put(0.5,-0.2){$\bullet$}
     \end{picture}}

\newsavebox{\boxhalflibpartwwwwww}
   \savebox{\boxhalflibpartwwwwww}
   { \begin{picture}(1.5,1.4)
     \put(0.3,0.15){\line(1,1){0.8}}
     \put(1.1,0.15){\line(-1,1){0.8}}
     \put(0.7,0.15){\line(0,1){0.8}}     
     \put(0,0.9){$\circ$}
     \put(0.5,0.9){$\circ$}
     \put(1,0.9){$\circ$}
     \put(0,-0.2){$\circ$}
     \put(0.5,-0.2){$\circ$}
     \put(1,-0.2){$\circ$}     
     \end{picture}}

\newsavebox{\boxpositionerd} 
   \savebox{\boxpositionerd}
   { \begin{picture}(2.6,1.5)
     \put(0.9,0.2){\line(0,1){0.9}}
     \put(2.3,0.2){\line(0,1){0.9}}
     \put(0.9,1.1){\line(1,0){1.4}}
     \put(0.05,-0.05){$^\uparrow$}
     \put(0.2,0.4){\tiny$^{\otimes d}$}
     \put(1.35,-0.05){$^\uparrow$}
     \put(1.5,0.4){\tiny$^{\otimes d}$}
     \put(0,-0.2){$\circ$}
     \put(0.7,-0.2){$\circ$}
     \put(1.3,-0.2){$\bullet$}
     \put(2.1,-0.2){$\bullet$}
     \end{picture}}
\newsavebox{\boxpositionerrpluseins} 
   \savebox{\boxpositionerrpluseins}
   { \begin{picture}(3.9,1.5)
     \put(1.6,0.2){\line(0,1){0.9}}
     \put(3.6,0.2){\line(0,1){0.9}}
     \put(1.6,1.1){\line(1,0){2}}
     \put(0.05,-0.05){$^\uparrow$}
     \put(0.2,0.4){\tiny$^{\otimes r+1}$}
     \put(2.05,-0.05){$^\uparrow$}
     \put(2.2,0.4){\tiny$^{\otimes r-1}$}
     \put(0,-0.2){$\circ$}
     \put(1.4,-0.2){$\bullet$}
     \put(2,-0.2){$\bullet$}
     \put(3.4,-0.2){$\bullet$}
     \end{picture}}
\newsavebox{\boxpositionerdt} 
   \savebox{\boxpositionerdt}
   { \begin{picture}(2.9,1.5)
     \put(1.2,0.2){\line(0,1){0.9}}
     \put(2.9,0.2){\line(0,1){0.9}}
     \put(1.2,1.1){\line(1,0){1.7}}
     \put(0.05,-0.05){$^\uparrow$}
     \put(0.2,0.4){\tiny$^{\otimes dt}$}
     \put(1.65,-0.05){$^\uparrow$}
     \put(1.8,0.4){\tiny$^{\otimes dt}$}
     \put(0,-0.2){$\circ$}
     \put(1,-0.2){$\circ$}
     \put(1.6,-0.2){$\bullet$}
     \put(2.7,-0.2){$\bullet$}
     \end{picture}}
\newsavebox{\boxpositioners} 
   \savebox{\boxpositioners}
   { \begin{picture}(2.6,1.5)
     \put(0.9,0.2){\line(0,1){0.9}}
     \put(2.3,0.2){\line(0,1){0.9}}
     \put(0.9,1.1){\line(1,0){1.4}}
     \put(0.05,-0.05){$^\uparrow$}
     \put(0.2,0.4){\tiny$^{\otimes s}$}
     \put(1.35,-0.05){$^\uparrow$}
     \put(1.5,0.4){\tiny$^{\otimes s}$}
     \put(0,-0.2){$\circ$}
     \put(0.7,-0.2){$\circ$}
     \put(1.3,-0.2){$\bullet$}
     \put(2.1,-0.2){$\bullet$}
     \end{picture}}
\newsavebox{\boxpositionerdinv} 
   \savebox{\boxpositionerdinv}
   { \begin{picture}(2.6,1.5)
     \put(0.9,0.2){\line(0,1){0.9}}
     \put(2.3,0.2){\line(0,1){0.9}}
     \put(0.9,1.1){\line(1,0){1.4}}
     \put(0.05,-0.05){$^\uparrow$}
     \put(0.2,0.4){\tiny$^{\otimes d}$}
     \put(1.35,-0.05){$^\uparrow$}
     \put(1.5,0.4){\tiny$^{\otimes d}$}
     \put(0,-0.2){$\circ$}
     \put(0.7,-0.2){$\bullet$}
     \put(1.3,-0.2){$\bullet$}
     \put(2.1,-0.2){$\circ$}
     \end{picture}}
\newsavebox{\boxpositionersinv} 
   \savebox{\boxpositionersinv}
   { \begin{picture}(2.6,1.5)
     \put(0.9,0.2){\line(0,1){0.9}}
     \put(2.3,0.2){\line(0,1){0.9}}
     \put(0.9,1.1){\line(1,0){1.4}}
     \put(0.05,-0.05){$^\uparrow$}
     \put(0.2,0.4){\tiny$^{\otimes s}$}
     \put(1.35,-0.05){$^\uparrow$}
     \put(1.5,0.4){\tiny$^{\otimes s}$}
     \put(0,-0.2){$\circ$}
     \put(0.7,-0.2){$\bullet$}
     \put(1.3,-0.2){$\bullet$}
     \put(2.1,-0.2){$\circ$}
     \end{picture}}
\newsavebox{\boxpositionerdpluszwei}
   \savebox{\boxpositionerdpluszwei}
   { \begin{picture}(3.4,1.5)
     \put(1.6,0.2){\line(0,1){0.9}}
     \put(3.0,0.2){\line(0,1){0.9}}
     \put(1.6,1.1){\line(1,0){1.4}}
     \put(0.05,-0.05){$^\uparrow$}
     \put(0.2,0.4){\tiny$^{\otimes d+2}$}
     \put(2.05,-0.05){$^\uparrow$}
     \put(2.2,0.4){\tiny$^{\otimes d}$}
     \put(0,-0.2){$\circ$}
     \put(1.4,-0.2){$\bullet$}
     \put(2,-0.2){$\bullet$}
     \put(2.8,-0.2){$\bullet$}
     \end{picture}}
\newsavebox{\boxpositionersminuszwei}
   \savebox{\boxpositionersminuszwei}
   { \begin{picture}(3.4,1.5)
     \put(1,0.2){\line(0,1){0.9}}
     \put(3.0,0.2){\line(0,1){0.9}}
     \put(1,1.1){\line(1,0){2}}
     \put(0.05,-0.05){$^\uparrow$}
     \put(0.2,0.4){\tiny$^{\otimes s}$}
     \put(1.45,-0.05){$^\uparrow$}
     \put(1.6,0.4){\tiny$^{\otimes s-2}$}
     \put(0,-0.2){$\circ$}
     \put(0.8,-0.2){$\bullet$}
     \put(1.4,-0.2){$\bullet$}
     \put(2.8,-0.2){$\bullet$}
     \end{picture}}
\newsavebox{\boxpositionerrnull}
   \savebox{\boxpositionerrnull}
   { \begin{picture}(3.8,1.5)
     \put(1.2,0.2){\line(0,1){0.9}}
     \put(3.4,0.2){\line(0,1){0.9}}
     \put(1.2,1.1){\line(1,0){2.2}}
     \put(0.05,-0.05){$^\uparrow$}
     \put(0.2,0.4){\tiny$^{\otimes r_0}$}
     \put(1.65,-0.05){$^\uparrow$}
     \put(1.8,0.4){\tiny$^{\otimes r_0-2}$}
     \put(0,-0.2){$\circ$}
     \put(1,-0.2){$\bullet$}
     \put(1.6,-0.2){$\bullet$}
     \put(3.2,-0.2){$\bullet$}
     \end{picture}}
\newsavebox{\boxpositionerwbwb}
   \savebox{\boxpositionerwbwb}
   { \begin{picture}(1.8,1)
     \put(0.6,0.2){\line(0,1){0.6}}
     \put(1.4,0.2){\line(0,1){0.6}}
     \put(0.6,0.8){\line(1,0){0.8}}
     \put(0.05,-0.05){$^\uparrow$}
     \put(0.85,-0.05){$^\uparrow$}
     \put(0,-0.2){$\circ$}
     \put(0.4,-0.2){$\bullet$}
     \put(0.8,-0.2){$\circ$}
     \put(1.2,-0.2){$\bullet$}
     \end{picture}}
\newsavebox{\boxpositionerwwbb}
   \savebox{\boxpositionerwwbb}
   { \begin{picture}(1.8,1)
     \put(0.6,0.2){\line(0,1){0.6}}
     \put(1.4,0.2){\line(0,1){0.6}}
     \put(0.6,0.8){\line(1,0){0.8}}
     \put(0.05,-0.05){$^\uparrow$}
     \put(0.85,-0.05){$^\uparrow$}
     \put(0,-0.2){$\circ$}
     \put(0.4,-0.2){$\circ$}
     \put(0.8,-0.2){$\bullet$}
     \put(1.2,-0.2){$\bullet$}
     \end{picture}}
\newsavebox{\boxpositionerrevwbwb}
   \savebox{\boxpositionerrevwbwb}
   { \begin{picture}(1.8,1)
     \put(0.2,0.2){\line(0,1){0.6}}
     \put(1.0,0.2){\line(0,1){0.6}}
     \put(0.2,0.8){\line(1,0){0.8}}
     \put(0.45,-0.05){$^\uparrow$}
     \put(1.25,-0.05){$^\uparrow$}
     \put(0,-0.2){$\circ$}
     \put(0.4,-0.2){$\bullet$}
     \put(0.8,-0.2){$\circ$}
     \put(1.2,-0.2){$\bullet$}
     \end{picture}}
\newsavebox{\boxpositionersalphaw} 
   \savebox{\boxpositionersalphaw}
   { \begin{picture}(2.6,1.5)
     \put(0.9,0.2){\line(0,1){0.9}}
     \put(2.3,0.2){\line(0,1){0.9}}
     \put(0.9,1.1){\line(1,0){1.4}}
     \put(0.05,-0.05){$^\uparrow$}
     \put(1.35,-0.05){$^\uparrow$}
     \put(0.2,0.4){\tiny$^{\otimes s}$}
     \put(1.5,0.4){\tiny$^{\otimes \alpha}$}
     \put(0,-0.2){$\circ$}
     \put(0.7,-0.2){$\circ$}
     \put(1.3,-0.2){$\circ$}
     \put(2.1,-0.2){$\bullet$}
     \end{picture}}
\newsavebox{\boxpositionersalphab} 
   \savebox{\boxpositionersalphab}
   { \begin{picture}(2.6,1.5)
     \put(0.9,0.2){\line(0,1){0.9}}
     \put(2.3,0.2){\line(0,1){0.9}}
     \put(0.9,1.1){\line(1,0){1.4}}
     \put(0.05,-0.05){$^\uparrow$}
     \put(1.35,-0.05){$^\uparrow$}
     \put(0.2,0.4){\tiny$^{\otimes s}$}
     \put(1.5,0.4){\tiny$^{\otimes \alpha}$}
     \put(0,-0.2){$\circ$}
     \put(0.7,-0.2){$\circ$}
     \put(1.3,-0.2){$\bullet$}
     \put(2.1,-0.2){$\bullet$}
     \end{picture}}
\newsavebox{\boxpositioner}
   \savebox{\boxpositioner}
   { \begin{picture}(1.4,1)
     \put(0.4,0){\line(0,1){0.6}}
     \put(1.2,0){\line(0,1){0.6}}
     \put(0.4,0.6){\line(1,0){0.8}}
     \put(0,0){\line(0,1){0.3}}
     \put(0.8,0){\line(0,1){0.3}}
     \end{picture}}
\newsavebox{\boxprimarypart} 
   \savebox{\boxprimarypart}
   { \begin{picture}(1,1.3)
     \put(0,-0.2){\line(0,1){0.3}}
     \put(0.4,-0.2){\line(0,1){0.3}}
     \put(0.8,-0.2){\line(0,1){0.3}}
     \put(0.4,0.1){\line(1,0){0.4}}
     \put(0,0.9){\line(0,1){0.3}}
     \put(0.4,0.9){\line(0,1){0.3}}
     \put(0.8,0.9){\line(0,1){0.3}}
     \put(0,0.9){\line(1,0){0.4}}
     \put(0,0.1){\line(1,1){0.8}}
     \put(0.2,0.9){\line(1,-2){0.4}}
     \end{picture}}

\DeclareMathOperator{\Fix}{Fix}

\begin{document}
\maketitle

\begin{abstract}
By a construction of Vaughan Jones, the bipartite graph $\Gamma(A)$ associated with the natural inclusion of $\C$ inside a finite-dimensional $C^*$-algebra $A$ gives rise to a planar algebra $\cP^{\Gamma(A)}$.
We prove that every subfactor planar subalgebra of $\cP^{\Gamma(A)}$ is the fixed point planar algebra of a uniquely determined action of a compact quantum group $\G$ on $A$. We use this result to introduce a conceptual framework for the free wreath product operation on compact quantum groups in the language of planar algebras/standard invariants of subfactors. Our approach will unify both previous definitions of the free wreath product due to Bichon and Fima-Pittau and extend them to a considerably larger class of compact quantum groups. In addition, we observe that the central Haagerup property for discrete quantum groups is stable under the free wreath product operation (on their duals).

\end{abstract}

\section{Introduction}
Soon after the study of compact quantum groups had been initiated in \cite{Wo87-1}, \cite{Wo87-2}, Woronowicz proved a powerful strengthening of the Tannaka-Krein reconstruction theorem \cite{Wo88} which has since then been an indispensable tool of compact quantum group theory.
The theorem asserts that, given a concrete rigid $C^*$-tensor category $\cC$ (see the book \cite{NT13} for a detailed definition), one can construct a compact quantum group $\G$ whose category of finite-dimensional unitary representations is precisely $\cC$. In other words, one describes the unitary representation theory one would like the compact quantum group to have, rather than the quantum group itself. Woronowicz's theorem has been used to study compact quantum groups from a categorical viewpoint and, perhaps even more frequently, to construct new examples from categorical data. Well-investigated results of such Tannaka-Krein constructions are Banica and Speicher's easy quantum groups \cite{BS09} (see also \cite{TW15}) whose representation categories are of a combinatorial nature.\\ 
Another way to produce new examples of compact quantum groups is to build them out of existing ones by performing binary operations such as the free product operation \cite{Wa95} and the free wreath product operation. The latter is a highly non-commutative version of the classical wreath product operation for groups and was introduced by Bichon in \cite{Bi04}. It will be the main focus of this paper. While Bichon's definition was given directly on the $C^*$-algebraic level, it has become apparent in the recent works \cite{LT16} and \cite{FP16} (see also \cite{FS15}) that the construction is at its core a categorical one. We emphasize this observation further by relating the free wreath product to another source of rigid $C^*$-tensor categories, namely Jones's theory of subfactors and their standard invariants. \\

The theory of subfactors emerged out of Jones's groundbreaking work \cite{J83} in which he introduced a notion of index $[M:N]$ as an invariant for unital inclusions of II$_1$-factors $N \subset M$ and classified the values that this index can take. Every subfactor inclusion $N \subset M$ of finite Jones index $[M:N] < \infty $ can be extended to a tower $ N= M_{-1} \subset M = M_0 \subset M_1 \subset M_2 \subset \dots$ of II$_1$-factors, satisfying $[M_{i+1}: M_i] = [M:N]$ for all $i \geq 0$ and having the property that every $M_{i+1}$ is generated by $M_i$ and a projection $e_i$ which commutes with $M_{i-1}$ (see \cite{J83} for details). The relative commutants $M_i^{'} \cap M_j$, with $i \leq j$, form a lattice of finite dimensional $C^*$-algebras called the standard invariant. The standard invariant has taken center stage in many aspects of subfactor theory such as the classification of subfactors of small index. In his seminal work \cite{Po95}, Popa found an abstract characterization of standard invariants of subfactor inclusions as $\lambda$-lattices, that is to say, he proved that any $\lambda$-lattice can be realized as the standard invariant of a subfactor $N \subset M$ with index $[M:N] = \lambda^{-1}$. A more diagrammatical axiomatization called subfactor planar algebras was given by Jones in \cite{J99}. In \cite{J98}, Jones constructed a class of (non-subfactor) planar algebras out of a bipartite graph and a function on the set of vertices of the graph. Such a bipartite graph can for instance be derived from any inclusion of finite-dimensional $C^*$-algebras. The class of bipartite graph planar algebras gained attention within the subfactor community due to its universality which has been highlighted in \cite{JP10} and \cite{Bu10}. \\ 

There are several articles of Banica, most prominently \cite{B01} and \cite{B05}, in which the relation between compact quantum groups and (standard invariants of) subfactors is explored. In \cite{B05}, Banica associates with every faithful action $\al$ of a compact quantum group $\G$ on a finite-dimensional $C^*$-algebra $A$ that is ergodic on the center of $A$ and that preserves the Markov trace $\tr$ of the inclusion $\C \subset A$, a subfactor planar subalgebra $\cP(\al)$ of the planar algebra $\cP^{\Gamma(A)}$. Here, $\cP^{\Gamma(A)}$ denotes the planar algebra associated with the bipartite graph $\Gamma(A)$ of the inclusion $\C \subset A$ (see Section \ref{bipartitepa}). Generalizing a result in \cite{B05-2}, we prove that the converse of Banica's result is also true. 

\begin{thmx}[Theorem \ref{main1}] \label{mainA}
Let $A$ be a finite-dimensional $C^*$-algebra.
Every subfactor planar subalgebra $\cQ$ of $\cP^{\Gamma(A)}$ is of the form $\cQ = \cP(\al)$ for a uniquely determined pair $(\G, \al)$ of a (universal) compact quantum group $\G$ and a faithful, centrally ergodic, $\tr$-preserving action $\al$ of $\G$ on $A$ (up to conjugacy).
\end{thmx}
 In view of \cite{Bu10}, Theorem \ref{mainA} is telling us that every subfactor planar subalgebra of $\cP^{\Gamma(A)}$ is the fixed point algebra of a quantum action on the Jones tower of the inclusion $\C \subset A$. From a quantum group perspective, the theorem should be interpreted as a version of Woronowicz's Tannaka-Krein duality in the setting of planar algebras and it should be compared to \cite[Theorem B]{B99-2}. \\

 The correspondence between (actions of) compact quantum groups and subfactor planar algebras established in Theorem \ref{mainA} yields a new interpretation of the free wreath product construction. Building the free wreath product $\G \wr_* \F$ of two compact quantum groups $\F, \G$ that act appropriately on finite-dimensional $C^*$-algebras, corresponds to building the free product of the associated subfactor planar algebras in the sense of Bisch and Jones \cite{BJ95}.

\begin{thmx} \label{mainB}
Let $\al$ be an action of a compact quantum group $\F$ on $C(X)$ where $X$ is a set of $n$ points and let $\beta$ be an action of a compact quantum group $\G$ on a finite-dimensional $C^*$-algebra $B$. Assume that both actions are faithful, centrally ergodic and preserve the Markov trace. Then the free wreath action $\beta \wr_* \al$ of $\G \wr_* \F$ on $C(X) \ot B$ (see \ref{wreathaction}) enjoys the same properties and
\[ \cP(\beta \wr_* \al) \ = \ \cP(\al) * \cP(\beta) \ \subset \ \cP^{\Gamma(C(X) \ot B)}, \]
where $*$ denotes the free product operation on planar algebras in the sense of Bisch and Jones.
\end{thmx}

We also prove an analogue of Theorem \ref{mainB} for the partial generalization of Bichon's free wreath product due to Fima and Pittau \cite{FP16}, see Theorem \ref{main2}. In particular, our approach allows us to unify both previous definitions and generalize them considerably (Definition \ref{wreathnewdef} and Remark \ref{defirem}). More precisely, we define a free wreath product $\G \wr_* (\F, \al)$ of a compact quantum group $\G$ with a pair $(\F, \al)$ consisting of a compact quantum group $\F$ and a faithful, centrally ergodic, Markov trace-preserving action $\al$ of $\F$ on a finite-dimensional $C^*$-algebra. One should note that $\F$ admits such an action whenever it is of Kac type and its category of finite-dimensional unitary representations is finitely generated (see Example \ref{exampleaction}).

Note that if $\chi_{\alpha}$ denotes the character of the action $\alpha$ of a compact quantum group $\G$, then $\chi_{\alpha}$ can be seen as a random variable in the non-commutative probability space $(C(\G),h_{\G})$, where $h_{\G}$ is the Haar state on $C(\G)$. Let us consider $\alpha$ and $\beta$ two faithful, centrally ergodic actions respectively on $C(X)$ and $C(Y)$, where $X$ and $Y$ are finite sets. Investigating connections between lexicographical products of graphs and free wreath products of quantum groups, Banica and Bichon have conjectured in \cite[Conjecture 3.4]{BB07} that the random variable $\chi_{\beta\wr_{*}\alpha}$  is distributed as the free multiplicative convolution of the random variables $\chi_\alpha$ and $\chi_\beta$. An application of Theorem \ref{mainB} yields a positive answer to this conjecture.

\begin{corx}\label{corD}
Let $(\F,\alpha),(\G,\beta)$ be two pairs of compact quantum groups and faithful, centrally ergodic actions acting respectively on $C(X)$ and $C(Y)$, where $X$ and $Y$ are finite sets. Then,
$$\chi_{\beta\wr_{*}\alpha}=\chi_{\alpha}\boxtimes\chi_{\beta},$$
where the equality holds in the distributional sense.
\end{corx}

The study of central approximation/rigidity properties of discrete quantum groups was initiated in \cite{dCFY14}. In \cite{PV15}, Popa and Vaes gave a systematic account of these properties in the setting of subfactor inclusions and in the setting of rigid $C^*$-tensor categories (see also \cite{NY15}, \cite{GJ16}). By combining the results of Popa-Vaes with our main theorems, we obtain the following corollary.

\begin{corx} \label{corC}
Let $(\F,\al), (\G, \beta)$ be pairs of compact quantum groups and faithful, centrally ergodic, Markov trace-preserving actions on finite-dimensional $C^*$-algebras. Consider their free wreath product $(\G \wr_* \F, \beta \wr_* \al)$. If the discrete duals of $\F$ and $\G$ have the central Haagerup property (the central ACPAP), the discrete dual of $\G \wr_* \F$ has the central Haagerup property (the central ACPAP) as well.
\end{corx}
 
The article is organized as follows. Sections \ref{prequantum} and \ref{preplanar} will serve as a summary of basic facts on compact quantum groups and planar algebras which we will use throughout the paper. Section \ref{reconstruction} will contain the proof of Theorem \ref{mainA}. In Sections \ref{partitionsection} and \ref{freeproductsection} we will relate planar tangles to well-known tools from free probability theory which we will use to prove Theorem \ref{mainB} and Corollary \ref{corD} in Section \ref{ProductSection}. In the following paragraph \ref{approximation}, we will discuss Corollary \ref{corC} in detail. The last section contains a description of a basis of a free product of subfactor planar algebras in terms of Boolean vector space cumulants. The methods used in Section \ref{basissection} stem from free probability theory and are related to techniques developed in the article \cite{GJS10}. They rely on a rather surprising result on Boolean cumulants found in \cite{BN08}. 

\section*{Acknowledgements}
PT is grateful to Roland Vergnioux for interesting discussions on the subject. A part of the results in the present manuscript has been obtained when PT was doing his PhD, and he would also like to thank Roland Speicher and Philippe Biane for their supervision during this period.
JW is deeply grateful to Yuki Arano, Tim de Laat and Moritz Weber for valuable discussions/suggestions and to his PhD advisor Stefaan Vaes for his patient guidance during the development of this article. PT would like to thank Stefaan Vaes for his suggestions as well.

\section{Preliminaries on compact quantum groups} \label{prequantum}

We will denote the underlying unital $C^*$-algebra of a compact quantum group $\G$ by $C(\G)$, its comultiplication by $\Delta$ and its canonical dense Hopf-$*$-algebra by $\Pol(\G)$. The counit will be denoted by $\varepsilon: \Pol(\G) \to \C$ and the category of finite-dimensional unitary representations by $\Rep(\G)$.

\subsection{Actions on finite dimensional $C^*$-algebras} \label{actions}
A vector space action of a compact quantum group $\G$ on a finite-dimensional vector space $V$ is a linear map $\al: V \to V \ot \Pol(\G)$ which is coassociative, i.e\[ (\al \ot \id) \circ \al = (\id \ot \Delta) \circ \al \]
and counital, i.e. $(\id \ot \varepsilon) \circ \al = \al$. A representation of $\G$ on $V$ is an element $u \in \mathcal L(V) \ot \Pol(\G)$ such that $(\id \ot \Delta) u = u_{12}u_{13}$ and $(\id \ot \varepsilon)u = 1$. The space of intertwiners between two representations $u$ and $v$ will be denoted by $\Mor(u,v)$. \\
Recall that the identification $\mathcal L(V, V \ot \Pol(\G)) \cong \mathcal L(V) \ot \Pol(\G)$ establishes a one-to-one correspondence $\al \mapsto u_{\al}$ (with inverse $u \mapsto \al_u$) between vector space actions of $\G$ on $V$ and representations of $\G$ on $V$. Let $A$ be a finite-dimensional $C^*$-algebra with multiplication map $m: A \ot A \to A$ and unit map $\eta: \C \to A$. An action $\al$ of $\G$ on $A$ will be a vector space action on $A$ which is a unital $*$-homomorphism. If $\varphi: A \to \C$ is a positive functional on $A$, we will say that $\al$ is $\varphi$-preserving if
\[ (\varphi \ot \id) \circ \al = \varphi(\cdot) \1_{\Pol(\G)}. \]

\begin{lem}[\cite{B99}] \label{actionlemma}
Let $\al: A \to A \ot \Pol(\G)$ be a vector space action on $A$ and let $\tr$ be a positive tracial functional on $A$. Then $\al$ is
\begin{enumerate}
\item multiplicative if and only if $m \in \Mor(u_{\al} \ot u_{\al}, u_{\al})$,
\item unital if and only if $\eta \in \Mor(\1, u_{\al})$,
\item $\varphi$-preserving if and only if $\varphi \in \Mor(u_{\al}, \1)$. \\
If these conditions are satisfied, 
\item $\al$ is involutive if and only if $u_{\al}$ is unitary (w.r.t. the inner product on $A$ given by $\langle x,y \rangle = \tr(y^* x), \ x,y \in A$).
\end{enumerate}
\end{lem}

\begin{rem}
In \cite{Wa98}, it was shown that, given a faithful state $\varphi$ on the finite-dimensional $C^*$-algebra $A$, there exists a universal compact quantum group $\G = \G_{aut}(A,\varphi)$ acting $\varphi$-preservingly on $A$. Turn $A$ into a Hilbert space $\cH$ by defining an inner product on $A$ through $\langle x,y \rangle = \varphi(y^*x)$ for all $x,y \in A$. In view of Lemma \ref{actionlemma}, $C(\G)$ can be described as the universal $C^*$-algebra generated by the coefficients of $u \in B(\cH) \ot C(\G)$ subject to the relations which make $ u$ unitary, $m \in \Mor(u \ot u, u),$ and $ \eta \in \Mor(\1, u)$.
\end{rem}

A fixed point of an action $\al: A \to A \ot \Pol(\G)$ is an element $x \in A$ such that $\al(x) = x \ot \1$. We will refer to the set of fixed points of $\al$ as $\Fix(\al)$. An action $\al$ is \emph{ergodic} if $\Fix(\al) = \C 1_A$ and \emph{centrally ergodic} if $Z(A) \cap \Fix(\al) = \C 1_A$. Moreover, it is \emph{faithful} if the coefficients of $u_{\al}$ generate $C(\G)$. \\
Recall that, given an inclusion $A_0 \subset A_1$ of finite-dimensional $C^*$-algebras such that $Z(A_0) \cap Z(A_1) = \C 1_{A_1}$, there exists a 'particularly good' choice of trace $\tr$ on $A_1$ which is commonly referred to as the \emph{Markov trace} of the inclusion. This trace extends to the basic construction $ A_2 = \langle A_1, e_1 \rangle $ with Jones projection $e_1$ and is uniquely determined under mild irreducibility assumptions. For a thorough discussion of inclusions of finite-dimensional $C^*$-algebras and their traces, we recommend the book \cite{GdlHJ89}. The following observation due to Banica will be fundamental throughout this article.

\begin{prop}[\cite{B01}] \label{ban:tow}
Let $A_0 \subset A_1$ be an inclusion of finite-dimensional $C^*$-algebras such that $Z(A_0) \cap Z(A_1) = \C 1_{A_1}$ with Markov trace $\tr$. Let
\[ A_0 \subset A_1 \subset A_2 = \langle A_1, e_1 \rangle  \subset \cdots  \]
be its Jones tower. If $\al: A_1 \to A_1 \ot \Pol(\G)$ is an action leaving $A_0$ and $\tr$ invariant, there is a unique sequence $(\al_i)_{i \geq 0}$ of actions $\al_i: A_i \to A_i \ot \Pol(\G)$ with $\al_1 = \al$ such that each $\al_i$ extends $\al_{i-1}$ and leaves the Jones projection $e_{i-1}$ invariant. 
\end{prop}

The following example shows that every compact quantum group of Kac type with finitely generated representation category admits a faithful, centrally ergodic action on some finite-dimensional $C^*$-algebra $A$ preserving the Markov trace of $\C \subset A$.

\begin{ex} \label{exampleaction}
Let $\G$ be a compact quantum group whose Haar state is a trace and let $u \in B(\cH_u) \ot \Pol(\G)$ be a finite-dimensional unitary representation. Consider the conjugation action
\[ \al_c(u): B(\cH_u) \to B(\cH_u) \ot \Pol(\G), \quad x \mapsto u (x \ot \1)u^*.  \]
The factoriality of $B(H_u)$ implies that $\al_c(u)$ is centrally ergodic and it is clear that $\al_c(u)$ preserves the unique normalized trace on $B(\cH_u)$. In addition, if $u$ contains the trivial representation as a subobject and if its coefficients generate $\Pol(\G)$, the action $\al_c(u)$ is faithful since every coefficient of $u$ (and its adjoint) appears as a coefficient of $\al_c(u)$.
\end{ex}

\subsection{Free wreath product construction} \label{freewreath}

The free wreath product of a compact quantum group $\G$ with a quantum subgroup $\F$ of Wang's quantum permutation group $S_n^+$ (see \cite{Wa98}) is a non-commutative analogue of the classical wreath product of groups and was introduced by Bichon in \cite{Bi04}. One should note that $S_n^+$ is the universal quantum group acting on the pair $(C(X), \varphi)$, where $X$ is a set consisting of $n$ points and $\varphi$ is the state on $C(X)$ obtained by integrating with respect to the uniform probability measure on $X$. In other words, the quantum subgroups of $S_n^+$ are exactly those compact quantum groups which act faithfully on $n$ points.

\begin{defi}[\cite{Bi04}] \label{DefBichon}
Let $\G$ be a compact matrix quantum group with fundamental representation $v =(v_{kl})_{1 \leq k,l \leq m}$ and let $\F$ be a quantum subgroup of $S_n^+$ with fundamental representation $u = (u_{ij})_{1 \leq i,j \leq n}$. Define the unital $C^*$-algebra $C(\G \wr_* \F)$ as the quotient of 
\[ C(\G)^{*n}* C(\F) \]
by the relations that make the $i$-th copy of $C(\G)$ inside the free product commute with the $i$-th row of $u$. Set $w = (v_{kl}^i u_{ij})_{(i,k)(j,l)}$ where $v^i_{kl}$ denotes the image of $v_{kl}$ inside the $i$-th copy of $C(\G)$. The compact matrix quantum group $\G \wr_* \F =(C(\G \wr_* \F),w) $ is called the \emph{free wreath product} of $\G = (C(\G),v)$ and $\F =(C(\F),u)$.
\end{defi} 

When $\F = S_n^+$, the representation theory of $\G \wr_* S_n^+$ has been studied in \cite{LT16}. Moreover, Bichon's construction was partially generalized in \cite{FP16} to the situation in which the right input $\F$ is replaced by the universal compact quantum group $\G_{aut}(A, \tr)$ acting on a finite-dimensional $C^*$-algebra $A$ in a Markov trace-preserving way. In fact, the definition in \cite{FP16} is slightly more general, see \cite{FP16} for details. \\
Choose for any $x \in \Irr(\G)$ a representative $u^{x} \in B(\cH_{x}) \ot \Pol(\G)$. Set $\cH = L^2(A,\tr)$.

\begin{defi}[\cite{FP16}] \label{DefFima}
Let $\G$ be a compact quantum group. Define $C(\G \wr_* \G_{aut}(A, \tr))$ to be the universal unital $C^*$-algebra generated by the coefficients of $a(x) \in B(\cH \ot \cH_{x}) \ot \Pol(\G\wr_* \G_{aut}(A, \tr)), \ x \in \Irr(\G)$, satisfying the relations:
\begin{itemize}
\item $a(x)$ is unitary for any $x \in \Irr(\G)$,
\item For all $x,y,z \in \Irr(\G)$ and all $S \in \Mor(x \ot y, z)$
\[ (m \ot S) \circ \Sigma_{23} \in \Mor(a(x) \ot a(y), a(z)), \]
where $\Sigma_{23}: \cH \ot \cH_x \ot \cH \ot \cH_y \to \cH \ot \cH \ot \cH_x \ot \cH_y$ is the tensor flip on the second and third leg.
\item $\eta \in \Mor(\1, a(\1_{\G}))$.
\end{itemize}
The compact quantum group $\G \wr_* \G_{aut}(A, \tr) = (C(\G \wr_* \G_{aut}(A, \tr)),\Delta)$ with comultiplication as defined in \cite[Proposition 2.7]{FP16} is called the \emph{free wreath product} of $\G$ and $\G_{aut}(A, \tr)$.
\end{defi}

We recall some facts on $\G \wr_* \G_{aut}(A, \tr)$ for later use. For any finite-dimensional unitary representation $u \in B(H_u) \ot \Pol(\G)$, one can define a unitary representation $a(u) \in B(\cH \ot H_u) \ot \Pol(\G \wr_* \G_{aut}(A, \tr))$ in the following way (see \cite[Section 8]{FP16}). For any $x \in \Irr(\G)$ appearing as a subobject of $u$, choose a family of isometries $S_{x,k} \in \Mor(u_x, u), \ 1 \leq k \leq \dim \Mor(u_x,u)$ such that $S_{x,k}S_{x,k}^*$ are pairwise orthogonal projections with $\sum_{x \subset u } \sum_{k} S_{x,k}S_{x,k}^* = \id_{\cH_u} $. We have
\[ u =  \sum_{x \subset u } \sum_{k} (S_{x,k} \ot \1) u_x (S_{x,k}^* \ot \1). \]
Set
\[ a(u) \ = \ \sum_{x \subset u } \sum_{k} (\id_{\cH} \ot S_{x,k} \ot \1) a(x) (\id_{\cH} \ot S_{x,k}^* \ot \1) \ \in \ B(\cH \ot \cH_u) \ot  \Pol(\G \wr_* \G_{aut}(A, \tr)). \]
Thanks to \cite[Proposition 8.1]{FP16}, $a(u)$ is a well defined finite-dimensional unitary representation.
Moreover, for all finite-dimensional unitary representations $u,v,w$ of $\G$ and for every $S \in \Mor(u \ot v, w)$ define
\[ a(S) : \cH \ot \cH_u \ot \cH \ot \cH_v \to \cH \ot \cH_w, \quad a(S) = (m \ot S) \Sigma_{23}. \]

\begin{prop} \label{wreathaction}
Let $\beta: B \to B \ot \Pol(\G)$ be a faithful, centrally ergodic $\tr_B$-preserving action of $\G$ on a finite-dimensional $C^*$-algebra $B$ with associated representation $u := u_{\beta}$. Denote the universal action of $\G_{aut}(A, \tr_A)$ on $A$ by $\al$. The linear map
\[ \beta \wr_* \al := \al_{a(u)} : A \ot B \to A \ot B \ot \Pol(\G \wr_* \G_{aut}(A, \tr)) \] 
defines a faithful, centrally ergodic $\tr_A \ot \tr_B$-preserving action of $\G \wr_* \G_{aut}(A, \tr)$ on $A \ot B$. It is called the free wreath action of $\beta$ and $\al$.
\end{prop}

\begin{proof}
For the multiplication map $m_{A\ot B}$ on $A \ot B$, we have $m_{A\ot B} = (m_A \ot m_B) \Sigma_{23} = a(m_B) \in Mor(a(u) \ot a(u), a(u))$ by \cite[Proposition 8.1]{FP16}. By Lemma \ref{actionlemma}, $\beta \wr_* \al$ is multiplicative. By definition, we have $\eta_A \in \Mor(\1, a(\1_{\G}))$ for the unit map of $A$ and it follows again from \cite[Proposition 8.1]{FP16} that $\id_A \ot \eta_B \in \Mor(a(\1_{\G}), a(u))$ and hence $\eta_{A \ot B} \in \Mor(\1, a(u))$ and by Lemma \ref{actionlemma}, $\beta \wr_* \al$ is unital. Since the universal action $\al$ on $A$ is centrally ergodic and $\tr_A$-preserving, we have $\tr_A = \eta_A^* \in \Mor(a(\1_{\G}), \1)$ and $\id_A \ot \tr_B \in \Mor(a(u), a(\1_{\G}))$ by \cite[Proposition 8.1]{FP16}. It follows that $\beta \wr_* \al$ preserves the trace and is involutive by Lemma \ref{actionlemma}. By definition, $a(u)$ decomposes as $a(u) = \sum_{x \subset u} a(x)$ and by central ergodicity of $\beta$, the summand $a(\1_{\G})$ appears exactly once in this decomposition. Due to the central ergodicity of $\al$, $a(\1_{\G})$ contains the trivial representation exactly once as well (see \cite[Proposition 5.1]{FP16}). Consequently, $\dim \Mor(\1, a(u)) =1$ and $\beta \wr_* \al$ is centrally ergodic. Lastly, the fact that $\beta \wr_* \al$ is faithful follows from point 5 in \cite[Proposition 8.1]{FP16}.  
\end{proof}

\begin{rem}
If we return to Bichon's setting, that is to say, we replace the input on the right by an arbitrary faithful, centrally ergodic action of $\F$ on a set $X$ of $n$ points preserving the uniform probability measure (i.e. $A = C(X)$), one can easily see that the vector space action $\al_w$ associated to the unitary representation $w$ in Definition \ref{DefBichon} is a faithful, trace-preserving action on $C(X) \ot A$. It is however not immediate that this action is centrally ergodic and we will deal with this in section \ref{ProductSection}.
\end{rem}

\section{Preliminaries on planar algebras} \label{preplanar}
Planar algebras were introduced by Jones in \cite{J99} in order to axiomatize the standard invariant of a subfactor in a diagrammatical way. The introduction to planar algebras in this section follows the PhD thesis of the first-named author.

\subsection{Planar tangles}

\begin{defi}[\cite{J99}]
A planar tangle $T$ of degree $k\geq 0$ consists of the following data.
\begin{itemize}
\item A disk $D_{0}$ of $\mathbb{R}^{2}$, called the outer disk.
\item Some disjoint disks $D_{1},\dots,D_{n}$ in the interior of $D_{0}$ which are called the inner disks.
\item For each $0\leq i\leq n$, a finite subset $S_{i}\in\partial D_{i}$ of cardinality $2k_{i}$ (such that $k_{0}=k$) with a distinguished element $i_{*}\in S_{i}$. The elements of $S_{i}$ are called the boundary points of $D_{i}$ and numbered clockwise starting from $i_{*}$.  $k_{i}$ is called the degree of the inner disk $D_{i}$. Whenever we draw pictures, we will distinguish $i_*$ by marking the boundary region of $D_i$ preceding $i_*$ by $*$.
\item A finite set of disjoint smooth curves $\lbrace\gamma_{j}\rbrace_{1\leq j\leq r}$  such that each $\mathring{\gamma}_{j}$ lies in the interior of $D_{0}\setminus \bigcup_{i\geq 1} D_{i}$ and such that $\bigcup_{1\leq j\leq r}\partial \gamma_{j}=\bigcup_{0\leq i\leq n} S_{i}$; it is also required that each curve meets a disk boundary orthogonally, and that its endpoints have opposite (resp. same) parity if they both belong to inner disks or both belong to the outer disk (resp. one belongs to an inner disk and the other one to the outer disk).
\item A region of $P$ is a connected component of $D_{0}\setminus (\bigcup_{i\geq 1} D_{i}\cup(\bigcup \gamma_{j}))$. Give a chessboard shading on the regions of $P$ in such a way that the interval components of type $(2i-1,2i)$ are boundaries of shaded regions.
\end{itemize}
The skeleton of $T$, denoted by $\Gamma T$, is the set $(\bigcup \partial D_{i})\cup(\bigcup \gamma_{j})$.
\end{defi}

Planar tangles will always be considered up to isotopy. \\
If the degree of $T$ is $0$, then $T$ is of degree $+$ (respectively $-$) if the boundary of the outer disk is the boundary of an unshaded (respectively shaded) region. An example of a planar tangle is given in Figure \ref{Fig1PlanarTangle}.
\begin{figure}[h!]
\begin{center}
\scalebox{0.7}{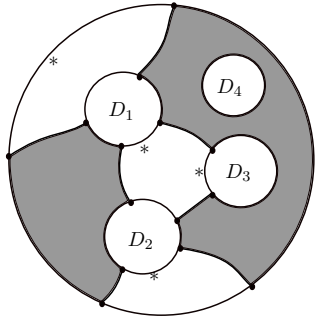}
\end{center}
\caption{\label{Fig1PlanarTangle}: Planar tangle of degree $4$ with $4$ inner disks.}
\end{figure}
\\
For clarity reasons, we will sometimes omit the numbering of the inner disks in our drawings.
A connected planar tangle is a planar tangle whose regions are simply connected; this implies that for any inner disk $D$ and any element $x\in \partial D$, there is a path from $x$ to the boundary of the outer disk which is contained in $(\bigcup \partial D_{i})\cup(\bigcup \gamma_{j})$. An irreducible planar tangle is a connected planar tangle such that each curve has an endpoint being a distinguished point of $D_{0}$ and the other one being on an inner disk. An example of a connected planar tangle (resp. irreducible planar tangle) is given in Figure \ref{Fig2ConnectIrrTangle}.
 \begin{figure}[h!]
\begin{center}
\scalebox{0.7}{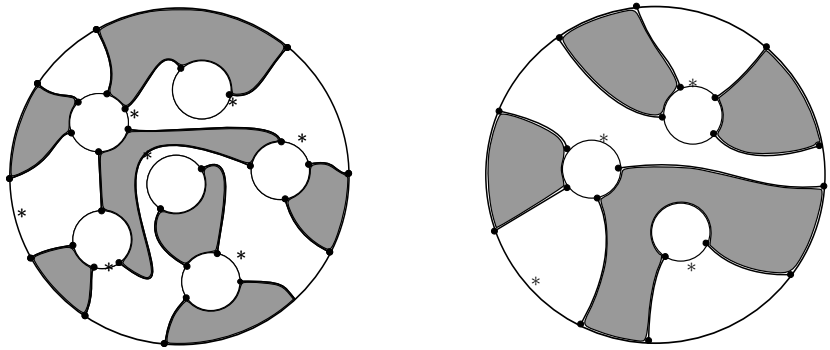}
\end{center}
\caption{\label{Fig2ConnectIrrTangle}: A connected and an irreducible planar tangle.}
\end{figure}
\subsubsection*{Composition and involution of planar tangles}
Let $T$ and $T'$ be two planar tangles of respective degree $k$ and $k'$, and let $D$ be an inner disk of $T$. We assume that the degree of $D$ is also $k'$. Then the tangle $T$ can be composed with the tangle $T'$ by isotoping $T'$ onto the inner disk $D$ in such a way that the boundary of $T'$ is mapped onto the boundary of $D$ and that the $i$-th marked boundary point of $T'$ is mapped onto the $i$-th marked boundary point of $D$. Hence, the strings adjacent to the boundary points of $T'$ connect with those adjacent to the boundary points of $D$ and by removing the boundary of $D$, we obtain a new planar tangle which is denoted by $T\circ_{D}T'$. It is called the composition of $T$ and $T'$ with respect to $D$. The composition of the planar tangle depicted in Figure \ref{Fig1PlanarTangle} with the second planar tangle shown in Figure \ref{Fig2ConnectIrrTangle} is given in Figure \ref{Fig3CompoTangle}. 
 \begin{figure}[h!]
\begin{center}
\scalebox{0.7}{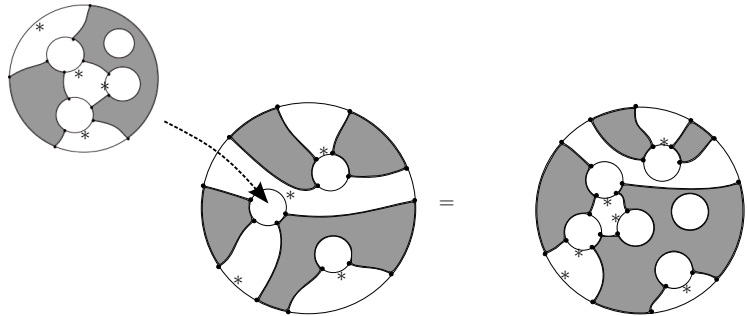}
\end{center}
\caption{\label{Fig3CompoTangle}: Composition of two planar tangles.}
\end{figure}

If $T_{1},\dots,T_{s}$ are planar tangles and $D_{i_{1}},\dots,D_{i_{s}}$ are distinct inner disks of $T$ such that $\deg T_{j}=\deg D_{i_{j}}$ for all $1\leq j \leq s$, we denote by $T\circ_{(D_{i_{1}},\dots,D_{i_{s}})}(T_{1},\dots,T_{s})$ the planar tangle obtained by iterating the composition with respect to the different inner disks. \\ \\
There also exists a natural involution operation on planar tangles. Number the regions adjacent to any disk of a tangle $T$ starting with the region preceding the first boundary point of the disk. Reflect the tangle in a diameter of the outer disk passing through its first region. Renumber the boundary point of the reflected tangle in such a way that the image of a first region is a first region again to obtain the tangle $T^*$.

\subsection{Planar algebras}

\begin{defi}[\cite{J99}]
A planar algebra $\mathcal{P}$ is a collection of finite-dimensional vector spaces $( \cP_n)_{n\in \mathbb{N}^{*}\cup\lbrace -,+\rbrace}$ such that each planar tangle $T$ of degree $k$ with $n$ inner disks $D_1, \dots, D_n $ of respective degree $k_{1},\dots,k_{n}$ yields a linear map
$$Z_{T}:\bigotimes_{1\leq i\leq n} \mathcal{P}_{k_{i}}\longrightarrow \mathcal{P}_{k}.$$
We require the composition of such maps to be compatible with the composition of planar tangles. More precisely, if $T'$ is another tangle of degree $k_{i_{0}}$ for some $1\leq i_{0}\leq n$, then 
$$Z_{T}\circ (\bigotimes_{i\not = i_{0}} \text{Id}_{\cP_{k_i}}\otimes Z_{T'})=Z_{T\circ_{D_{i_{0}}}T'}.$$
If every $\cP_n$ carries the structure of a $C^*$-algebra with multiplication induced by the natural multiplication tangles (see \cite{J99}) whose involution commutes with the involution of planar tangles, it is called a $C^*$-planar algebra.  
\end{defi}

A planar algebra is \emph{spherical} if the action of planar tangles is invariant under symmetries of the two-sphere (obtained by adding a point at infinity to $\R^2$). Note that the left and right trace tangles (as depicted in Figure \ref{tracetangles}) coincide up to a spherical symmetry.

\begin{figure}[h!]
\begin{center}
\scalebox{0.7}{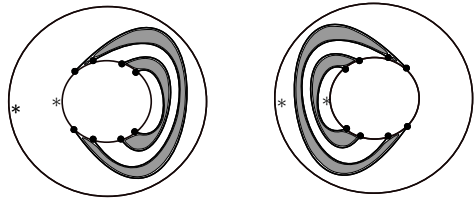}
\end{center}
\caption{\label{tracetangles}: Trace tangles of degree $4$.}
\end{figure}

\begin{defi}[\cite{J99}] \label{subfactorpa}
A subfactor planar algebra $\cP$ is a $C^*$-planar algebra such that
\begin{itemize}
\item $\cP$ is spherical,
\item $\dim \cP_+ = \dim \cP_- = 1, $
\item For every $n$, the linear map $\Tr_n$ induced by the $n$-th trace tangles satisfies $\Tr(x^*x) > 0$ for all $ 0 \neq x \in \cP_n$.
\end{itemize}
\end{defi}
The sequence of dimensions of the vector spaces $(\cP_{n})_{n\geq 1}$ of a planar algebra $\cP$ is denoted by $(\mu_{\cP}(n))_{n\geq 1}$. Note that this sequence corresponds to the sequence of moments of a real bounded measure, which we denote by $\mu_\cP$.
\subsection{Annular tangles and Hilbert modules over a planar algebra} \label{annular}

Since we will need to discuss the annular category of a planar algebra in the proof of Theorem \ref{mainA}, we introduce the relevant types of tangles in this section.

\begin{defi}[\cite{J01}]
Let $m,n \in \N_* \cup \{ +,- \}$.
An \emph{annular tangle} is a planar tangle together with a fixed choice of an inner disk which will henceforth be called the input disk. An annular $(m,n)$-tangle is an annular tangle whose outer disk has $2n$ boundary points and whose input disk has $2m$ boundary points.
\end{defi}

Similar to arbitrary tangles, annular tangles come with a natural composition operation. Given an annular $(m,n)$-tangle $T$ with distinguished disk $D$ and an annular $(l,m)$-tangle $S$ we obtain a natural composition tangle $T \circ_D S$ by isotoping $S$ into the input disk of $T$ such that boundary points meet and by erasing the common boundary afterwards. The input disk of $T \circ S$ is simply the image of the input disk of $S$ under this isotopy. If $T$ has internal disks other than the distinguished one, we can also compose it with an arbitrary tangle of the correct degree. More precisely, if $\tilde{D}$ is such an internal disk with $2k$ boundary points and $\tilde{S}$ is an arbitrary tangle with $2k$ outer boundary points the composition $T \circ_{\tilde{D}} \tilde{S}$ yields another annular tangle. \\
There also exists a natural notion of involution on annular tangles. Starting from an annular $(m,n)$-tangle $T$, we obtain an annular $(n,m)$-tangle $T^*$ by reflecting $T$ in a circle halfway between the inner and outer boundaries (after isotoping the input disk to the center of $T$). The marked points of the inner and outer disks of $T^*$ are defined as the images of the marked points of $T$ under this reflection which is well-defined since the boundary of the input disk of $T$ is mapped to the outer boundary of $T^*$ and the other way around. Moreover the reflection maps the boundary points of a non-distinguished internal disk of $T$ to the boundary points of a non-distinguished internal disk of $T^*$. See \cite{J01} for all this. \\
Let us quickly recall the definition of modules and Hilbert modules over a planar algebra.

\begin{defi}[\cite{J01}]
Let $\cP = (\cP_{n})_{n \in \N_* \cup \{ +,- \} }$ be a planar algebra. A \emph{left module} over $\cP$ is a graded vector space $V = (V_{n})_{n \in \N_* \cup \{ +,- \}}$ such that, given an annular $(m,n)$-tangle $T$ with input disk $D_1$ and other internal disks $D_i, \ i=2, \dots, s$ of degree $k_i$, there is a linear map 
\[Z_T: V_m \ot (\ot_{p=2}^s P_{k_p}) \to V_n. \]
 We require the family $(Z_T)_{T}$ to be compatible with both the composition of annular tangles and the composition with arbitrary tangles.
\end{defi}

\begin{defi}[\cite{J01}]
Let $\cP = (\cP_{n})_{n \in \N_* \cup \{ +,- \} }$ be a $C^*$-planar algebra. A $\cP$-module $V = (V_k)_{n \in \N_* \cup \{ +,- \} }$ will be called a Hilbert $\cP$-module if each $V_n$ is a finite-dimensional Hilbert space with inner product $\langle \cdot, \cdot \rangle_n$ such that 
\[ \langle Z_T(v,x_2,\dots,x_s), w \rangle_n = \langle v, Z_{T^*}(w,x_2^*,\dots,x_s^*) \rangle_m  \]
whenever $T$ is an annular $(m,n)$-tangle with non-distinguished internal disks $D_i, i=2,\dots,s$ of respective degree $k_i$, $\ x_i \in \cP_{k_i}$ and $v \in V_m, \ w \in V_n$.
\end{defi}
\begin{figure}[h!]
\begin{center}
\scalebox{0.3}{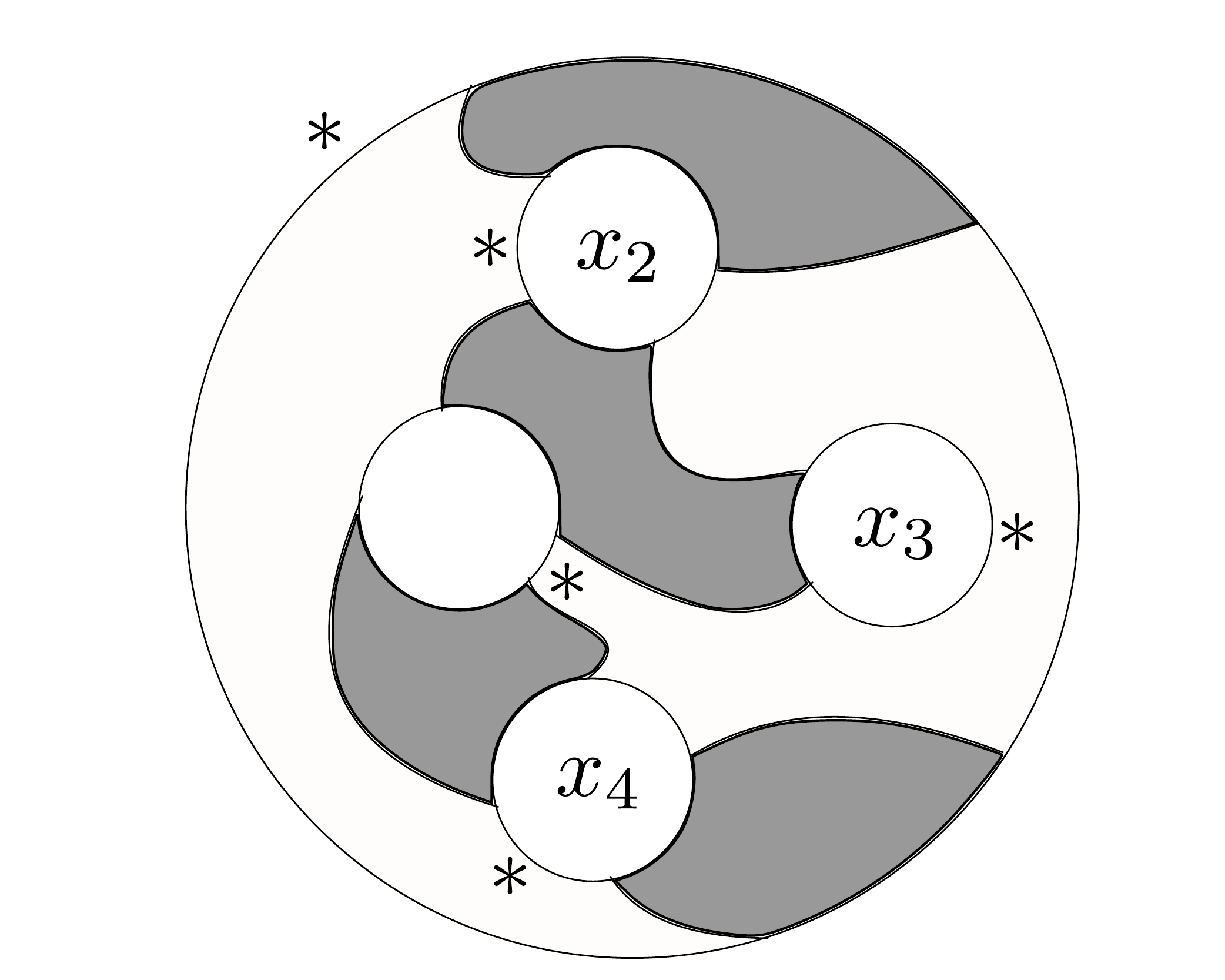}
\end{center}
\caption{\label{annulartangle}: Labelled annular $(2,2)$-tangle.}
\end{figure}
Since the Hilbert spaces in the above definition are all assumed to be finite-dimensional, an annular $(m,n)$-tangle $T$ with non-distinguished internal disks $D_i, \ i=2,\dots,s$ of respective degree $k_i$ which are labeled by $x_i \in \cP_{k_i}$ induces a bounded linear map $Z_T(\cdot, x_2,\dots,x_s): V_m \to V_n$. Lemma 3.12 in \cite{J01} provides us with an estimate on the norm of this map. We have
\[ \| Z_T(\cdot, x_2,\dots,x_s) \| \leq C_T \prod_{i=2}^s \| x_i \| \]
for some non-negative number $C_T \geq 0$ depending on $T$. 
To simplify notation, whenever we speak of labelled tangles we write $Z_T$ instead of $Z_T(\cdot, x_2,\dots,x_s)$. In order to be coherent, whenever we speak of a labelled tangle $T$, $Z_{T^*}$ means $Z_{T^*}(\cdot, x_2^*,\dots,x_p^*)$.

\begin{ex}[\cite{J01}]
Let $\cQ = (\cQ_{n})_{n \in \N_* \cup \{ +,- \} }$ be a spherical $C^*$-planar algebra with $C^*$-planar subalgebra $\cP$. The inner product 
\[ \langle v,w \rangle = \tr_n(w^* v) \qquad v,w \in \cQ_n, \ n \in \N_* \cup \{ +,- \} \]
turns $\cQ$ into a left Hilbert module over $\cP$.
\end{ex}

Let us single out a specific class of annular tangles which we will call special tangles.

\begin{defi}
A \emph{special} $(m,n)$-tangle is an annular $(m,n)$-tangle having exactly one non-distinguished disk of degree $(m+n)$. We ask for the first $2n$ boundary points of this disk to be connected by strings to the boundary points of the outer disk in such a way that marked boundary points are connected. The other $2m$ boundary points are required to be connected to the input disk with the marked boundary point of the input disk being connected to the last boundary point of the non-distinguished disk.
\end{defi}

Let $\cP = (\cP_{n})_{n \in \N_* \cup \{ +,- \} }$ be a $C^*$-planar algebra and let $\cH = (\cH_n)_{n \in \N_* \cup \{ +,- \} }$ be a Hilbert module over $P$. We say that a special $(m,n)-$tangle is labelled by $x \in \cP_{m+n}$ if the non-distinguished disk is. Let $T$ be such a special $(m,n)$-tangle labelled by $x \in \cP_{m+n}$ and let $S$ be a special $(l,m)$-tangle labelled by $y \in \cP_{l+m}$. Figure \ref{specialcomposition} defines a special $(l,n)$-tangle $T \circ S$ such that $Z_T \circ Z_S = Z_{T \circ S} : \cH_l \to \cH_n$.
\begin{figure}[h!]
\begin{center}
\scalebox{0.3}{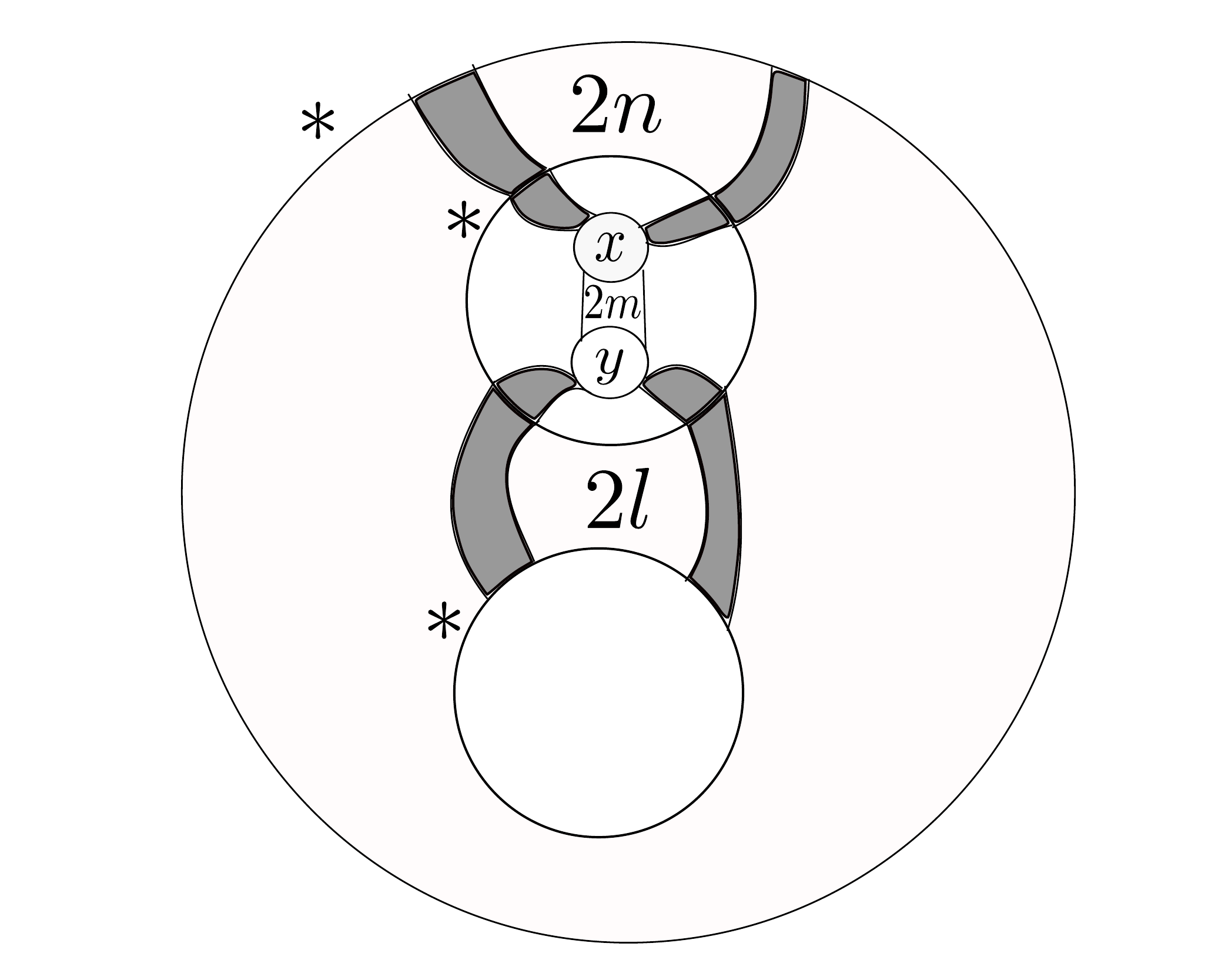}
\end{center}
\caption{\label{specialcomposition}: Composition $T \circ S$ of special tangles $T$ and $S$.}
\end{figure}
\\
Note that, since $\cH$ is a Hilbert module, we have $Z_T^* = Z_{T^*}$.
Given another special $(i,j)$-tangle $T'$ labelled by $z$, we define a special $(m+i,n+j)$-tangle $T \ot T'$ (the \emph{tensor tangle}) through Figure \ref{tensortangle}.
\begin{figure}[h!]
\begin{center}
\scalebox{0.3}{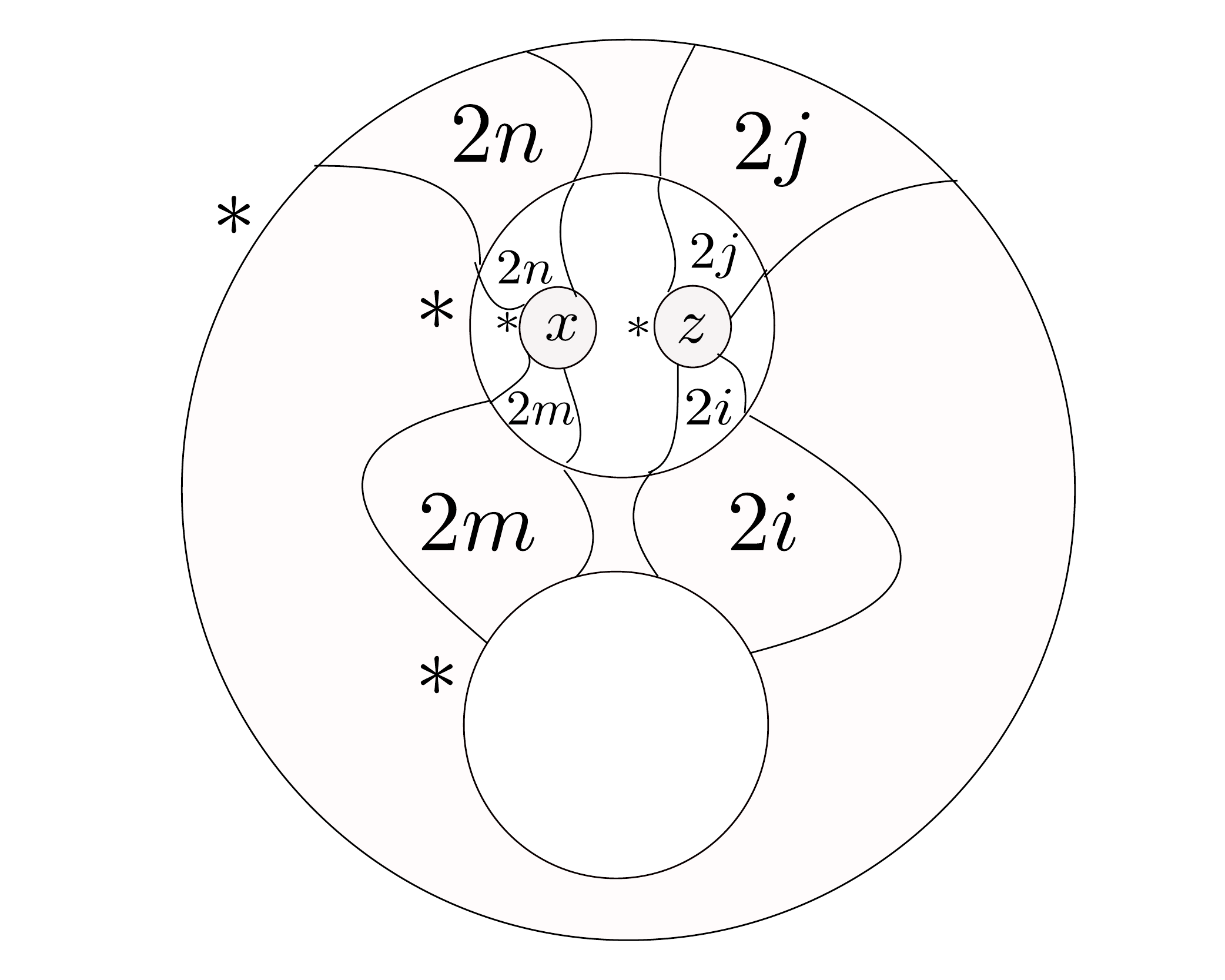}
\end{center}
\caption{\label{tensortangle}: Tensor tangle $T \ot T'$ (shading omitted in order to get a clear picture).}
\end{figure}
Clearly, in general the maps $Z_{T \ot T'}$ and $ Z_T \ot Z_{T'}$ will not coincide, however the tensor tangle will be useful in the specific case considered later on.

\subsection{The planar algebra associated to a finite dimensional $C^*$-algebra} \label{bipartitepa}

In \cite{J98}, given a finite bipartite graph $\Gamma = (V,E)$ and an eigenvector $\mu: V \to \C$ of the adjacency matrix of $\Gamma$, Jones constructs a planar algebra which satisfies all conditions in the definition of subfactor planar algebras (Definition \ref{subfactorpa}) except for the dimension condition. Every inclusion of finite dimensional $C^*$-algebras $A_0 \subset A_1$ induces a bipartite graph $\Gamma$ with vertex set $V = V_+ \cup V_-$ where the vertices in $V_+$ are the irreducible representations of $A_0$ and the vertices in $V_-$ are the irreducible representations of $A_1$. The number of edges between $a \in V_+$ and $b \in V_-$ is the multiplicity of the irreducible representation $a$ in the restriction of $b$ to $A_0$. Let $A_0 = \C$, let $\dim A_1 = d$ and decompose $A_1$ into simple components 
\[ A_1 = \oplus_{i=1}^s M_{m_i}(\C). \]
In this case, the bipartite graph $\Gamma$ has one vertex $a \in V_+$ and $s$ vertices $b_1, \dots, b_s \in V_-$ with $m_i$ edges $\beta_i^1, \dots, \beta_i^{m_i}$ connecting $a$ and $b_i$. The unique normalized Markov trace of the inclusion is given by $\tr(p_i) = \frac{m_i}{d}$ for a minimal projection $p_i \in M_{m_i}(\C)$. We define the spin vector $ \mu: V \to \C$ by setting 
\[ \mu(a) = 1, \quad \mu(b_i) = \frac{m_i}{\sqrt d} \quad i=1,\dots, s. \] This indeed defines an eigenvector $\mu$ of the adjacency matrix of $\Gamma$ with non-zero positive entries and non-zero eigenvalue $\delta = \sqrt d$ satisfying the normalization condition of \cite[Definition 3.5]{J98}. Note at this point that our notation differs from Jones's in the sense that the entries of our vector $\mu$ are the squares of the entries of the vector $\mu$ in \cite{J98}.

Let us analyze the planar algebra $\cP^{\Gamma} = (\cP^{\Gamma}_n)_{n \in \N_{*} \cup \{+, - \}}$ of this particular bipartite graph $\Gamma = \Gamma(A)$.
Recall that $\cP^{\Gamma}_n$ is defined as the vector space with basis
\[ B_n \ = \ \{\xi; \ \xi = \text{ loop of length } 2n \text{ starting and ending at } a \} \]
for $n \in \N_{\geq 1} \cup \{+ \}$, whereas $\cP^{\Gamma}_{-}$ is the $s$-dimensional vector space spanned by $b_1, \dots, b_s$. Let us represent a loop $\eta \in \cP^{\Gamma}_n$ by a tuple 
\[ \eta = (\be_{i_1}^{k_1},\be_{i_1}^{k_2},\be_{i_2}^{k_3},\dots,\be_{i_n}^{k_{2n-1}}, \be_{i_n}^{k_{2n}}). \]
Let $T$ be a planar tangle of degree $n$ with inner disks $D_1,\dots, D_l$. Let $\eta_i \in B_i, \ i =1, \dots, l$ be basis loops. To describe the action $Z_T: \bigotimes_{i=1}^l \cP^{\Gamma}_i \to \cP^{\Gamma}_n$, it suffices to specify the coefficients of $Z_T(\eta_1,\dots,\eta_l)$ in the basis expansion 
\[ Z_T(\eta_1,\dots,\eta_l) = \sum_{\eta_0 \in B_n} c(\eta_0, \eta_1, \dots, \eta_l) \  \eta_0 . \]
In order to do so, recall that a \emph{state} on $T$ is function $\sigma$ which maps strings of $T$ to edges of $\Gamma$, shaded regions of $T$ to odd vertices of $\Gamma$ and unshaded regions to even vertices of $\Gamma$ (in our case the last part can be omitted since $V_+ = \{ a \}$). A state must satisfy the following compatibility condition. If a string $S$ is adjacent to a region $R$, then $\sigma(R)$ must be an endpoint of the edge $\sigma(S)$. \\
A state $\sigma$ is compatible with a loop $\eta$ at the disk $D$ (outer or inner) if reading the output of $\sigma$ clockwise around $D$ starting from the first region produces $\eta$. \\
Isotope every string $S$ such that its $y$-coordinate function is non-constant and any singularity of the $y$-coordinate function is a local minimum or maximum. With any such singularity $s$ appearing on a string in $T$, we associate the value $\rho(s) = \sqrt{\frac{\mu(v)}{\mu(w)}}$, where $v$ is the vertex labelling the convex side and $w$ is the vertex labelling the concave side of the singularity. For a more formal discussion of this procedure, see \cite{JP10}.\\
The coefficient $c(\eta_0, \eta_1, \dots, \eta_l)$ is defined by the formula
\[ c(\eta_0, \eta_1, \dots, \eta_l) \quad = \quad \sum_{\sigma} \prod_{s \text{ singularity} \atop \text{in } T} \rho(s). \]
Here, the first sum runs over all states $\sigma$ which are compatible with $\eta_i$ at disk $D_i$ for every $i=0,\dots,l$.
Note that the partition function $Z:V \to \C$ of $\cP^{\Gamma}$ is given by $ \ Z(a) = 1, \ Z(b_i) = \mu(b_i)^2 = \frac{m_i^2}{d}$. By \cite[Theorem 3.6]{J98}, for every $k \geq 0$, the partition function induces a normalized trace $\Tr_k$ on $\cP^{\Gamma}_k$ through the formula $\Tr_k(x) = (\sqrt d)^{-k+1} Z(\hat x), \ x \in \cP_k, $ where $\hat x$ is the $0$-tangle obtained by connecting the first $k$ boundary points to the last $k$ as in Figure \ref{tracetangles}.  \\

As explained in section 5 of \cite{J98}, we can identify $\cP^{\Gamma}_{n}$ with the $n$-th algebra $A_n$ in the Jones tower of the inclusion $\C \subset A$. When $n=1$, the identification is simply given by mapping the loop $(\be_{i}^k, \be_i^l)$ to the $(k,l)$-matrix unit $e_{kl}^i$ of $b_i$. It can be easily checked that this mapping indeed identifies $\cP^{\Gamma}_1$ and $A$ as $C^*$-algebras. Moreover the trace $\Tr_1$ coincides with the normalized Markov trace on $A$ under this identification which follows immediately by applying the definition of $\Tr_1$ to a matrix unit of $A$. It is pointed out in \cite{J98} that for general inclusions of finite-dimensional $C^*$-algebras these traces need not be the same. We will denote the GNS-Hilbert spaces associated to these traces by $\cH_k := L^2(\cP^{\Gamma}_k, \Tr_k)$.

\section{The reconstruction theorem} \label{reconstruction}

In this section we will reformulate and prove Theorem \ref{mainA}.
Let $A$ be a finite-dimensional $C^*$-algebra and let $\al$ be a faithful, centrally ergodic action $\al: A \to A \ot \Pol(\G)$ of a compact quantum group $\G$ which preserves the Markov trace of the inclusion $\C \subset A$ (cf. Section \ref{actions}). Recall from Lemma \ref{ban:tow}, that $\al$ induces a sequence of actions $(\al_i)$ on the Jones tower
\[ \C \subset A_1 = A \subset A_2 = \langle A_1, e_1 \rangle  \subset \cdots  \]
of $\C \subset A$. Let $\Gamma(A)$ be the bipartite graph of this inclusion with spin vector $\mu$ and graph planar algebra $\cP^{\Gamma(A)} = (\cP^{\Gamma(A)}_n)_{n \in \N_* \cup \{ +,- \}}$ as in Section \ref{bipartitepa} where we identify $\cP^{\Gamma(A)}_n$ with $A_n$ as explained near the end of the previous section. The next result is the main result of the article \cite{B05}.

\begin{thm}[\cite{B05}]\label{resultBanica}
The graded vector space $\cP(\al) = (\Fix(\al_n))_{n \in \N_* \cup \{ +,- \}}$ is a subfactor planar subalgebra of $\cP^{\Gamma(A)}$.
\end{thm}

We will prove that the converse of this result holds as well. The following theorem is a reformulation of Theorem \ref{mainA}. 

\begin{thm} \label{main1}
Let $\cQ$ be a subfactor planar subalgebra of $\cP^{\Gamma(A)}$.\\
There exists a compact quantum group $\G$ and a faithful, $\tr$-preserving, centrally ergodic action $\al: A \to A \ot \Pol(\G)$ such that $\cQ = \cP(\al)$. Moreover, the map
\begin{align*}  
\{\text{faithful } \tr \text{-pres., centrally erg. actions on } A \} &\to \{ \text{subfactor planar subalg. of } \cP^{\Gamma(A)}  \}, \\
 \al &\mapsto \cP(\al),
\end{align*}
is one-to-one with inverse map denoted by $\cP \mapsto \al(\cP)$.
\end{thm}

In order to prove this theorem, we will apply Woronowicz's Tannaka-Krein duality (\cite{Wo88}) to a concrete rigid $C^*$-tensor category whose morphisms are induced by special tangles. We refer to the book \cite{NT13} for details on the definition of rigid $C^*$-tensor categories and the precise statement of Woronowicz's theorem. \\
Let $\cQ = (\cQ_n)_{n \in \N_* \cup \{ +,- \}}$ be a subfactor planar subalgebra of $\cP^{\Gamma(A)}$. Recall from section \ref{annular} that, since $(\cH_n = L^2(\cP^{\Gamma(A)}_n, \tr_n))_{n \in \N_* \cup \{ +,- \}}$ is a Hilbert module over $\cQ$, every special $(k,l)$-tangle $T$ induces a bounded linear map $Z_T: \cH_k \to \cH_l$ such that $Z_T^* = Z_{T^*}$. \\
During the rest of this section, we identify the index value $+$ with $0$ if not explicitly specified otherwise. Set $\cH = \cH_1$.\\
For $k,m \geq 0$, define the \emph{concatenation tangle} $M_{k,m}$ as in Figure \ref{concatenation}.
\begin{figure}[h!]
\begin{center}
\scalebox{0.3}{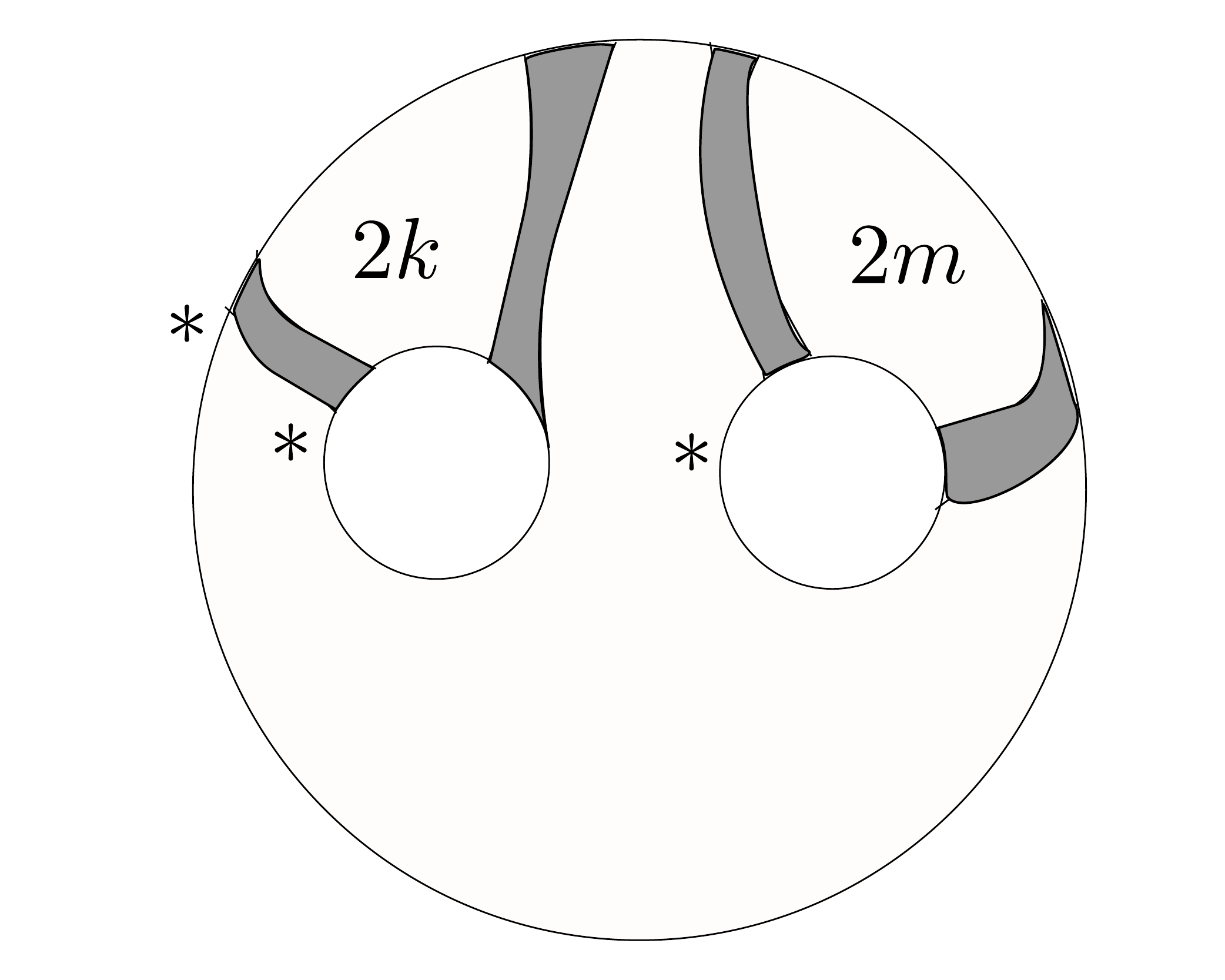}
\end{center}
\caption{\label{concatenation}: Concatenation tangle $M_{k,m}$.}
\end{figure}  

\begin{lem} \label{helplemma}
The maps $U_{k,m}:= Z_{M_{k,m}}: \cH_k \ot \cH_m \to \cH_{k+m}$ define a family of unitary operators $U = (U_{k,m})_{k,m \geq 0}$ satisfying
\[ U_{l,n} \circ (Z_T \ot Z_S ) = Z_{T \ot S} \circ U_{k,m},  \]
whenever $S$ is a special $(m,n)$-tangle, $T$ is a special $(k,l)$-tangle and $T \ot S$ is the tensor tangle defined in Figure \ref{tensortangle}.
\end{lem}

\begin{proof}
By definition of the action of tangles on $\cP^{\Gamma(A)}$, the operator $U_{k,m}$ maps the tensor product $\xi \ot \eta$ of a basis loop $\xi \in \cH_k$ of length $2k$ and a basis loop $\eta \in \cH_m$ of length $2m$ to a non-zero scalar multiple of the concatenated loop $(\xi,\eta) \in \cH_{k+m}$. On the other hand, since any basis loop of length $2(k+m)$ in $\cH_{k+m}$ reaches the base vertex $a$ after each even number of steps and in particular after $2k$ steps, it follows that $U_{k,m}$ is surjective. Recall that the square of the norm $\| x \|_{\cH_{i}}^2 = \Tr_i(x^* x), \ x \in \cH_i$ can be implemented by the tangle shown in Figure \ref{normtangle}.
\begin{figure}[h!]
\begin{center}
\scalebox{0.3}{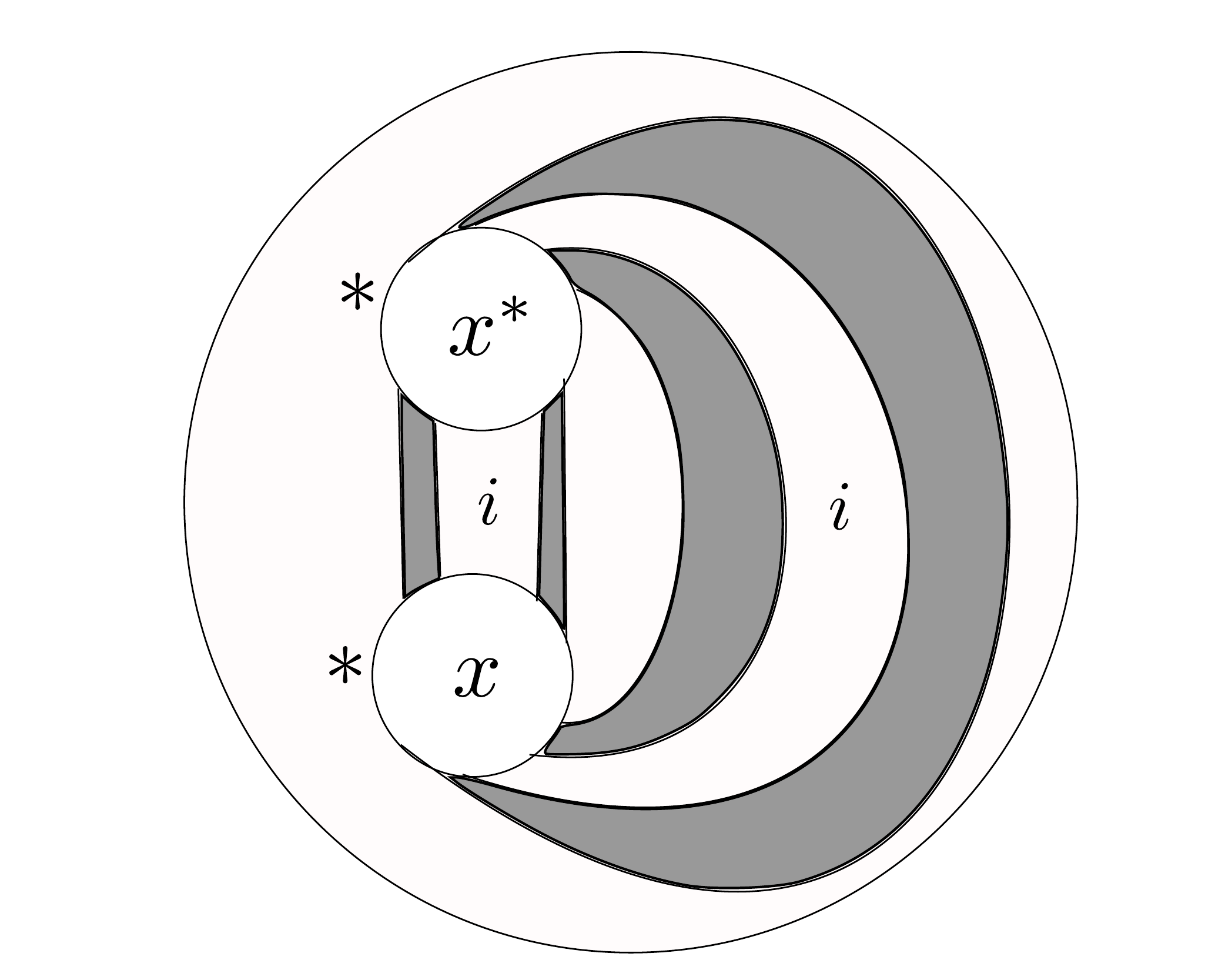}
\end{center}
\caption{\label{normtangle}: Tangle implementing the norm on $\cH_i$.}
\end{figure}
The fact that $U_{k,m}$ is isometric follows by composing this tangle with $M_{k,m}$ and its adjoint in the obvious way. The fact that the family $U$ conjugates the tensor product of maps into the tensor product of tangles also follows immediately by drawing the associated tangles. 
\end{proof}

By the previous lemma, we inductively obtain a sequence of unitaries $(U_m)_{m \geq 0}$ such that $U_m: \cH^{\ot m} \to \cH_m$ implements a natural isomorphism of Hilbert spaces $\cH^{\ot m} \cong \cH_m$ and such that $U_{k+m} = U_{k,m} (U_k \ot U_m)$.

\begin{lem} \label{mapslemma}
The category $\cC_{\cQ}$ with $\Obj(\cC_{\cQ}) = \N$ and morphism spaces
\[ \Mor(k,l) =  \{ U_l^* Z_T U_k : \cH^{\ot k} \to \cH^{\ot l} \ ; \ T \text{ labelled special } (k,l)-\text{tangle}  \} \subset B(\cH^{\ot k},\cH^{\ot l}) \]
is a rigid $C^*$-tensor category with tensor functor given by $k\ot l = k+l$ on objects and by the tensor product of linear maps on Hilbert spaces on morphisms. In particular, there exists a compact matrix quantum group $\G$ with fundamental corepresentation $u \in B(\cH) \ot \Pol(\G)$ such that $\Mor(u^{\ot k}, u^{\ot l}) = \Mor(k,l)$.
\end{lem}

\begin{proof}
The observations prior to Lemma \ref{helplemma} show that $\cC_{\cQ}$ is a $C^*$-category and Lemma \ref{helplemma} itself shows that $\cC_{\cQ}$ is a $C^*$-tensor category.
The map $R = \bar{R} \in \Mor(0,2)$ induced by the special $(0,2)-$tangle shown in Figure \ref{conjugatetangle}
\begin{figure}[h!]
\begin{center}
\scalebox{0.3}{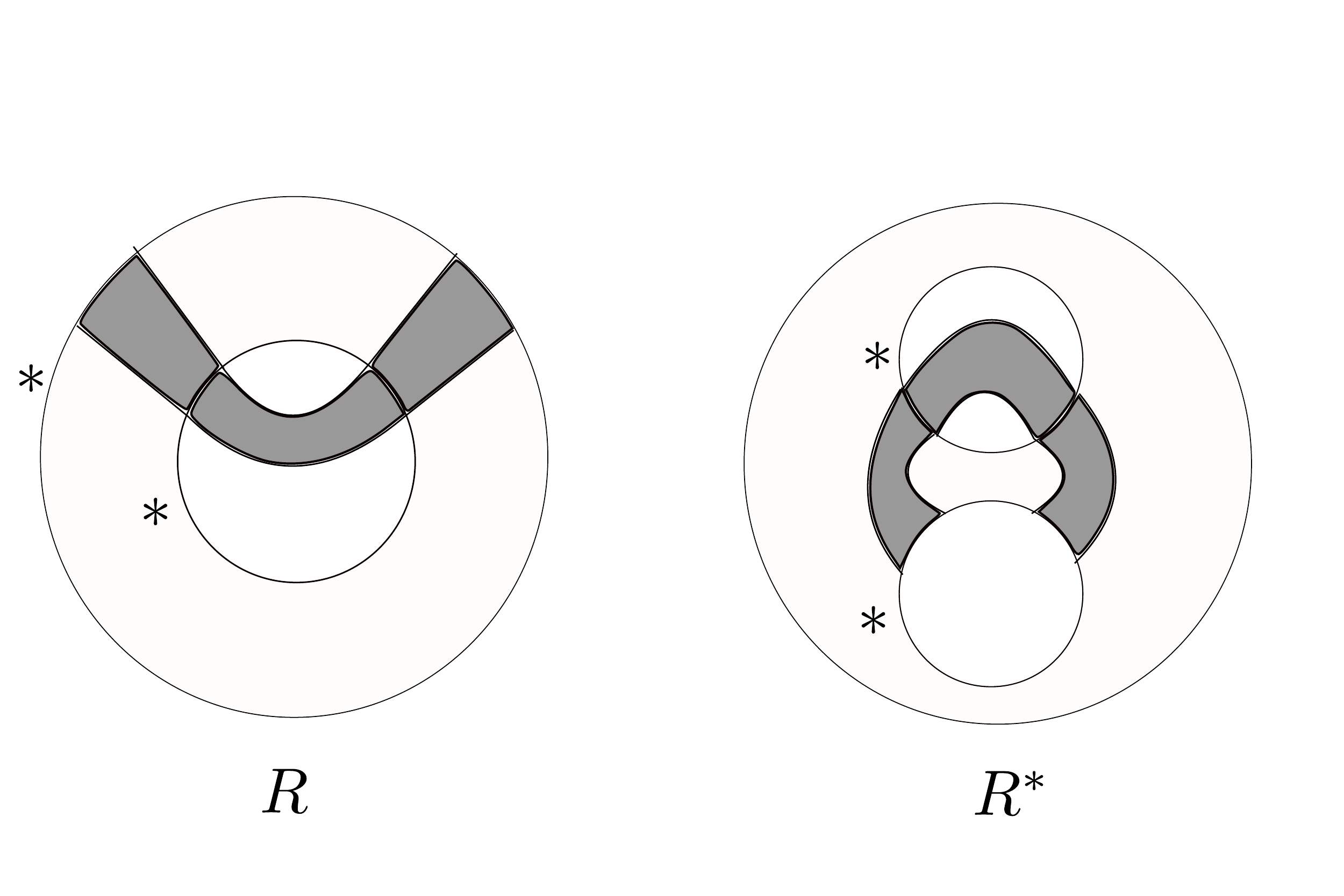}
\end{center}
\caption{\label{conjugatetangle}: The conjugate tangle $R$ and its adjoint.}
\end{figure}
satisfies the conjugate equations for the generating object which is therefore self-conjugate. By \cite{Wo88}, the category is rigid and Tannaka-Krein duality applies.
\end{proof}

\begin{proof}[Proof of Theorem \ref{main1}]
Let $u \in B(\cH) \ot \Pol(\G)$ be the fundamental representation of the quantum group $\G$ obtained in the previous lemma. Consider the linear map $\al_u: A \to A \ot \Pol(\G)$ as in Subsection \ref{actions}. We have to prove that $\al_u$ is a faithful, $\tr$-preserving, ergodic coaction such that $\cQ = \cP(\al_u)$. Note that, since the coefficients of $u$ generate $\Pol(\G)$ by definition, the faithfulness of $\al_u$ is immediate. We invoke Lemma \ref{actionlemma} to establish the missing properties of $\al_u$. The multiplication map $m: A \ot A \to A$, the unit map $\eta: \C \to A$ and the Markov trace $\tr:A \to \C$ are induced by the special tangles drawn in Figure \ref{multiplicationunit}.
\begin{figure}[h!]
\begin{center}
\scalebox{0.3}{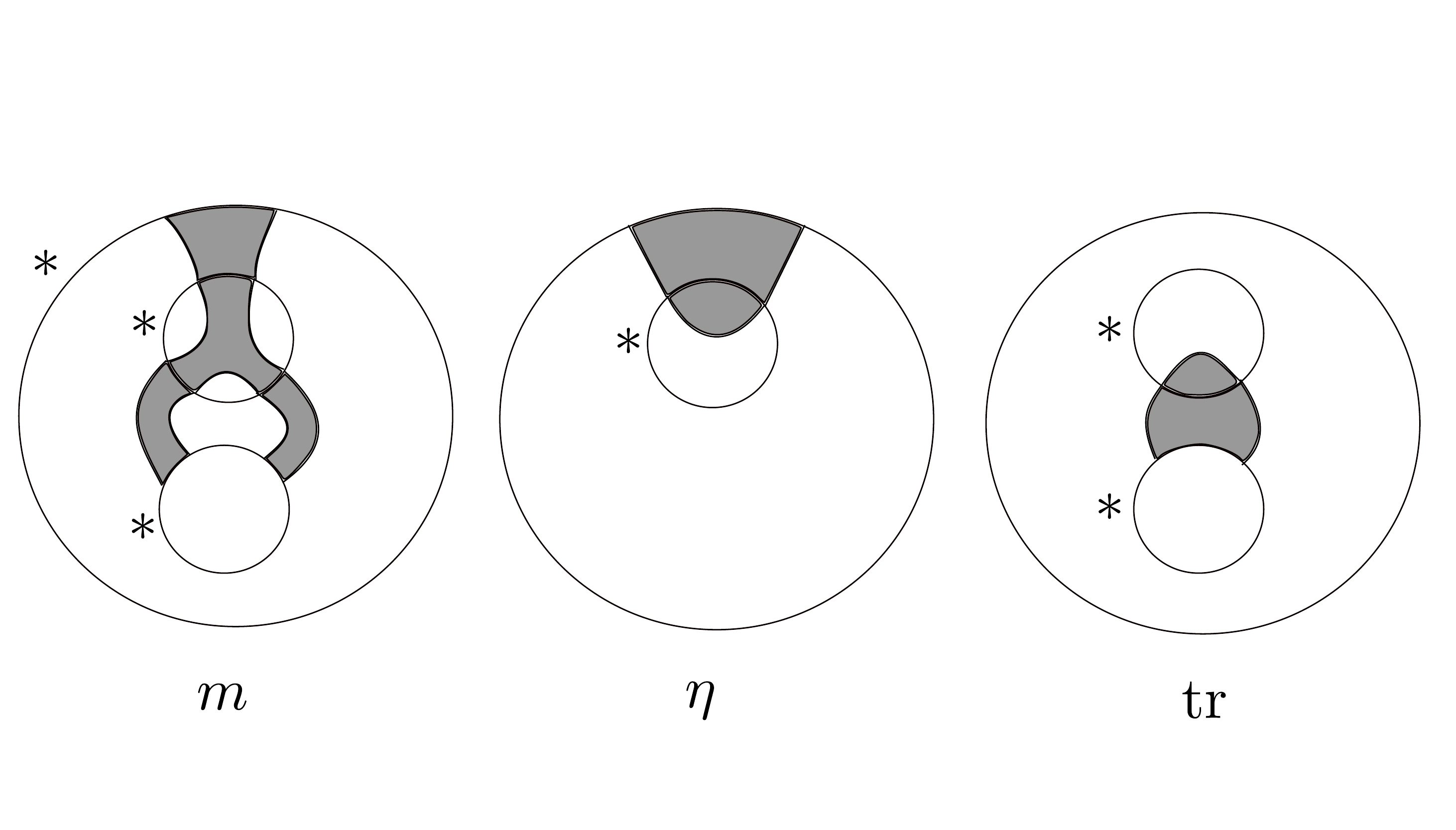}
\end{center}
\caption{\label{multiplicationunit}: Multiplication tangle $m$, unit tangle $\eta$ and trace tangle $\tr$.}
\end{figure}
Hence, $m \in \Mor(2,1)$, $\eta \in \Mor(0,1)$ and $\tr \in \Mor(1,0)$ and by Lemma \ref{actionlemma}, $\al_u$ is multiplicative, unital and $\tr$-preserving. Since $\cQ$ is a subfactor planar algebra, we have $ \dim Z(A) \cap \Fix(\al_u) = \dim \cQ_- = 1$ which implies the central ergodicity of $\al_u$. \\
If we start with a faithful, centrally ergodic, $\tr$ preserving action $\al$ on $A$, it is shown in \cite[Section 2]{B05} that $\cP(\al)_n = \Mor(\1,u_{\al}^n)$ for all $n \geq 0$. As a consequence, if we associate to $\cP(\al)$ another action $\beta$ of a compact quantum group $\mathbb H$ through the first part of Theorem \ref{main1}, we have $\Mor(\1,u_{\al}^n) = \Mor(\1, u_{\beta}^n)$ for all $n$ and by Frobenius duality and the Tannaka-Krein theorem we get an isomorphism $\Pol(G) \cong \Pol(\mathbb H)$ intertwining the actions. On the other hand, if we start with a subfactor planar algebra $\cQ$ of $\cP^{\Gamma(A)}$ and produce an action $\al(\cQ)$, then $\cQ$ is completely determined by $\al(\cQ)$ since $\cQ_n = \Mor(\1,u_{\al(\cQ)}^n)$ for all $n \geq 0$ and $\cQ_{-} = Z(A) \cap \Fix(\al(\cQ))$.
\end{proof}

\begin{ex} \label{TLJ}
If $A$ is a finite-dimensional $C*$-algebra, every quantum group $\G$ acting faithfully and $\tr$-preservingly on $A$ is contained in $\G_{aut}(A,\tr)$ as a quantum subgroup through and isomorphism $C(\G_{aut}(A,\tr)) \to C(\G)$ which intertwines the universal action of $\G_{aut}(A,\tr)$ and the action of $\G$. On the other hand, every subfactor planar algebra $\cP^{\Gamma(A)}$ contains the Temperley-Lieb-Jones $\rm{TLJ}$ algebra of modulus $\delta= \sqrt{\dim A}$ (see for instance \cite{J99}) as a planar subalgebra. Therefore, the universal action of $\G_{aut}(A,\tr)$ corresponds to the Temperley-Lieb-Jones algebra through Theorem \ref{main1}.
\end{ex}
\section{Planar tangles and non-crossing partitions} \label{partitionsection}

We will start this paragraph by recapitulating useful facts from the study of non-crossing partitions. Non-crossing partitions are recurring tools in the theory of compact quantum groups (see for instance \cite{TW15}) and in free probability \cite{NS06} as they emulate the absence of commutativity relations quite nicely. We will describe an underlying partition structure of planar tangles in order to express freeness of planar tangles in a simpler manner in Section \ref{freeproductsection}.
\subsection{Partial partitions}
A partition $p$ of order $k$ is a partition of $\llbracket 1,k\rrbracket := \{ 1, \dots, k \}$ into disjoint subsets. The subsets of $\llbracket 1,k\rrbracket$ which appear in the partition are called the blocks of $p$. The number of blocks of $p$ is denoted by $\vert p\vert$. A partition $p$ of order $k$ yields an equivalence relation $\sim_{p}$ on $\llbracket 1,k\rrbracket$ by setting $x\sim_{p}y$ whenever $x$ and $y$ are in the same block of $p$. We get a bijection between equivalence relations on $\llbracket 1,k\rrbracket$ and partitions of order $k$. The set of partitions of order $k$ is denoted by $P(k)$.\\
A partition of $k$ is depicted by drawing $k$ points on an horizontal line numbered from $1$ to $k$ and by linking two points whenever they belong to the same block of the partition. An example is drawn below.
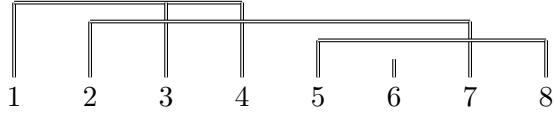
\begin{figure}[h!]
\centering
\begin{tikzpicture}
\node (1) at (1,0){$1$};
\node (2) at (2,0){$2$};
\node (3) at (3,0){$3$};
\node (4) at (4,0){$4$};
\node (5) at (5,0){$5$};
\node (6) at (6,0){$6$};
\node (7) at (7,0){$7$};
\node (8) at (8,0){$8$};

\draw [double] (1) -- (1,1.25);
\draw [double]  (1,1.25)--(4,1.25);
\draw [double]  (3,1.25)--(3);
\draw [double] (4,1.25)--(4);
\draw [double] (2)--(2,1);
\draw [double] (2,1)--(7,1);
\draw [double] (7,1)--(7);
\draw [double] (5)--(5,0.75);
\draw [double] (5,0.75)--(8,0.75);
\draw [double] (8,0.75)--(8);
\draw [double] (6)--(6,0.5);
\end{tikzpicture}
\caption{\label{figPart}Partition $\left\lbrace \lbrace 1,3,4\rbrace, \lbrace 2,7\rbrace, \lbrace 5,8\rbrace,\lbrace 6\rbrace\right\rbrace$ with $4$ blocks.}
\end{figure}\\
We define an order relation on $P(k)$ by setting $p\leq p'$ whenever each block of $p$ is included in a block of $p'$. This order relation turns $P(k)$ into a lattice, and we denote by $p\wedge p'$ (resp. $p\vee p'$) the infimum (resp. the supremum) of two partitions $p$ and $p'$. See \cite{NS06} for details on this lattice structure.\\ 
A partition $p$ of order $n$ is called non-crossing whenever the following condition holds. If $i<j<k<l$ are integers between $1$ and $n$ such that $i\sim_{p}k$ and $j\sim_{p}l$, then $i\sim_{p}j\sim_{p}k\sim_{p}l$. The set of non-crossing partitions of order $k$ is denoted by $NC(k)$. This set forms a sublattice of $P(k)$. A partition of $k$ is called an interval partition if all its blocks are intervals of integers. For example, $\lbrace \lbrace 1,2,3\rbrace,\lbrace 4,5\rbrace,\lbrace 6,7,8\rbrace\rbrace$ is an interval partition. The set of interval partitions is denoted by $\mathcal{I}(k)$.\\
Finally, we call a partition even if all its blocks have even cardinalities.

\subsubsection{Partial partitions and Kreweras complement}\label{notaKreweras}
\begin{defi}
Let $k\geq 1$. A partial partition of $\llbracket 1,k\rrbracket$ is a pair $(p,S)$ consisting of a subset $S$ of $\llbracket 1,k\rrbracket$ and a a partition $p$ of order $\vert S\vert$.
\end{defi}
Equivalently, a partial partition is a set $\lbrace B_{1},\dots,B_{r}\rbrace$ of disjoint subsets of $\llbracket 1,n\rrbracket$, with $S=\bigcup B_{i}$.
We denote by $P(S) \ (NC(S))$ the set of partial partitions (non-crossing partial partitions) with support $S$. 
Two partial partitions $(p,S)$ and $(p',S^{c})$ yield a partition $(p,S)\vee(p',S^{c})$ of $\llbracket 1,n\rrbracket$ by identifying $ \llbracket 1,|S| \rrbracket$ with $S$ (resp. $ \llbracket 1,|S^c| \rrbracket$ with $S^c$) as ordered sets and by considering the union of $p$ and $p'$. 
\begin{defi}
Let $(p,S)$ be a non-crossing partial partition of $S$, the Kreweras complement $\kr(p)$ of $p$ is the biggest partial partition $(p',S^{c})$ such that $(p,S)\vee(p',S^{c})$ is non-crossing.
\end{defi}
In the above definition, the order on $NC(S^{c})$ is the refinement order coming from the canonical bijection between $NC(S^{c})$ and $NC(\vert S^{c}\vert)$. The proof of the following lemma is straightforward.
\begin{lem}\label{characKreweras}
Let $(p,S)$ be a partial partition of $ \llbracket 1,n \rrbracket$. Then $\kr(p)$ is the partial partition with support $S^{c}$ defined by
$$i\sim_{\kr(p)}j\Leftrightarrow k\not\sim_{(p,S)}l,\text{ for all }k\in \llbracket i,j\rrbracket \cap S, l\in S\setminus \llbracket i,j\rrbracket.$$
\end{lem}
 \subsubsection{Non-crossing partitions and free probability}\label{nonCrossFreeProba}
We briefly review here the relations between non-crossing partitions and free probability which will be required in this paper. The interested reader should refer to \cite{NS06} for a thorough reading on the subject.\\
For any sequence $(c(n))_{n\geq 1}$ of real numbers and any partition $p$ with blocks $B_{1},\dots,B_{r}$, we denote by $c(p)$ the quantity 
$$c(p)=\prod_{1\leq i\leq r}c(\vert B_{i}\vert).$$
Let $\mu$ be a probability measure on $\R$ whose moments $(m_{\mu}(n))_{n\geq 1}$ are finite. Then, there exist two real sequences constructed from $(m_\mu(n))_{n\geq 1}$, both of which completely characterize the measure $\mu$.
\begin{itemize}
\item The sequence $(c_{\mu}(n))_{n\geq 1}$ of free cumulants of $\mu$ is defined by induction with the formula
$$m_{\mu}(n)=\sum_{p\in NC(n)}c_{\mu}(p).$$
\item The sequence $(b_{\mu}(n))_{n\geq 1}$ of Boolean cumulants of $\mu$ is defined by induction with the formula 
$$m_{\mu}(n)=\sum_{I\in \mathcal{I}(n)}b_{\mu}(p).$$
\end{itemize}
Suppose that $\nu$ is another probability measure with finite moments. Then, the moments of the free convolution $\mu\boxtimes\nu$ of the measures $\mu$ and $\nu$ are finite as well, and its sequence of free cumulants satisfies the relation $c_{\mu\boxtimes \nu}(n)=\sum_{p\in NC(n)}c_{\mu}(p)c_{\nu}(K(p)),$ where $K(p)$ denotes the usual Kreweras complement of a non-crossing partition $p$ (see \cite{NS06} or Section \ref{surgeryNoncrossPart} for a definition of the Kreweras complement for a non-crossing partition). Surprisingly, Belinschi and Nica discovered in \cite{BN08} that the Boolean cumulants of a free convolution of probability measures satisfies the same relation. More precisely,
\begin{equation}\label{booleanFreeNB}
b_{\mu\boxtimes \nu}(n)=\sum_{p\in NC(n)}b_{\mu}(p)b_{\nu}(K(p)).
\end{equation}

\subsection{Non-crossing partitions and irreducible planar tangles}\label{noncrossNotation}
In this subsection, we associate to each planar tangle $T$ a non-crossing partition $\pi_{T}$ and express the corresponding Kreweras complement in terms of shaded regions of $T$. 
\subsubsection{Non-crossing partition associated to a planar tangle}
This non-crossing partition plays an important role in the description of the free product of planar algebras.
\begin{defi}
Let $T$ be a planar tangle of degree $k$. The partition of $T$, denoted by $\pi_{T}$, is the set partition of $\llbracket 1, 2k\rrbracket$ such that  
$i\sim_{\pi_{T}}j$ if and only if the distinguished points $i$ and $j$ of the outer disk $D_{0}$ belong to the closure of the same connected component of $\Gamma T\setminus \partial D_{0}$.
\end{defi}
Note that $\pi_{T}$ is well defined, since a distinguished point of the outer disk $D_{0}$ is in the closure of a unique connected component of $\Gamma T\setminus \partial D_{0}$.
\begin{lem}
$\pi_{T}$ is an even non-crossing partition of $2k$.
\end{lem}

\begin{proof}
Let $1\leq i<j<k<l\leq 2k$ be such that $i\sim_{T}k$ and $j\sim_{T}l$. Thus, there exist a path $\rho_{1}$ in $\overline{\Gamma T\setminus D_{0}}$ between $i$ and $k$ and a path $\rho_{2}$ in $\overline{\Gamma T\setminus D_{0}}$ between $j$ and $l$. Since the distinguished points on the outer disk are numbered counterclockwise, $\gamma_{1}$ and $\gamma_{2}$ intersect. Therefore, the four points are in the closure of the same connected component of $\Gamma T\setminus D_{0}$ and $i\sim_{T}j\sim_{T}k\sim_{T}l$. $\pi_{T}$ is thus non-crossing. \\
Since each inner disk has an even number of distinguished points and each curve connects two distinguished points, a counting argument yields the parity of the size of the blocks.
\end{proof}
Reciprocally, an even non-crossing partition $\pi$ of $2k$ yields an irreducible planar tangle $T_{\pi}$ such that $\pi_{T_{\pi}}=\pi$. The construction is done recursively on the number of blocks as follows.
\begin{enumerate}
\item If $\pi$ is the one block partition of $2k$, $T_{\pi}$ is the planar tangle with one outer disk $D_{0}$ of degree $2k$, one inner disk $D_{1}$ of degree $2k$, and a curve between the point $i$ of $D_{0}$ and the point $i$ of $D_{1}$.
\item Let $r\geq 2$, and suppose that $P_{p}$ is constructed for all even non-crossing partitions $p$ having less than $r$ blocks. Let $\pi$ be an even non-crossing partition with $r$ blocks. Let $B=\llbracket i_{1},i_{2}\rrbracket$ be an interval block of $\pi$, and let $\pi'$ be the non-crossing partition obtained by removing this block (and relabelling the integers). $\pi'$ has also even blocks. Let $T_{\pi,B}$ be the planar tangle consisting of an outer disk $D_{0}$ of degree $2k$ and two inner disks $D_{1}$ and $D_{2}$ of respective degree $i_{2}-i_{1}+1$ and $2k-(i_{2}-i_{1}+1)$, and curves connecting 
\begin{itemize}
\item $i_{D_{0}}$ to $i_{D_{2}}$ for $i<i_{1}$.
\item $i_{D_{0}}$ to $(i-i_{1}+1)_{D_{1}}$ for $i_{1}\leq i\leq i_{2}$ if $i_{1}$ is odd, and $i_{D_{0}}$ to $(i-i_{1})_{D_{1}}$ for $i_{1}\leq i\leq i_{2}$ if $i_{1}$ is even. 
\item $i_{D_{0}}$ to $(i-(i_{2}-i_{1}+1))$ for $i> i_{2}$.
\end{itemize}
Set $T_{\pi}=T_{\pi,B}\circ_{D_{2}}T_{\pi'}$. Note that the resulting planar tangle does not depend on the choice of the interval block $B$ in $\pi$.
\end{enumerate}
By construction, $T_{\pi}$ is irreducible and $\pi_{T_{\pi}}=\pi$. The inner disk of $T_{\pi}$ corresponding to the block $B$ of $\pi$ is denoted by $D_{B}$. An outer point $i$ is linked to $D_{B}$ in $T_{\pi}$ if and only if $i\in B$ in $\pi$. The distinguished boundary point of $D_{B}$ is the one linked to the smallest odd boundary point on the outer disk. The next lemma is a straightforward deduction from the previous description of $T_{\pi}$.
\begin{lem}\label{orderPartitionCompoTangles}
Suppose that $\pi$ and $\pi'$ are two even non-crossing partitions such that $\pi'\leq \pi$. Denotes by $B_{1},\dots,B_{r}$ the blocks of $\pi$. Then,
$$T_{\pi'}=T_{\pi}\circ_{D_{B_{1}},\dots,D_{B_{r}}}(T_{\pi'_{\vert B_{1}}},\dots,T_{\pi'_{\vert B_{r}}}).$$
Reciprocally, if $\pi$ is an even non-crossing partition, $B$ is a block of $\pi$ and $(\tau,B)$ is an even non-crossing partial partition of $B$, then 
$$T_{\pi}\circ_{D_{B}}T_{\tau}=T_{\pi'},$$
where $\pi'$ is the non-crossing partition $\pi\setminus\lbrace B\rbrace\cup (\tau,B)$. In particular, $\pi'\leq \pi$.
\end{lem}
The planar tangles $(T_{\pi})_{\pi\in NC(k)}$ yield a decomposition of connected planar tangles of degree $k$. Let $T$ be a connected planar tangle of degree $k$. Let $B_{1},\dots,B_{r}$ be the blocks of $\pi_{T}$ ordered lexicographically and let $C_{1},\dots,C_{r}$ be the corresponding connected components of $T$. For $1\leq i\leq r$, $T_{i}$ is defined as the planar tangle $T\setminus (\bigcup_{j\not=i}C_{j})$, where the distinguished points of the outer boundary of $T$ which are in $B_{i}$ have been counterclockwise relabeled in such a way that the first odd point is labeled $1$. The planar tangle $T_{1}$ of the connected planar tangle of Figure \ref{Fig2ConnectIrrTangle} is depicted in Figure \ref{Fig4ConnectCompo}. \\
 \begin{figure}[h!]
\begin{center}
\scalebox{0.7}{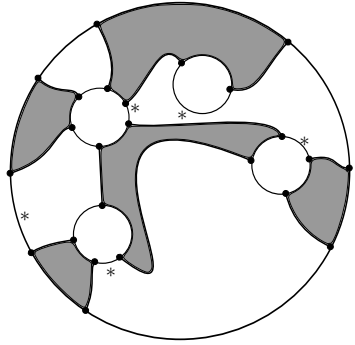}
\end{center}
\caption{\label{Fig4ConnectCompo}: First connected component of the first planar tangle of Figure \ref{Fig2ConnectIrrTangle}.}
\end{figure}\\
\begin{prop}\label{irredFactoriz}
Let $T$ be a connected planar tangle, and set $\pi=\pi_{T}$. Then
$$T=T_{\pi}\circ_{D_{B_{1}},\dots,D_{B_{r}}}(T_{1},\dots,T_{r}).$$
\end{prop}
\begin{proof}
It is possible to draw $r$ disjoint Jordan curves $\lbrace\rho_{i} \rbrace _{1\leq i\leq r}$ such that $\rho_{i}$ intersects $T$ $\vert B_{i}\vert $ times, once at each curve of $C_{i}$ having an endpoint on the outer disk (or two times at a curve joining two distinguished points of the outer boundary). The intersection points are labeled counterclockwise around $\rho_{i}$, in such a way that the intersection point with the curve coming from the first odd point of $B_{i}$ is labeled $1$. \\
Let $\Gamma_{i}$ be the bounded region whose boundary is $\rho_{i}$ and set $\tilde{T}_{i}=(T\cap \Gamma_{i})\cup\rho_{i}$, with the labeling of the distinguished points of $\rho_{i}$ given above. Then $\tilde{T}_{i}$ is a planar tangle which is an isotopy of $T_{i}$. Figure \ref{Fig5ChoiceJordan} shows a possible choice of Jordan curves for the connected planar tangle of Figure \ref{Fig2ConnectIrrTangle}.\\
 \begin{figure}[h!]
\begin{center}
\scalebox{0.7}{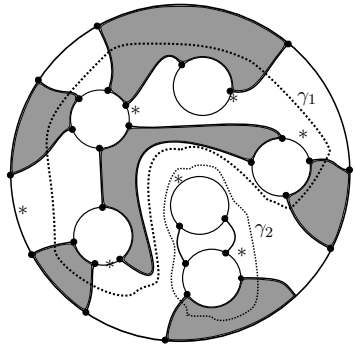}
\end{center}
\caption{\label{Fig5ChoiceJordan}: Jordan curves surrounding the connected components of a planar tangle.}
\end{figure}
Let $\tilde{T}$ be the planar tangle whose inner disks are $(\Gamma_{i})_{1\leq i\leq r}$, with the distinguished points being the ones of $\rho_{i}$, and whose skeleton is $\Gamma T\setminus \left(\Gamma T\cap(\bigcup \mathring{\Gamma}_{i})\right)$. Then $\tilde{T}$ is isotopic to $T_{\pi}$ and, by construction, $\tilde{T}\circ_{\Gamma_{1},\dots,\Gamma_{r}}(\tilde{T}_{1},\dots,\tilde{T}_{r})=T.$
\end{proof}
\subsubsection{Shaded regions and Kreweras complement} \label{Krewerassec}
If $S$ is a set of cardinality $k$, $f:\llbracket 1,k\rrbracket\longrightarrow S$ is a bijective function and $\pi\in P(k)$, then $f(\pi)$ is the partition of $S$ defined by $f(i)\sim_{f(\pi)}f(j)$ if and only if $i\sim_{\pi}j$. 
For $i\geq 1$, set $\delta(i)=1$ if $i$ is odd and $0$ else. A partial partition $(\tilde{\pi},S)$ of $4k$ is associated to each $\pi\in NC(2k)$ as follows.
\begin{itemize}
\item $S=\lbrace 2i-\delta(i)\rbrace_{1\leq i\leq 2k}$.
\item $\tilde{\pi}=f(\pi)$ where $f:\llbracket 1,2k\rrbracket\longrightarrow \llbracket 1,4k\rrbracket$ given by $f(i)=2i-\delta(i)$.
\end{itemize}
Note that $S$ is the set $\lbrace 1,4,5,8,\dots,4k-3,4k\rbrace$. Let $\tilde{f}$ be the map from $\llbracket 1,2k\rrbracket$ to $\llbracket 1,4k\rrbracket\setminus S$ defined by $\tilde{f}(i)=2i-(1-\delta(i))$.
\begin{defi}
Let $\pi\in NC(2k)$. The nested Kreweras complement of $\pi$, denoted by $\kr'(\pi)$, is the partition of $2k$ such that $\tilde{f}(\kr'(\pi))=\kr(\tilde{\pi},S)$.
\end{defi}
The nested Kreweras complement of the partition $\lbrace \lbrace 1,3,4\rbrace,\lbrace 2\rbrace,\lbrace 5,6\rbrace\rbrace$ is given in Figure \ref{Figure6Kreweras}.\\
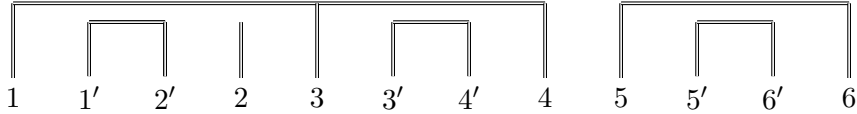
\begin{figure}[h!]
\centering
\begin{tikzpicture}
\node (1) at (1,0){$1$};
\node (2) at (2,0){$1'$};
\node (3) at (3,0){$2'$};
\node (4) at (4,0){$2$};
\node (5) at (5,0){$3$};
\node (6) at (6,0){$3'$};
\node (7) at (7,0){$4'$};
\node (8) at (8,0){$4$};
\node (9) at (9,0){$5$};
\node (10) at (10,0){$5'$};
\node (11) at (11,0){$6'$};
\node (12) at (12,0){$6$};

\draw [double]  (1) -- (1,1.25);
\draw [double] (1,1.25)--(8,1.25);
\draw [double] (5,1.25)--(5);
\draw [double] (8,1.25)--(8);

\draw [double] (4,1)--(4);

\draw [double] (2,1)--(3,1);
\draw [double] (3,1)--(3);

\draw [double] (9)--(9,1.25);
\draw [double] (9,1.25)--(12,1.25);
\draw [double] (12,1.25)--(12);

\draw [double] (2)--(2,1);
\draw [double] (2,1)--(3,1);
\draw [double] (3,1)--(3);

\draw [double] (6)--(6,1);
\draw [double] (6,1)--(7,1);
\draw [double] (7,1)--(7);

\draw [double] (10)--(10,1);
\draw [double] (10,1)--(11,1);
\draw [double] (11,1)--(11);

\end{tikzpicture}
\caption{\label{Figure6Kreweras} The partition $\lbrace \lbrace 1,3,4\rbrace,\lbrace 2\rbrace,\lbrace 5,6\rbrace\rbrace$ and its nested Kreweras complement $\lbrace \lbrace 1,2\rbrace,\lbrace 3,4\rbrace,\lbrace 5,6\rbrace\rbrace$.}
\end{figure}\\
Contrary to the usual Kreweras complement, the nested Kreweras complement is not bijective, as we will see in the next lemma. Let $\pi_{0}$ (resp $\pi_{1}$) be the partition of $2k$ with $k$ blocks being $\{2i,2i+1\}$ (resp. $\{2i-1,2i\}$) for $1\leq i\leq k$.
\begin{lem}\label{krewerPio}
We have $\kr'(\pi)=\kr'(\pi\vee\pi_{0})$.
\end{lem}
\begin{proof}
Since $\pi\leq(\pi\vee\pi_{0})$, it follows that $\kr'(\pi\vee\pi_{0})\leq \kr'(\pi)$. \\
Suppose that $i\sim_{\kr'(\pi)}j$. Then $2i-(1-\delta(i))\sim_{\kr(\tilde{\pi},S^{c})}2j-(1-\delta(j))$. By Lemma \ref{characKreweras}, this implies that for all $k\in\llbracket 2i-(1-\delta(i)),2j-(1-\delta(j))\rrbracket\cap S$, $l\in S\setminus\llbracket 2i-(1-\delta(i)),2j-(1-\delta(j))\rrbracket$, $k\not\sim_{\pi}l$. If $k\in \llbracket 2i-(1-\delta(i)),2j-(1-\delta(j))\rrbracket\cap S$ and $l\not\in\llbracket 2i-(1-\delta(i)),2j-(1-\delta(j))\rrbracket$, then $l\not=k\pm1$ and thus  $k\not\sim_{\pi_{0}}l$; therefore, for all $k\in\llbracket 2i-(1-\delta(i)),2j-(1-\delta(j))\rrbracket\cap S$, $l\in S\setminus\llbracket 2i-(1-\delta(i)),2j-(1-\delta(j))\rrbracket$, $k\not\sim_{\pi\vee \pi_{0}}l$. By Lemma \ref{characKreweras}, this implies that $i\sim_{\kr'(\pi\vee\pi_{0})}j$ and therefore $\kr'(\pi)\leq kr'(\pi\vee\pi_{0})$.
\end{proof}
The nested Kreweras complement is involved in the description of planar tangles in the following way.
\begin{prop}\label{kewerasShaded}
Suppose that $T$ is a planar tangle of degree $k$. Then, $i$ and $j$ are in the same block of $\kr'(\pi_{T})$ if and only if $i_{D_{0}}$ and $j_{D_{0}}$ are boundary points of the same shaded region.
\end{prop}
\begin{proof}
Let $T$ be a planar tangle. Relabel each distinguished point $i_{D_{0}}$ with $2i-\delta(i)$.With this new labelling each interval of type $(4i+1,4i+4)$ is the boundary interval of a shaded region, and $\pi_T$ becomes a partial non-crossing partition of $4k$ with support $S$.  Add two points $4i+2$ and $4i+3$ clockwise on each segment $(4i+1,4i+4)$ of the outer disk. We define the equivalence relation $\sim$ on $S^{c}$ as follows: $i\sim j$ if and only if the boundary points $i$ and $j$ are boundary points of a same shaded region. Let $(\pi',S^{c})$ be the partial partition associated to $\sim$. $\pi'$ is non-crossing since two regions that intersect are the same. \\
Let $\pi=\left((\pi_{T},S)\vee (\pi',S^{c})\right)$. Suppose that $1\leq i<j<r<s\leq 4k$ are such that $i\sim_{\pi}r$ and $j\sim_{\pi}s$. Since $\pi_{T}$ is non-crossing, if $i,j,r,s$ are all in $S$ then $i\sim_{\pi}j\sim_{\pi}r\sim_{\pi}s$. Assume from now on that they are not all in $S$, and suppose without loss of generality that $i\in S^{c}$. We have $i\sim_{\pi}r$,  thus $r$ is also in $S^{c}$ and $i$ and $r$ are boundary points of a same shaded region $\sigma$. Since $j\in(i,r)$ and $s\in(r,i)$, any path on $\Gamma T$ between $j$ and $s$ would cut $\sigma$ in two distinct regions. Thus, if $j,s\in S$, then $j\not\sim_{\pi}s$. Therefore, the hypothesis $j\sim_{\pi}s$ yields that $j,s\in S^{c}$. As $\pi'$ is non-crossing, $i,j,r,s$ are in the same block of $\pi$. Finally, $\pi_{T}\vee\pi'$ is non-crossing, which yields that $\pi'\leq kr'(\pi_{T})$.\\
Let $\pi_{2}$ be a partial partition with support $S^{c}$, such that $\pi_{T}\vee \pi_{2}$ is non-crossing. Suppose that $i\sim_{\pi_{2}}j$, with $i,j\in S^{c}$. Let $\sigma_{i}$ (resp. $\sigma_{j}$) be the shaded region having $i$ (resp. $j$) as boundary point. $\pi_{T}\vee \pi_{2}$ is non-crossing, thus for all $r,s\in S$ such that $i\leq r\leq j$ and $j\leq s\leq i$, $r\not\sim_{\pi_{T}}s$. Thus, there is no path in $\Gamma T$ between $(i,j)$ and $(j,i)$, and $\sigma_{i}=\sigma_{j}$. This yields $i\sim_{\pi'}j$. Therefore, $(\pi_{2},S^{c})\leq (\pi',S^{c})$ and $(\pi',S^{c})=kr(\pi_{T},S)$.\\
Let $1\leq i,j\leq 2k$. By the two previous paragraphs, $i$ and $j$ are in the same block of $\kr'(\pi_{T})$ if and only if $2i-(1-\delta(i))$ and $2j-(1-\delta(j))$ are boundary points of the same shaded region. Since the points $2i-(1-\delta(i))$ and $2i-\delta(i)$ both belong to the interval $(2(i+\delta(i))-3,2(i+\delta(i)))$ (which is part of the boundary of a shaded region), $2i-(1-\delta(i))$ and $2j-(1-\delta(j))$ are boundary points of the same shaded region if and only if $2i-\delta(i)$ and $2j-\delta(j)$ are boundary points of the same shaded region which yields the result.
\end{proof}

\section{Tensor products and free products of planar algebras} \label{freeproductsection}

\begin{defi}
Let $\cP$ and $\cQ$ be two planar algebras. The tensor product planar algebra $\mathcal{P}\otimes \mathcal{Q}$ is the collection of vector spaces $(\mathcal{P}\otimes \mathcal{Q})_{i}=\mathcal{P}_{i}\otimes \mathcal{Q}_{i}$, with the action of any planar tangle being given by the tensor product of the action on each component.
More precisely, for a planar tangle $T$, 
$$Z_{T}(\bigotimes_{D_{i}}(x_{i}\otimes y_{i}))=Z_{T}(\bigotimes_{D_{i}}x_{i})\otimes Z_{T}(\bigotimes_{D_{i}}y_{i}),$$
where $x_{i}\in \mathcal{P}_{k_{i}}$ and $y_{i}\in \mathcal{Q}_{k_{i}}$.
\end{defi}

\begin{lem}
Let $A,B$ be finite-dimensional $C^*$-algebras. Then
\[ \cP^{\Gamma(A)} \ot \cP^{\Gamma(B)} \cong \cP^{\Gamma(A \ot B)}. \]
\end{lem}

\begin{proof}
Let $\Gamma(A) = (V_1, E_1)$ and $\Gamma(B) = (V_2, E_2)$ be the bipartite graphs associated to $A$ and $B$. Note first that the bipartite graph $\Gamma(A\ot B) = (V,E)$ is given by gluing $\Gamma(B)$ (at the even vertex of $\Gamma(B)$) onto the uneven endpoints of every edge of $\Gamma(A)$ and deleting the glued vertices. Therefore, $\Gamma(A\ot B)$ can be identified with the direct product $\Gamma(A) \times \Gamma(B)$, that is to say, $V_+ = V_{1,+} \times V_{2,+}, \ V_- = V_{1,-} \times V_{2,-}$ and $E = E_1 \times E_2$. As the Markov trace of $A \ot B$ is the tensor product of the Markov traces of the components, the spin vector $\mu_{A \ot B}: V \to \C$ is given by 
\[ \mu_{A\ot B}((x,y)) \ = \ \mu_A(x) \mu_B(y), \qquad x \in V_1, \ y \in V_2. \]
For any $n \geq 0$, the map
\begin{align*}
\Phi_n: \cP^{\Gamma(A)}_n \ot \cP^{\Gamma(B)}_n \quad &\to \quad \cP^{\Gamma(A \ot B)}_n \\
(\beta_1,\dots,\beta_{2n}) \ot (\gamma_1,\dots,\gamma_{2n}) \quad &\mapsto \quad ((\beta_1,\gamma_1),\dots, (\beta_{2n},\gamma_{2n}))
\end{align*}
is a linear isomorphism of vector spaces. If $n = -$, the corresponding isomorphism $\Phi_-$ is given by the identification $V_- = V_{1,-} \times V_{2,-}$. We have to show that the isomorphism $\Phi =(\Phi_n)_{n \in \N_* \cup \{ +,- \}}$ of graded vector spaces commutes with the action of planar tangles. Let $T$ be such a planar tangle and let $\sigma$ be a state on $T$ for $\Gamma(A\ot B)$. Clearly, we can decompose $\sigma$ as $\sigma_1 \times \sigma_2$ where $\sigma_1$ is a state on $T$ for $\Gamma(A)$ and $\sigma_2$ is a state on $T$ for $\Gamma(B)$. Note further that $\sigma$ is compatible with the loop $((\beta_1,\gamma_1),\dots, (\beta_{2n},\gamma_{2n}))$ at a disk $D$ if and only if $\sigma_1$ is compatible with $(\beta_1,\dots,\beta_{2n})$ at $D$ and $\sigma_2$ is compatible with $(\gamma_1,\dots,\gamma_{2n}) $ at $D$. This implies that, if $T$ has outer disk $D_0$ and inner disks $D_1, \dots, D_s$ of degrees $k_i$ and if $\xi_i \in \cP^{\Gamma(A)}_{k_i}, \ \eta_i \in \cP^{\Gamma(B)}_{k_i}, \ i=0,\dots,s$ are basis loops, the coefficients of the action of $T$ satisfy
\[ c((\xi_0,\eta_0),\dots,(\xi_s,\eta_s)) = c(\xi_0,\dots,\xi_s) c(\eta_0,\dots, \eta_s).  \]
Therefore, $\Phi$ commutes with the action of $T$.
\end{proof}

The free product of two planar algebras $\mathcal{P}$ and $\mathcal{Q}$ is a planar subalgebra of $\mathcal{P}\otimes \mathcal{Q}$ defined by the image of certain planar tangles. For planar algebras the free product operation has appeared first in \cite{J99}. On the level of subfactors, however, the construction goes back to Bisch and Jones's celebrated article \cite{BJ95}.  \\
A pair $(T,T')$ of planar tangles of degree $k$ is called free if there exists a planar tangle $R$ of degree $2k$ and two isotopies $\phi_{1}$ and $\phi_{2}$, respectively of $T$ and $T'$, such that
\begin{itemize}
\item $\Gamma R=\Gamma\phi_{1}(T)\cup\Gamma\phi_{2}(T')$, and the set of distinguished points of $R$ is the image through $\phi_{1}$ and $\phi_{2}$ of the set of distinguished points of $T$ and $T'$.
\item $\phi_{1}(T\setminus D_{0}(T))\cap \phi_{2}(T'\setminus D_{0}(T'))=\emptyset$. This means that a connected component of $R$ is the image of a connected component of either $T$ or $T'$.
\item The distinguished point numbered $i$ of $\partial D_{0}(T)$ is sent by $\phi_{1}$ to the distinguished point numbered $2i-\delta(i)$ of $\partial D_{0}(R)$.
\item  The distinguished point numbered $i$ of $\partial D_{0}(T')$ is sent by $\phi_{2}$ to the distinguished point numbered $2i-(1-\delta(i))$, where $\delta(i) = i \mod 2$.
\item A distinguished point of an inner disk coming from $T$ is labelled as in $T$; a distinguished point $i$ of an inner disk coming from $T'$ is labelled $i-1$.
\end{itemize}
The last condition ensures that curves of $R$ have endpoints with correct parities. If $T$ and $T'$ are connected planar tangles and $R$ exists, then $R$ is unique up to isotopy. This planar tangle is called the \emph{free composition} of $T$ and $T'$ and denoted by $T*T'$. An example of a free pair of planar tangles, with the resulting free composition, is drawn in Figure \ref{Fig13FreeCompo}.\\
 \begin{figure}[h!]
\begin{center}
\scalebox{0.7}{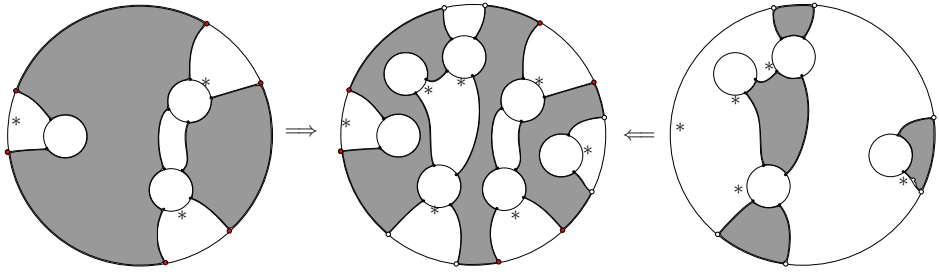}
\end{center}
\caption{\label{Fig13FreeCompo}: Free composition of two planar tangles.}
\end{figure}\\

\begin{defi}\label{freeProduct}
Let $\mathcal{P}$ and $\mathcal{Q}$ be two planar algebras. The free product planar algebra $\mathcal{P}*\mathcal{Q}$ is the collection of vector subspaces $(\mathcal{P}*\mathcal{Q})_{k}$ of $(\mathcal{P}\otimes\mathcal{Q})_{k}$ spanned by the image of the maps $Z_{T}\otimes Z_{T'}$ for all free pairs of planar tangles of degree $k$.
\end{defi}

The following result is certainly well known within the planar algebra community. Since we could not find a precise reference, we include a proof.

\begin{lem}
$\mathcal{P}*\mathcal{Q}$ is a planar subalgebra of $\mathcal{P}\otimes\mathcal{Q}$, that is to say $\cP* \cQ$ is stable under the action of planar tangles.
\end{lem}
\begin{proof}
It suffices to check the stability on the generating sets of the vector spaces $(\mathcal{P}*\mathcal{Q})_{k}$ given in Definition \ref{freeProduct}. Let $T$ be a planar tangle, and for each inner disk $D_{i}$ of $T$, let $v_{i}$ be an element of $(\mathcal{P}*\mathcal{Q})_{k_{i}}$ of the form $(Z_{T_{i}}\otimes Z_{T'_{i}})((\bigotimes_{D_{j}(T_{i})} v^{i}_{j})\otimes (\bigotimes_{D_{j}(T'_{i})} w^{i}_{j})$ where $(T_{i},T'_{i})$ is a free pair for any $i$. The compatibility condition on the composition of actions of planar tangles yields
\begin{align*}
Z_{T}(\bigotimes_{D_{i}(T)}v_{i})=&Z_{T\circ_{(D_{1},\dots D_{n})}(T_{1},\dots,T_{n})}(\bigotimes_{D_{i}(T)} \bigotimes_{D_{j}(T_{i})} v_{j}^{i})\\
&\otimes Z_{T\circ_{(D_{1},\dots D_{n})}(T'_{1},\dots,T'_{n})}(\bigotimes_{D_{i}(T)} \bigotimes_{D_{j}(T'_{i})} w_{j}^{i}).\\
\end{align*}
Thus, it is enough to prove that the tangles $S=T\circ_{(D_{1},\dots D_{n})}(T_{1},\dots,T_{n})$ and $S'=T\circ_{(D_{1},\dots D_{n})}(T'_{1},\dots,T'_{n})$ form a free pair. Let $\tilde{T}$ be the planar tangle of order $2k$ obtained from $T$ by doubling all the curves of $T$ and all the distinguished points (in such a way that the tangle still remains planar). By construction, a curve joining the point $j$ of $D_{i}$ to the point $j'$ of $D_{i'}$ in $T$ yields two curves in $\tilde{T}$: one joining the point $2j-1$ of $D_{i}$ to the point $2j'-1$ of $D_{i'}$ and the other one joining the point $2j$ of $D_{i}$ to the point $2j'$ of $D_{i'}$. Since $T$ is a planar tangle, the conditions on the parities of $j$ and $j'$ yield that in $\tilde{T}$, the curves join points labelled $0$ or $1$ modulo $4$ (resp. $2$ or $3\mod 4$) to points labelled $0$ or $1$ modulo $4$ (resp. $2$ or $3 \mod 4$). Reciprocally, by removing all the distinguished points labelled $0$ and $1$ modulo $4$ and the curves joining them from $\tilde{T}$, we recover the planar tangle $T$. The same holds for the distinguished points labelled $2$ and $3$.
Therefore, if we compose the tangle $T_{i}*T'_{i}$ inside each disk $D_{i}$, the resulting tangle is exactly $S*S'$ (after relabelling). Thus, $\tilde{T}\circ_{D_{1},\dots, D_{n}}(T_{1}*T'_{1},\dots,T_{n}*T'_{n})=S*S'$ and $\mathcal{P}*\mathcal{Q}$ is stable under the action of planar tangles.\\
\end{proof}
The remaining part of this section aims at simplifying the description of a free product of planar algebras.

\subsection{Reduced free pairs}
We now introduce a set of pairs of planar tangles which is smaller than the set of free pairs and whose image in $\cP\otimes \cQ$ still spans $\cP*\cQ$. Let us begin by characterizing free pairs of planar tangles in terms of their associated non-crossing partitions. In this subsection, all planar tangles are assumed to be connected.
\begin{lem}\label{freePairIrredu}
If $(T,T')$ is a free pair, then $\pi_{T* T'} =(\pi_{T},S)\vee(\pi_{T'},S^{c})$.\\
In particular $(T,T')$ is a free pair if and only if $(T_{\pi_{T}},T_{\pi_{T'}})$ is a free pair. 
\end{lem}
\begin{proof}
The first statement of the lemma is a direct consequence of the definition of $S$ and the fact that $i\sim_{T}j$ if and only if $2i-\delta(i)\sim_{T*T'}2j-\delta(j)$ and $i\sim_{T'}j$ if and only if $2i-(1-\delta(i))\sim_{T*T'} 2j-(1-\delta(j)$. Thus, if $(T,T')$ is a free pair, then $T_{\pi(T*T')}$ is exactly the free composition of $T_{\pi_{T}}$ with $T_{\pi_{T'}}$ and $(T_{\pi_{T}},T_{\pi_{T'}})$ is also a free pair.\\
Suppose that $(T_{\pi_{T}},T_{\pi_{T'}})$ is a free pair. By Proposition \ref{irredFactoriz}, there exist $T_{1},\dots,T_{r},T'_{1},\dots,T'_{r'}$ such that $T=T_{\pi_{T}}\circ_{D_{1},\dots,D_{R}}(T_{1},\dots,T_{r})$ and $T'=P_{\pi_{T'}}\circ_{D'_{1},\dots,D'_{r'}}(T'_{1},\dots,T'_{r'})$. Therefore, $(T_{\pi_{T}}*T_{\pi_{T'}})\circ_{D_{1},\dots,D_{r},D'_{1},\dots,D'_{r'}}(T_{1},\dots,T_{r},T'_{1},\dots,T'_{r'})$ gives the free composition of $T$ and $T'$.
\end{proof}
\begin{prop}\label{freePairKreweras}
Let $T$ and $T'$ be two connected planar tangles. Then $(T,T')$ is a free pair if and only if $\pi_{T'}\leq \kr'(\pi_{T})$. In particular if $T,U$ and $T',U'$ satisfy $\pi_{T}=\pi_{U}$ and $\pi_{T'}=\pi_{U'}$, then $(T,T')$ is a free pair if  and only if $(U,U')$ is a free pair.
\end{prop}
\begin{proof}
If $(T,T')$ is a free pair, then by the previous lemma $(\pi_{T},S)\vee(\pi_{T'},S^{c})$ is non-crossing. Therefore $\pi_{T'}\leq kr'(\pi_{T})$.\\
If $\pi_{T'}\leq kr'(\pi_{T})$, then $\tilde{\pi}=(\pi_{T},S)\vee(\pi_{T'},S^{c}) $ is even and non-crossing. Therefore, $T_{\tilde{\pi}}$ is a well-defined planar tangle. Let $\lbrace C_{i}\rbrace$ be the connected components of $T_{\tilde{\pi}}$ coming from blocks of $(\pi_{T},S)$ and $\lbrace D_{i}\rbrace$ the ones coming from blocks of $(\pi_{T'},S^{c})$. Then, $T_{\tilde{\pi}}\setminus \bigcup C_{i}$ is an irreducible planar tangle and $\pi(T_{\tilde{\pi}}\setminus \bigcup C_{i})=\pi_{T'}$. Thus, $T_{\tilde{\pi}}\setminus \bigcup C_{i}=T_{\pi_{T'}}$ up to a relabelling of the distinguished points, and similarly $T_{\tilde{\pi}}\setminus \bigcup D_{i}=T_{\pi_{T}}$ up to relabelling.  Therefore, $T_{\tilde{\pi}}$ is, up to a relabelling, the free composition of $T_{\pi_{T}}$ with $T_{\pi_{T'}}$. Finally, $(T_{\pi_{T}},T_{\pi_{T'}})$ is a free pair and by the previous lemma, $(T,T')$ is also a free pair.
\end{proof}
Recall that $\pi_{0}$ (resp. $\pi_{1}$) is the pair partition of $2k$ with blocks $\lbrace (2i,2i+1)\rbrace$ (resp. $\lbrace 2i+1,2i+2)\rbrace$).
\begin{defi}
A free pair $(T,T')$ of planar tangles is called reduced if $T=T_{\pi}$ and $T'=T_{kr'(\pi)}$ for some non-crossing partition $\pi$ such that $\pi\geq \pi_{0}$.
\end{defi}
 An example of reduced free pair is given in Figure \ref{Fig14ReducedFreePair} with $\pi=\lbrace \lbrace 1,6\rbrace,\lbrace 2,3,4,5\rbrace\rbrace$.
 \begin{figure}[h!]
\begin{center}
\scalebox{0.7}{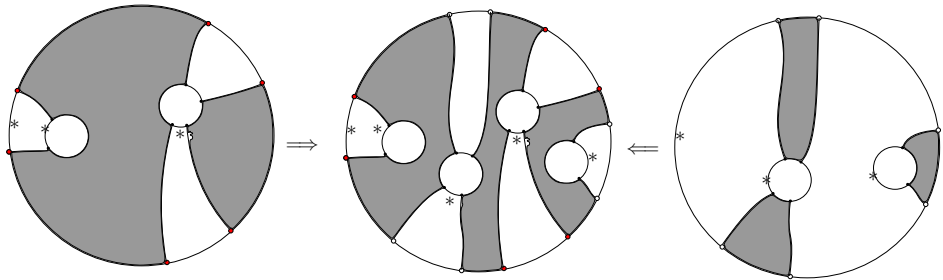}
\end{center}
\caption{\label{Fig14ReducedFreePair}: Example of reduced free pair.}
\end{figure}\\
Note that the number of reduced free pairs of degree $k$ is exactly the cardinality of non-crossing partition of $k$. Despite the small number of reduced free pairs, only considering these pairs is nonetheless enough to describe free product of planar algebras. 
\begin{prop}\label{spinReducedPairs}
The free planar algebra $\mathcal{P}\otimes \mathcal{Q}$ is spanned by the images of $Z_{T}\otimes Z_{T'}$ for all reduced free pairs $(T,T')$.
\end{prop}
\begin{proof}
It suffices to prove that the image of $Z_{T}\otimes Z_{T'}$ with $(T,T')$ a free pair is contained in the image of $(U,U')$ with $(U,U')$ a reduced free pair. \\
The image of $Z_{T}\otimes Z_{T'}$ is contained in the image of $Z_{T_{\pi_{T}}}\otimes Z_{T_{\pi_{T'}}}$ by Proposition \ref{irredFactoriz} and by Lemma \ref{freePairIrredu} $(T_{\pi_{T}},T_{\pi_{T'}})$ is again a free pair. We can therefore assume that $T=T_{\pi}$ and $T=T_{\pi'}$, with the condition $\pi'\leq kr'(\pi)$ being given by Proposition \ref{freePairKreweras}.\\
Suppose that $\mu\leq \nu$ are two noncrossing partitions of $k$. Let $B_{1},\dots,B_{r}$ be the blocks of $\nu$ in the lexicographical order. Since $\mu\leq \nu$, $\mu=\bigvee (\mu_{\vert B_{i}},B_{i})$. Therefore, 
$$T_{\mu}=T_{ \nu}\circ_{D_{1},\dots,D_{r}}(T_{\mu_{\vert B_{1}}},\dots,T_{\mu_{\vert B_{r}}}),$$
and the image of $Z_{T_{\mu}}$ is contained in the one of $Z_{T_{ \nu}}$. \\
Since $kr'(\pi)=kr'(\pi_{0}\vee\pi)$, $\pi'\leq kr'(\pi_{0}\vee\pi)$. $\pi\leq \pi\vee\pi_{0}$ and $\pi'\leq kr'(\pi\vee \pi_{0})$, thus the image of $Z_{T_{\pi}}$ is included in the image of $Z_{T_{\pi\vee\pi_{0}}}$ and the image of $Z_{T_{\pi'}}$ is included in the image of $Z_{T_{kr'(\pi\vee\pi_{0})}}$. By definition, $(T_{\pi\vee\pi_{0}},T_{kr'(\pi_{0}\vee\pi)})$ is reduced.
\end{proof}
Thanks to the previous proposition, there is a simpler way to describe the vector space $(\mathcal{P}*\mathcal{Q})_{n}$ for $n\geq 1$. Let $\pi = \lbrace B_{1},\dots,B_{r}\rbrace$ be an even non-crossing partition, such that its blocks are ordered lexicographically. For any planar algebra $\mathcal{P}$, we define the space $\mathcal{P}_{\pi}$ as the vector space $\mathcal{P}_{\vert B_{1}\vert/2 }\otimes \dots\mathcal{P}_{\vert B_{r}\vert/2}$. Proposition \ref{spinReducedPairs} then precisely translates to the statement that $(\mathcal{P}*\mathcal{Q})_{n}$ is spanned by 
$$\lbrace Z_{T_{\pi}}(v)\otimes Z_{T_{kr'(\pi)}}(w)\vert v\in \mathcal{P}_{\pi},w\in\mathcal{Q}_{kr'(\pi)},\pi\in NC(2n),\pi\geq \pi_{0}\rbrace.$$

\subsection{A generating subset of a free product of planar algebras}
Let $\cP$ be a planar algebra, and fix a particular subset $X_{n}$ of $\cP_{n}$ for each $n\in \mathbb{N}_{*}\cup\{+,-\}$. We say that $(X_{n})_{n\in\N_{*}\cup\{+,-\}}$ generates $\cP$ (or that $(X_{n})_{n\in\N_{*}\cup\{+,-\}}$ is a generating subset of $\cP$) if each $\cP_{n}$ with $n\in\N_{*}\cup\{+,-\}$ is spanned by the union of all images $Z_{T}(X_{k_{1}},\dots,X_{k_{m}})$, where $T$ is any planar tangle of degree $n$ and $k_{1},\dots,k_{m}$ are the respective degrees of the inner disks of $T$.\\ 
The goal of this subsection is to introduce a simple generating subset of the free product $\cP*\cQ$. For this purpose let us introduce the planar tangle $S_{k}$ (resp. $U_{k}$) which is the tangle without inner disk order $k$ where $2i-1$ is linked to $2i$ (resp. $2i$ is linked to $2i+1$) for all $1\leq i\leq k$. Both tangles are drawn in Figure \ref{Fig15SpecialTangles} for $k=4$. We simply denote by $S_{\mathcal{P}}(k)$ (resp. $U_{\mathcal{P}}(k)$) the image of the element $Z_{S_{k}}$ in $\mathcal{P}_{k}$ (resp. $Z_{U_{k}}$).
\begin{figure}[h!]
\begin{center}
\scalebox{0.7}{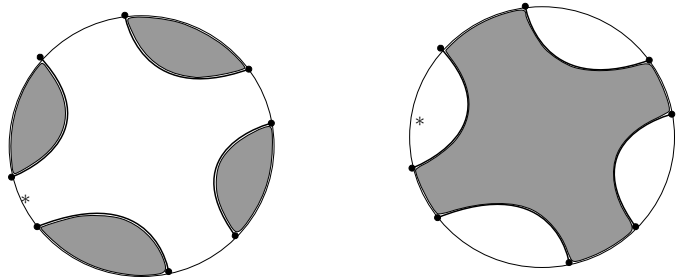}
\end{center}
\caption{\label{Fig15SpecialTangles}: Tangles $S_{4}$ and $U_{4}$.}
\end{figure}
\begin{lem}\label{constReducedSU}
Let $(T,T')$ be a reduced free pair of degree $k$. There exists a planar tangle $R$ of degree $k$ with $r$ inner disks $D_{i}$ of respective degree $k_{i}$, such that $T=R_{D_{1},\dots,D_{r}}(X_{1},\dots,X_{r})$ and $T'=R_{D_{1},\dots,D_{r}}(\tilde{X}_{1},\dots,\tilde{X}_{r})$, where for each $1\leq i\leq r$, $(X_{i},\tilde{X}_{i})$ is either $(U_{k_{i}},Id_{k_{i}})$ or $(Id_{k_{i}},S_{k_{i}})$.
\end{lem}
\begin{proof}
Since $(T,T')$ is a free pair, the free composition $T*T'$ exists; since this pair is reduced, each inner disk of $T*T'$ is only connected to the outer boundary, and, by definition, $\pi_{T}\geq \pi_{0}$ and $\pi_{T'}\geq \pi_{1}$. Thus, for each $1\leq i\leq k$, both elements of $\lbrace 4i,4i+1\rbrace$ (resp. $\lbrace 4i-2,4i-1\rbrace$) are connected to the same inner disk coming from $T$ (resp. $T'$). Color an inner disk $D_{i}$ of $T*T'$ with $1$ if it comes from $T$ and with $2$ if it comes from $T'$. We denote by $\gamma_{i}$ the curve arriving on the boundary point $i$ of the outer boundary and by $\bar{i}$ the boundary point of an inner disk which is connected to $i$. \\
We use the following operation on $T*T'$. For each interval $(4i-1,4i)$, let $\sigma$ be the region having $(4i-1,4i)$ as a boundary interval. Add a curve $\tilde{\gamma}$ within this region connecting $\overline{4i-1}$ to $\overline{4i}$ and erase $\gamma_{4i-1},\gamma_{4i}$ and the boundary points $4i-1$ and $4i$. \\
The degrees of the inner disks do not change and this yields a planar tangle $R$ with $2k$ boundary points and $r$ inner disks (where $r$ is the sum of the number of inner disks in $T$ and in $T'$).  In the resulting planar tangle $R$, an odd point $i$ of the outer boundary is still connected to the point $\bar{i}$ on a disk colored $1$, and an even point $i$ of the outer boundary is still connected to the point $\bar{i}$ on a disk colored $2$. The construction of the tangle $R$ is shown in Figure \ref{Fig16ConstructionTangleR}. \\
\begin{figure}[h!]
\begin{center}
\scalebox{0.9}{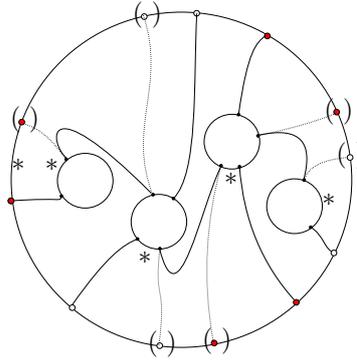}
\end{center}
\caption{\label{Fig16ConstructionTangleR}: Construction of the planar tangle $R$ for the reduced free pair of Figure \ref{Fig14ReducedFreePair}.}
\end{figure}\\
Set $X_{i}=U_{k_{i}},\tilde{X}_{i}=Id_{k_{i}}$ if $D_{i}$ is colored $2$, and $X_{i}=Id_{k_{i}},\tilde{X}_{i}=S_{k_{i}}$ if $D_{i}$ is colored $1$. Consider $R_{1}=R_{D_{1},\dots,D_{r}}(X_{1},\dots,X_{r})$. Each disk of $R$ colored $2$ is replaced by a planar tangle without inner disk, and thus disappears in $R_{1}$. A disk of $R$ colored  $1$ is composed with the identity, and thus remains the same in $R_{1}$. An odd point $4i+1$ is already connected to $\overline{4i+1}$. An even point $4i+2$ is connected to an odd point $\overline{4i+2}$ of a disk $D$ colored $2$. Therefore, since $D$ is composed with $U_{k_{i}}$, $\overline{4i+2}$ is connected in $R_{1}$ by a curve to the following point of $D$ in clockwise direction. Since $4i+2$ and $4i+3$ are in the same connected component, the following point is exactly $\overline{4i+3}$. By the modification we made on $P*Q$, $\overline{4i+3}$ is connected by a curve to the point $\overline{4i+4}$. Therefore, in $R_{1}$, $4i+2$ is connected to $\overline{4i+4}$. Thus, relabelling the outer boundary point $4i+1$ by $2i+1$ and $4i+2$ by $2i+2$ in $R_{1}$ yields exactly the image of $T$ in $T*T'$. This reconstruction of $T$ is shown in Figure \ref{Fig17ReconstructionP}.\\
\begin{figure}[h!]
\begin{center}
\scalebox{0.7}{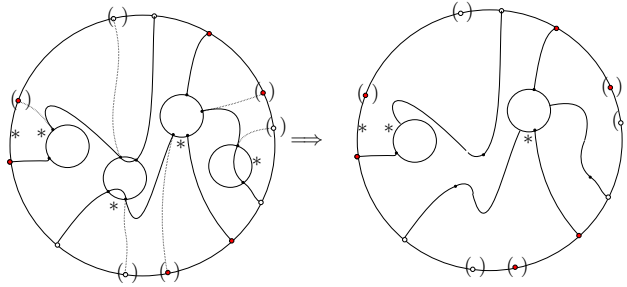}
\end{center}
\caption{\label{Fig17ReconstructionP}: Reconstruction of $T$ from $R$.}
\end{figure}\\
Likewise, $R_{2}=R_{D_{1},\dots,D_{r}}(\tilde{X}_{1},\dots,\tilde{X}_{r})$ is equal to the image of $T'$ in $T*T'$.
\end{proof}
\begin{prop}\label{generatTangle}
Let $\mathcal{P}$ and $\mathcal{Q}$ be two planar algebras. Then, $\lbrace U_{\mathcal{P}}(k)\otimes \mathcal{Q}_{k}\rbrace_{k\geq 1}\cup\lbrace \mathcal{P}_{k}\otimes S_{\mathcal{Q}}(k)\rbrace_{k\geq 1}$ is a generating subset of $\cP*\cQ$.
\end{prop} 
\begin{proof}
First, note that $(U_{k},Id_{k})$ and $(Id_{k},S_{k})$ are two free pairs of planar tangles. Thus, for all $k\geq 1$, $U_{\mathcal{P}}(k)\otimes \mathcal{Q}_{k}$ and $\mathcal{P}_{k}\otimes S_{\mathcal{Q}}(k)$ are subspaces of $(\mathcal{P}*\mathcal{Q})_{k}$. In particular the subplanar algebra generated by $\lbrace U_{\mathcal{P}}(k)\otimes \mathcal{Q}_{k}\rbrace_{k\geq 1}\cup\lbrace \mathcal{P}_{k}\otimes S_{\mathcal{Q}}(k)\rbrace_{k\geq 1}$ is also a subplanar algebra of $\mathcal{P}*\mathcal{Q}$.\\
Reciprocally,  let $(T,T')$ be a reduced free pair. By Lemma \ref{constReducedSU}, there exists a planar tangle $R$ of degree $k$ with $r$ inner disks $D_{i}$ of respective degree $k_{i}$, such that $T=R_{D_{1},\dots,D_{r}}(X_{1},\dots,X_{r})$ and $T'=R_{D_{1},\dots,D_{r}}(\tilde{X}_{1},\dots,\tilde{X}_{r})$, where for each $1\leq i\leq r$, $(X_{i},\tilde{X}_{i})$ is either $(U_{k_{i}},Id_{k_{i}})$ or $(Id_{k_{i}},S_{k_{i}})$. Thus, the image $Z_{T}\otimes Z_{T'}$ of the reduced free pair $(T,T')$ is equal to $Z_{R}(\chi_{1},\dots,\chi_{r})$, where $\chi_{i}=Z_{X_{i}}(\cP)\otimes Z_{\tilde{X}_{i}}(\cQ)$ is either equal to $U_{\cP}(k_{i})\otimes \cQ_{k_{i}}$ or  to $\cP_{k_{i}}\otimes S_{\cQ}(k_{i})$. 
Since $\cP*\cQ$ is spanned by the images $Z_{T}\otimes Z_{T'}$ where $(T,T'  )$ is any reduced free pair, the result follows.
\end{proof}

\begin{rem}
In \cite{L02}, the free product of two planar algebras $\mathcal{P}*\mathcal{Q}$ is directly defined as the subplanar algebra of $\mathcal{P}\otimes\mathcal{Q}$ generated by $\lbrace U_{\mathcal{P}}(k)\otimes \mathcal{Q}_{k}\rbrace_{k\geq 1}$ and $\lbrace \mathcal{P}_{k}\otimes S_{\mathcal{Q}}(k)\rbrace_{k\geq 1}$.
\end{rem}

\section{Free wreath products vs. free products of planar algebras} \label{ProductSection}

The goal of this section is to show that, under the bijective correspondence established in Theorem \ref{main1}, the free wreath product of a compact matrix quantum group $(\G,u)$ 
with an \emph{arbitrary} action on $\C^n$ satisfying the conditions of \ref{main1} (see Subsection \ref{freewreath}) is mapped to the free product of the associated planar algebras. We will also prove that the same results holds true for the free wreath product of $(\G,u)$ with the \emph{universal} action on an arbitrary finite dimensional $C^*$-algebra $A$. As a consequence, our approach in terms of planar algebras gives rise to a free wreath product construction generalizing both previous definitions. \\ \\
Recall from Proposition \ref{wreathaction}, that if $\beta: B \to B \ot \Pol(\G)$ is a faithful, centrally ergodic $\tr_B$-preserving action of $\G$ on a finite dimensional $C^*$-algebra $B$ and if $\al$ is the universal action of $\G_{aut}(A, \tr_A)$ on $A$, the free wreath product $\G \wr_* \G_{aut}(A, \tr_A)$ admits a faithful, centrally ergodic, trace-preserving action $\beta \wr_* \al$ on $A \ot B$. Consider the associated subfactor planar algebras $\cP(\beta) \subset \cP^{\Gamma(B)}, \ \cP(\al) = \rm{TLJ}_{\delta} \subset \cP^{\Gamma(A)} $ with $\delta = \sqrt{\dim A}$ (see Example \ref{TLJ}) and $\cP(\beta \wr_* \al)$.

\begin{thm} \label{main2}
We have 
\[ \cP(\beta \wr_* \al) = \cP(\al) * \cP(\beta) =  \TLJ_{\delta} * \cP(\beta). \]

\end{thm}

In order to prove Theorem \ref{main2}, we will review some facts on the representation theory of $\G_{aut}(A,\tr)$ and $\G \wr_* \G_{aut}(A, \tr_A)$. Recall first that the morphisms between tensor powers of the fundamental representation $u \in B(\cH) \ot \Pol(\G_{aut}(A,\tr))$ of $\G_{aut}(A,\tr)$ are given by
\[ \Mor(u^{\ot k}, u^{\ot l}) = \{ Z_p \ ; \ p \in NC(k,l) \}, \]
where $NC(k,l) $ denotes the set of non-crossing partitions with $k$ upper points and $l$ lower points and where $Z_p: \cH^{\ot k} \to H^{\ot l}$ denotes the linear map given in \cite[Definition 1.11]{FP16}. The 'fattening' isomorphism used in \cite[Proposition 5.2]{LT16} identifies any noncrossing partition $p \in NC(k,l) $ with a non-crossing pair partition $\psi(p) \in NC_2(2k,2l) $ ( i.e. a Temperley-Lieb-Jones diagram with first boundary point being the first upper point of $\psi(p)$) by drawing boundary lines around each block. If $T$ is the special $(k,l)$-tangle labelled by the diagram $\psi(p^*)$, then the map $\tilde{Z}_T := U_l^* Z_T U_k : \cH^{\ot k} \to \cH^{\ot l}$ introduced in \ref{mapslemma} is exactly the map $Z_p$ as can easily be checked. \\ \\
The representation category of $\G \wr_* \G_{aut}(A, \tr_A)$ is described in \cite[Section 3]{FP16}. Label the upper points of a partition $p \in NC(k,l)$ by unitary representations $u_1,\dots, u_k \in \Rep(\G)$ and the lower points by $v_1,\dots, v_l \in \Rep(\G)$. If $b$ is a block of $p$, $u_{U_b}$ will represent the tensor products of the representations labelling the upper part of $b$ read from left to right and $v_{L_p}$ will be the tensor product of the representations labelling the lower part of $B$. Similarly $\cH_{U_b}$ and $\cH_{L_b}$ will denote the tensor product of the Hilbert spaces belonging to these representations. If $b$ does not have an upper (lower) part, we set $u_{U_b} = \1_{\G}$ ($v_{L_b} = \1_{\G}$) by convention. A non-crossing partition $p \in NC(k,l)$ is called well-decorated w.r.t $u_1,\dots, u_k; v_1,\dots, v_l$, if for every block $b$ of $p$, we have $\Mor(u_{U_b}, v_{L_b}) \neq \{ 0 \}$. The set of such partitions is denoted by $NC_{\G}(u_1,\dots, u_k; v_1,\dots, v_l)$.\\ \\
If $p \in NC_{\G}(u_1,\dots, u_k; v_1,\dots, v_l)$ is a well-decorated partition and if $S_b \in \Mor(u_{U_b}, v_{L_b})$ is a morphism for any block $b$, we obtain a well-defined map 
\[ S = \otimes_{b \in p} S_b: \cH_{u_1} \ot \dots \ot \cH_{u_k} \to \cH_{v_1} \ot \dots \ot \cH_{v_l}, \]
by applying the maps $S_b$ iteratively on the legs of the tensor product belonging to $b$. 
To any well-decorated partition $p \in NC_{\G}(u_1,\dots, u_k; v_1,\dots, v_l)$ consisting of blocks $b$ and any such morphism $S = \bigotimes_{b \in p} S_b \in \bigotimes_{b \in p} \Mor( u_{U_{b}}, v_{L_{b}})$, one associates a map 
\[ Z_{p,S} = s_{p,L}^{-1} (Z_p \ot S) s_{p,U}: \bigotimes_{i=1}^k \cH \ot \cH_{u_i} \to \bigotimes_{j=1}^l \cH \ot \cH_{v_j}.  \]
Here, $s_{p,U}: \bigotimes_{i=1}^k \cH \ot \cH_{u_i} \to \cH^{\ot k} \ot \bigotimes_{i=1}^r \cH_{U_{b_i}}$ is the map reordering the spaces on the upper part. The map $s_{p,L}$ is defined analogously. It is shown in \cite{FP16} that
\begin{align}
\Mor(\bigotimes_{i=1}^k a(u_i), \bigotimes_{j=1}^l a(v_j) ) \ &= \nonumber
\\ \spann \{ Z_{p,S} \ &; \ p \in NC_{\G}(u_1,\dots, u_k; v_1,\dots, v_l), \ S \in \bigotimes_{i=1}^r \Mor( u_{U_{b_i}}, v_{L_{b_i}})  \}, \label{equa1}
\end{align}
whenever $u_1,\dots, u_k; v_1,\dots, v_l$ are irreducible. \\
Now, let $v = u_{\beta}$ be the unitary representation of $\G$ corresponding to the action $\beta$. Recall the definition of the unitary representation 
\begin{align} a(v) = \ \sum_{x \subset v } \sum_{k} (\id_{\cH} \ot S_{x,k} \ot \1) a(x) (\id_{\cH} \ot S_{x,k}^* \ot \1) \label{equa2}
\end{align}
from Section \ref{freewreath}. We would like to give a description of the intertwiner spaces of tensor powers of $a(v)$.

\begin{lem} The formula \ref{equa1} holds whenever $x_1,\dots, x_k, y_1, \dots, y_l \in \Rep(\G)$ are arbitrary unitary representations. In particular,
\[ \Mor(a(v)^{\ot k}, a(v)^{\ot l}) = \spann \{ Z_{p,S} \ ; \ p \in NC_{\G}(v^{k}; v^{l}), \ S \in \bigotimes_{b \in p} \Mor( v^{\ot |U_b|}, v^{\ot |L_b|})  \}. \] 
\end{lem}

\begin{proof}
This result follows directly by decomposing the unitary representation using \ref{equa2} and applying the result \ref{equa1} to the irreducible components. 
\end{proof}

\begin{proof}[Proof of Theorem \ref{main2}]
Let us write $\cQ = \TLJ_{\delta} * \cP(\beta)$ and consider the annular category $\cC_{\cQ}$ defined in Lemma \ref{mapslemma} with morphism spaces $\Mor_{\cQ}(k,l) \subset B((\cH \ot \cH_v)^{\ot k}, (\cH \ot \cH_v)^{\ot l} )$. By Theorem \ref{main1} and its proof, it is enough to show that
\[ \Mor(a(v)^{\ot k}, a(v)^{\ot l}) \ = \ \Mor_{\cQ}(k,l). \]
To prove this, let us choose a well-decorated partition $p \in NC_{\G}(v^{k}; v^{l})$ with blocks labelled by $b \mapsto S_b \in \Mor( v^{\ot |U_b|}, v^{\ot |L_b|}) = \Mor_{\cC_{\cP}}(|U_b|, |L_b|)$. For any $b \in p$, choose an element $\eta_b \in \cP_{|b|}$ such that $S_b = Z_{T_b}$, where $T_b$ is the special $(|U_b|, |L_b|)$-tangle labelled by $\eta_b$.  Let $\psi$ be the 'fattening' isomorphism from the discussion above, mapping a noncrossing partition $p \in NC(k,l)$ to a Temperley-Lieb-Jones diagram $\psi(p) \in NC_2(2k,2l)$. In particular the map $ \phi: p \mapsto \psi(p^*)$ maps a block of $p$ to a shaded region of $\phi(p)$. Hence we can identify the pair $(p, S = \bigotimes_{b\in p} S_b)$ with the Temperley-Lieb-Jones diagram $\phi(p)$ together with a labelling of its shaded regions $r = \phi(b)$ by elements in $ \Mor_{\cC_{\cP}}(|U_b|, |L_b|)$. Consider the unique irreducible planar tangle $T$ of degree $k+l$ with $\pi_T = \kr'(\phi(p))$. Then, $(\phi(p), T)$ constitutes a free pair by Proposition \ref{freePairKreweras}. Note that $T$ has an inner disk $D_b$ of degree $|b|$ for every block $b$ in a canonical manner. Consider $ \eta = Z_T((\eta_b)_{b \in p}) \in \cP_{k+l}$. The special $(k,l)$-tangle $T_{\eta}$ labelled by $\eta$ exactly induces the map $S$, i.e. $Z_{T_{\eta}} = S$. In conclusion, we have that the special $(k,l)$-tangle $W$ labelled by $\phi(i) \ot \eta \in \cQ_{k+l}$ induces a morphism in $\Mor_{\cQ}(k,l)$ which can easily be checked to coincide with $Z_{p,S}$. It follows that $\Mor(a(v)^{\ot k}, a(v)^{\ot l}) \subset \Mor_{\cQ}(k,l).$   \\ \\
On the other hand, consider a special $(k,l)$-tangle labelled by $Z_{p'} \ot \eta \in \cQ_{k+l}$ with $p' \in \TLJ_{\delta, k+l}$ a pair partition and $\eta \in \cQ_{k+l}$. By Proposition \ref{irredFactoriz}, we can write $\eta = Z_{T'}\circ_{D_1,\dots, D_r}(\eta_1, \dots, \eta_r)$ where $(p',T')$ is a reduced free pair. In particular, $\pi_{T'} = \kr'(p')$ and $T'$ has an inner disk $D_i$ for any shaded region of $p'$ or equivalently for every block $b_i$ of $\phi^{-1}(p')$ labelled by $\eta_i \in \cP_{|b_i|}$. In particular, we get a morphism $S_{b_i} = Z_{T_{\eta_i}} \in \Mor_{\cP}(k,l)$ induced by the special $(k,l)$-tangle labelled by $\eta_i$ and by construction $ Z_{T_{\eta}} = \bigotimes_{b \in \phi^{-1}(p')} S_b$. Consequently, $\Mor_{\cQ}(k,l) \subset \Mor(a(v)^{\ot k}, a(v)^{\ot l})$.


\end{proof}

Let us move towards the setting of Bichon's initial definition (see \ref{DefBichon}).
Let $X$ be a set of $d$ points and let $\tr: C(X) \to \C$ be the trace on $C(X)$ given by integrating against the uniform probability measure. Moreover, let $\al: C(X) \to C(X) \ot \Pol(\F)$ be a faithful, centrally ergodic, $\tr$-preserving action. Recall that the compact matrix quantum group $(\F,u)$ with $u = u_{\al}$ is a quantum subgroup of the quantum permutation group $S_d^+ = \G_{aut}(C(X),\tr)$. Let $\beta: B \to B \ot \Pol(\G)$ be a faithful, centrally ergodic $\tr_B$-preserving action of a compact quantum group $\G$ on $B$ and put $v=u_{\beta}$.

\begin{thm} \label{main3}
Consider the subfactor planar algebra $ \cQ = \cP(\al)* \cP(\beta) \subset \cP^{\Gamma(C(X) \ot  B)}$ and let $\gamma = \al_{\cQ}: C(X) \ot  B \to C(X) \ot B \ot \Pol(\mathbb H)$ be the action corresponding to $\cQ$ in the sense of Theorem \ref{main1}. We have an isomorphism of compact matrix quantum groups 
\[ (\mathbb H, u_{\gamma}) \cong (\G \wr_* \F, w) \]
with $w$ as in Definition \ref{DefBichon}.

\end{thm}

The reason that the formulation of Theorem \ref{main3} is somewhat inverse to the one of Theorem \ref{main2} is that we have not shown yet that the action $\al_w$ is centrally ergodic and therefore $\cP(\G \wr_* \F)$ is not yet well-defined. However, this will follow from the proof of the theorem and therefore a posteriori we can rephrase Theorem \ref{main3} in exactly the same way as Theorem \ref{main2}. Theorem \ref{mainB} corresponds to this rephrasing.

Before proving Theorem \ref{main3}, we first explain why it yields Corollary \ref{corD}. Suppose that $(\F,\alpha)$ and $(\G,\beta)$ are two pairs of compact quantum groups and faithful, centrally ergodic actions acting respectively on $C(X)$ and $C(Y)$, where $X$ and $Y$ are finite sets. Let us use the notations of Theorem \ref{main3}. On the one hand, the law of $\chi_{\beta\wr_{*}\alpha}$ is exactly the law of $\chi_{w}$ in this theorem. By a standard result of Woronowicz \cite{Wo88}, $h_{\G\wr_{*}\F}(\chi_{w}^{k})=\dim \Fix(w^{\otimes k})$, and by Theorem \ref{resultBanica}, $\Fix(w^{\otimes k})=\cP(w)_{k}$. Thus, $h_{\G\wr_{*}\F}(\chi_{w}^{k})=\dim(\cP(w)_{k})$, and therefore $\chi_{w}$ is distributed as $\mu_{\cP(w)}$. Likewise, $\chi_{\alpha}$ is distributed as $\mu_{\cP(\alpha)}$ and $\chi_{\beta}$ is distributed as $\mu_{\cP(\beta)}$. On the other hand, Theorem \ref{main3} yields that $\cP(w)=\cP(\alpha)*\cP(\beta)$. Thus, 
$$\mu_{\cP(w)}=\mu_{\cP(\alpha)*\cP(\beta)}=\mu_{\cP(\alpha)}\boxtimes\mu_{\cP(\beta)},$$
where the last equality is given by \cite{BJXX} or \cite[Chapter 7.3]{Ta15}. Finally, we get
$$\chi_{\beta\wr_{*}\alpha}=\chi_{\alpha}\boxtimes\chi_{\beta},$$
which is the result of Corollary \ref{corD}.

We will now turn to the proof of Theorem \ref{main3}. We compute first some basic relations satisfied by the coefficients of the fundamental representation $w = (u_{ij} v_{kl}^{(i)})_{\substack{1\leq i,j\leq n\\1\leq k,l\leq m}} $ of the free wreath product $\G\wr_{*}\F$. We expand the unique invariant vector $1_{B}$ of the action $\beta$ as $1_{B}=\sum_{i=1}^{m}\nu_{i}f_{i}$ in the canonical orthonormal basis.
\begin{prop}\label{algebraicDescriptionFWP}
The fundamental representation $w = (w_{(i,k),(j,l)}) = (u_{ij} v_{kl}^{(i)})_{\substack{1\leq i,j\leq n\\1\leq k,l\leq m}} $ of $\G\wr_{*}\F$ satisfies the following properties.
\begin{enumerate}
\item For each $1\leq i,j\leq n$ and  each $1\leq k\leq m$ such that $\nu_{k}\not=0$, we have $u_{ij}=\frac{1}{\nu_{k}}\sum_{l=1}^{m}\nu_{l} w_{(i,k),(j,l)}$. 
\item For all $1\leq i\leq n$, we have $v_{k,l}^{(i)}=\sum_{j=1}^{n} w_{(i,k),(j,l)}$. 
\item For each $1\leq i,j,j'\leq n$ and $1\leq k,l\leq m$, $v_{k,l}^{(i)}$ commutes with $u_{ij}$ and $u_{ij} w_{(i,k)(j',l)}=\delta_{jj'} w_{(i,k),(j,l)}$.
\end{enumerate}
\end{prop}
\begin{proof}
For all $1\leq i,j\leq n$, we compute 
\begin{equation}\label{relation1}
\sum_{l=1}^{m}\nu_{l} w_{(i,k)(j,k)}=\sum_{l=1}^{m}\nu_{l}u_{ij} v_{kl}^{(i)}=\nu_{k}u_{ij},
\end{equation}
where the last equality is follows from the fact that $u_{ij}$ commutes with $ v_{kl}^{(i)}$ (and thus also with $ v_{kl}^{(i)*}$), and from the equalities $u_{ij}^{2}=u_{ij}^{*}=u_{ij}$ and $\sum_{l=1}^{m} v_{kl}^{(i)}=\nu_{k}$.\\
Similarly,
\begin{equation}\label{relation2}
\sum_{j=1}^{n}w_{(i,k)(j,l)}=\sum_{j=1}^{n}u_{ij} v_{kl}^{(i)}= v_{kl}^{(i)},
\end{equation}
where the last equality is following from $\sum_{j=1}^{n}u_{ij}=\1$. We have the relation $u_{ij} v_{kl}^{(i)}= v_{kl}^{(i)}u_{ij}$ by definition of the free wreath product and hence we get
$$u_{ij}w_{(i,k)(j,l)}=u_{ij}u_{ij'} v_{kl}^{(i)}=\delta_{jj'}u_{ij} v_{kl}^{(i)}=\delta_{jj'}w_{(i,k)(j,l)}.$$ 
\end{proof}

\begin{proof}[Proof of Theorem \ref{main3}]
The compact matrix quantum groups $(\mathbb H,u_{\gamma})$ and $(\G\wr_{*}\F,w)$ both act on the same Hilbert space, since we are considering the action of two compact quantum groups on the same $C^{*}$-algebra $C(X)\otimes B$. The proof of the theorem is done in two steps.

\textit{Step 1}: Let us show that there exists a $*$-homomorphism $ C(\G \wr_* \F) \to C(\mathbb H)$ which maps $w$ onto $u_{\gamma}$ by showing that all the relations satisfied by the coefficients of $w$ are also satisfied by the coefficients of $u_{\gamma}$. In particular, this will show that the free wreath action $\beta \wr_* \al$ is centrally ergodic. \\
By Theorem \ref{main2}, we have $\TLJ_{\sqrt{n}}*\cP(\beta)=\cP(\al_{\tilde{w}})$, where $\tilde{w}$ is the fundamental representation of the free wreath product $\G \wr_* S_n^+$ of $(\G, v)$ with the quantum permutation group $(S_{n}^{+},s)$. Since $\TLJ_{\sqrt{n}}\subset \cP(\alpha)$, we obtain
 \[ \cP(\al_{\tilde{w}})=\TLJ_{\sqrt{n}}*\cP(\beta)\subset \cP(\alpha)*\cP(\beta)=\cP(\gamma). \] Therefore, there exists a $*$-homomorphism $ \Phi: C(\G \wr_* S_n^+) \to C(\mathbb H)$ mapping $\tilde{w}$ to $u_{\gamma}$.\\
Let us write $\tilde{u}_{ij}=\frac{1}{\nu_{k}}\sum_{l} (u_{\gamma})_{(i,k)(j,l)}$ for $k$ such that $\nu_{k}\not=0$ and $\tilde{ v}_{kl}^{(i)}=\sum_{j}(u_{\gamma})_{(i,k)(j,l)}$. By the aforementioned quantum group inclusion, the matrix $\tilde{ v}^{(i)}$ is the image of $v$ under the composition $\pi \circ \iota_i: C(\G) \to C(\G)^{*n} * C(S_n^+) \to C(\G \wr_* S_n^+),$ where $\iota_i$ denotes the $i$-th embedding into the free product and $\pi$ denotes the quotient map. Moreover the matrices $\tilde{u},\tilde{ v}^{(i)}$ satisfy the relations appearing in part (3) of the previous proposition. By the same reasoning as before, we obtain a $*$-homomorphism $ \varphi: C(S_n^+) \to C(\mathbb H)$ mapping the fundamental representation $s$ onto the matrix $\tilde{u}$ by composing the natural embedding with the quotient map. Let us show that $\varphi$ factors through the quotient $C(\F)$.  Note that $\Pol(\F)$ is the universal $*-$algebra generated by the coefficients of $(u_{ij})_{1\leq i,j\leq n}$ such that all the relations $\sum_{i_{1},\dots,i_{r}}\lambda_{i_{1}\dots i_{r}}u_{i_{1}j_{1}}\dots u_{i_{r}j_{r}}=\lambda_{j_{1}\dots j_{r}}$
are satisfied for each vector $(\lambda_{i_{1}\dots i_{r}})_{1\leq i_{1},\dots,i_{r}\leq n}\in \Mor(\1,u^{ \ot r})$, $r\geq 1$ and all tuples $\vec{i}$ and $\vec{j}$. Fix a vector $(\lambda_{i_{1}\dots i_{r}})_{1\leq i_{1},\dots,i_{r}\leq n}\in \Mor(\1,u^{ \ot r})$. Then, $\lambda\in U_k^{*}(\cP(\alpha))$ and thus $U_k(\lambda)\otimes S_{\cP(\beta)}\in \cP(\gamma)$. Applying $U_k^{*}$ on the latter vector yields that
\begin{equation}\label{invariantVectorStep2}
x:=\sum_{i_{1},\dots,i_{r}=1}^{n}\lambda_{i_{1}\dots i_{r}}(X_{i_{1}}\otimes 1_{B})\otimes \dots\otimes (X_{i_{r}}\otimes 1_{B})\in \Mor(\1,u_{\gamma}^{ \ot r})).
\end{equation}
On the other hand, the action of $\mathbb H$ on this vector yields
\begin{align*}
u_{\gamma}(x \ot \1)=&\sum_{\substack{\vec{i},\vec{j}\in\llbracket 1,n\rrbracket^{r}\\\vec{k},\vec{l}\in\llbracket 1,m\rrbracket^{r}}}\lambda_{\vec{i}}\prod_{s=1}^{r}\nu_{k_{s}}(X_{j_{1}}\otimes f_{l_{1}})\otimes\dots\otimes (X_{j_{r}}\otimes f_{l_{r}})\otimes (u_{\gamma})_{(j_{1},i_{1})(l_{1},k_{1})}\dots(u_{\gamma})_{(j_{r},i_{r})(l_{r},k_{r})}\\
=&\sum_{\substack{\vec{i},\vec{j}\in\llbracket 1,n\rrbracket^{r}\\\vec{l}\in\llbracket 1,m\rrbracket^{r}}}\lambda_{\vec{i}}(X_{j_{1}}\otimes f_{l_{1}})\otimes\dots\otimes (X_{j_{r}}\otimes f_{l_{r}})\\
&\quad\quad\quad\quad\otimes \left(\sum_{k_{1}=1}^{m}\nu_{k_{1}}(u_{\gamma})_{(j_{1},i_{1})(l_{1},k_{1})}\right)\dots\left(\sum_{k_{r}=1}^{m}\nu_{k_{r}}(u_{\gamma})_{(j_{r},i_{r})(l_{r},k_{r})}\right).
\end{align*}
By definition of $\tilde{u}_{ji}$, we have $\sum_{k=1}^{m}\nu_{k}(u_{\gamma})_{(j,i)(l,k)}=\nu_{l}\tilde{u}_{ji}$ for $1\leq i,j\leq n,1\leq k,l\leq m$ and thus
\begin{align*}
u_{\gamma}(x)=&\sum_{\substack{\vec{i},\vec{j}\in\llbracket 1,n\rrbracket^{r}\\\vec{l}\in\llbracket 1,m\rrbracket^{r}}}\lambda_{\vec{i}}(X_{j_{1}}\otimes f_{l_{1}})\otimes\dots\otimes (X_{j_{r}}\otimes f_{l_{r}})\otimes (\nu_{l_{1}})\tilde{u}_{j_{1}i_{1}}\dots\nu_{l_{r}}\tilde{u}_{j_{r}i_{r}}\\
=&\sum_{\vec{i},\vec{j}\in\llbracket 1,n\rrbracket^{r}}\lambda_{\vec{i}}\left((X_{j_{1}}\otimes\left(\sum_{l_{1}=1}^{m}\nu_{l_{1}} f_{l_{1}}\right)\right)\otimes\dots\otimes \left(X_{j_{r}}\otimes \left(\sum_{l_{r}=1}^{m}\nu_{l_{r}}f_{l_{r}}\right)\right)\otimes \tilde{u}_{j_{1}i_{1}}\dots\tilde{u}_{j_{r}i_{r}}\\
=&\sum_{\vec{j}\in\llbracket 1,n\rrbracket^{r}}(X_{j_{1}}\otimes 1_{B})\otimes\dots\otimes (X_{j_{r}}\otimes 1_{B})\otimes \sum_{\vec{i}\in\llbracket 1,n\rrbracket^{r}}\lambda_{\vec{i}}\tilde{u}_{j_{1}i_{1}}\dots\tilde{u}_{j_{r}i_{r}}.
\end{align*}
The last equality together with \eqref{invariantVectorStep2} implies that 
\[\sum_{\vec{i}\in\llbracket 1,n\rrbracket ^{r}}\lambda_{\vec{i}}\tilde{u}_{j_{1}i_{1}}\dots\tilde{u}_{j_{r}i_{r}}=\lambda_{\vec{j}},
\]
for any $\vec{j}\in\llbracket 1,n\rrbracket ^{r}$. Hence, $\varphi$ factors through $C(\F)$ and by the universal properties of the free product and of the quotient, we obtain a surjective $*$-homomorphism $C(\G \wr_* \F) \to C(\mathbb H)$ mapping $w$ to $u_{\gamma}$. \\
\textit{Step 2}: Since we know now that $\beta \wr_* \al$ is centrally ergodic, we can build the corresponding planar algebra $\cP(\beta \wr_* \al)$. Let us prove that $\cP(\gamma)$ is a planar subalgebra of $\cP(\beta \wr_* \al)$. It suffices to prove that a generating subset of $\cP(\gamma)$ is contained in $\cP(\beta \wr_* \al)$. Since $\cP(\gamma) =\cP(\alpha)*\cP(\beta)$, we can consider the generating subset given in Proposition \ref{generatTangle}. For each $k\geq 1$, this subset is given by elements of two kinds:
\begin{itemize}
\item $U_{\cP(\alpha)}(k)\otimes \cP(\beta)_{k}$: note that $U_{\cP(\alpha)}(k)\otimes \cP(\beta)_{k}\subset \TLJ_{\sqrt{n}}*\cP(\beta)$. By Theorem \ref{main2}, $\TLJ_{\sqrt{n}}*\cP(\beta)=\cP(\beta\wr_{*}\alpha')$, where $\alpha'$ is the action of $\G_{aut}(C(X),\tr) = S_n^+$ on $C(X)$. Since $(\F,\alpha)$ is a quantum subgroup of $(S_n^+,\alpha')$ we get $\cP(\beta\wr_{*}\alpha')\subset \cP(\beta \wr_* \al)$, and $U_{\cP(\alpha)}(k)\otimes \cP(\beta)_{k}\subset \cP(w)$.
\item $\cP(\alpha)_{k}\otimes S_{\cP(\beta)}(k)$: let $x\in \cP(\alpha)_k$ and let us prove that $x\otimes S_{\cP(\beta)}(k)\in \cP(\beta \wr_* \al)_{k}$ by showing that $U_k^{*}(x\otimes S_{\cP(\beta)}(k))\in U_k^{*}(\cP(\beta \wr_* \al)_{k})=\Mor(\1, w^{\otimes k})$. Write $U_k^{*}(x)=\sum_{i_{1},\dots,i_{k}=1}^{d}\lambda_{i_{1}\dots i_{k}}X_{i_{1}}\otimes \dots\otimes X_{i_{k}}$, where $(X_i)_{i=1}^n$ denotes the canonical orthonormal basis of $C(X)$ with respect to the Markov trace. Since 
$U_k^{*}(S_{\cP(\beta)}(k))= \eta_{B}^{\otimes k}$,
\[ U_k^{*}(x\otimes S_{\cP(\beta)}(k))=\sum_{i_{1},\dots,i_{k}=1}^{d}\lambda_{i_{1}\dots i_{k}}(X_{i_{1}}\otimes \eta_{B})\otimes \dots\otimes (X_{i_{k}}\otimes \eta_{B}). \]
Applying the action of $\G\wr_{*}\F$ on $U_k^{*}(x\otimes S_{\cP(\beta)}(k))$ yields
\begin{align*}
&w^{\otimes k}( \sum_{i_{1},\dots,i_{k}=1}^{d}\lambda_{i_{1}\dots i_{k}}(X_{i_{1}}\otimes \eta_{B})\otimes \dots\otimes (X_{i_{k}}\otimes \eta_{B}) \ot \1_{\G\wr_{*}\F})\\
&\quad\quad=\sum_{i_{1},\dots,i_{k}=1}^{d}\lambda_{i_{1}\dots i_{k}}w^{\otimes k}\left((X_{i_{1}}\otimes \eta_{B})\otimes \dots\otimes (X_{i_{k}}\otimes \eta_{B}) \ot \1_{\G\wr_{*}\F}\right).
\end{align*}
Since $\eta_{B}\in \Fix(\beta)$, $w(X_{i}\otimes \eta_{B} \1_{\G\wr_{*}\F})=\sum_{j=1}^{d}X_{j}\otimes \eta_{B})\otimes u_{ji}$. Thus, 
\begin{align*}
&\sum_{i_{1},\dots,i_{k}=1}^{d}\lambda_{i_{1}\dots i_{k}}w^{\otimes k}\left((X_{i_{1}}\otimes \eta_{B})\otimes \dots\otimes (X_{i_{k}}\otimes \eta_{B}) \ot \1_{\G\wr_{*}\F} \right)\\=&\sum_{i_{1},\dots,i_{k}=1}^{d}\lambda_{i_{1}\dots i_{k}}\sum_{j_{1},\dots,j_{k}=1}^{d}(X_{j_{1}}\otimes \eta_{B})\otimes \dots\otimes (X_{j_{k}}\otimes \eta_{B})\otimes u_{j_{1}i_{1}}\dots u_{j_{k}i_{k}}\\
=&\sum_{j_{1},\dots,j_{k}=1}^{d}(X_{j_{1}}\otimes \eta_{B})\otimes \dots\otimes (X_{j_{k}}\otimes \eta_{B})\otimes(\sum_{i_{1},\dots,i_{k}=1}^{d}\lambda_{i_{1}\dots i_{k}}u_{j_{1}i_{1}}\dots u_{j_{k}i_{k}})\\
=&\sum_{j_{1},\dots,j_{k}=1}^{d}\lambda_{j_{1}\dots j_{k}}(X_{j_{1}}\otimes \eta_{B}\otimes \dots\otimes (X_{j_{k}}\otimes \eta_{B}))\otimes \1_{\G\wr_{*}\F},
\end{align*}
the last equality being due to the fact that $\sum_{i_{1},\dots,i_{k}=1}^{d}\lambda_{i_{1}\dots i_{k}}X_{i_{1}}\otimes \dots\otimes X_{i_{k}}\in \Mor(\1_{\F}, u^{\otimes k})$, which is equivalent to the equality $(\sum_{i_{1},\dots,i_{k}=1}^{d}\lambda_{i_{1}\dots i_{k}}u_{j_{1}i_{1}}\dots u_{j_{k}i_{k}})=\lambda_{j_{1}\dots j_{k}}$. Hence,
\begin{align*}
&w^{\otimes k}\left( \sum_{i_{1},\dots,i_{k}=1}^{d}\lambda_{i_{1}\dots i_{k}}(X_{i_{1}}\otimes \eta_{B})\otimes \dots\otimes (X_{i_{k}}\otimes \eta_{B})\ot \1_{\G\wr_{*}\F} \right)\\
&\quad\quad= \sum_{i_{1},\dots,i_{k}=1}^{d}\lambda_{i_{1}\dots i_{k}}(X_{i_{1}}\otimes \eta_{B})\otimes \dots\otimes (X_{i_{k}}\otimes \eta_{B})\otimes \1_{\G\wr_{*}\F},
\end{align*}
which means that $U_k^{*}(x\otimes S_{\cP(\beta)}(k))\in \Mor(\1_{\G \wr_* \F}, w^{\otimes k})$. Therefore, $x\otimes S_{\cP(\beta)}(k)\in U_k(\Mor(\1_{\G \wr_* \F}, w^{\otimes k}))=\cP(\beta \wr_* \al)$.
\end{itemize}
Finally, $U_{\cP(\alpha)}(k)\otimes \cP(\beta)_{k}\subset \cP(\beta \wr_* \al)$ and $\cP(\alpha)_{k}\otimes S_{\cP(\beta)}(k)\subset \cP(\beta \wr_* \al)$ for all $k\geq 1$. By Proposition \ref{generatTangle}, this yields that $\cP(\alpha)*\cP(\beta)\subset \cP(\beta \wr_* \al)$. Since $\cP(\alpha)*\cP(\beta)=\cP(\gamma)$, applying the isomorphism $U_k^{*}$ yields that $\Mor(\1_{\mathbb H}, u_{\gamma}^{\otimes k})\subset \Mor(\1_{\G \wr_* \F}, w^{\otimes k})$ for all $k\geq 1$. By Tannaka-Krein duality, this means that $(\G\wr_{*}\F,w)\subset (\mathbb H,u_{\gamma})$. 

\end{proof}

Due to the theorems \ref{main2} and \ref{main3}, the following definition is consistent with both previous definitions.

\begin{defi} \label{wreathnewdef}
Let $\al$ (respectively $\beta$) be an action of the compact quantum group $\F$ (resp. $\G$) on the finite-dimensional $C^*$-algebra $A \ (B)$. Assume that both $\al$ and $\beta$ are faithful, centrally ergodic and preserve the Markov trace on $A$, respectively $B$. The action $\beta \wr_* \al: A \ot B \to A \ot B \ot C(\mathbb H)$ corresponding to 
\[ \cP(\al) * \cP(\beta) \ \subset \ \cP^{\Gamma(A \ot B)} \]
is called the \emph{free wreath action} of $\beta$ and $\al$ and the compact matrix quantum group $(\mathbb H, u_{\beta \wr_* \al})$ is called the free wreath product of $(\G,u_{\beta})$ and $(\F, u_{\al})$.
\end{defi}

\begin{rem} \label{defirem}
Recall that by Example \ref{exampleaction}, every compact quantum group of Kac type whose representation category $\Rep(\G)$ is finitely generated, admits a faithful, centrally ergodic, $\tr$-preserving action on \emph{some} finite-dimensional $C^*$-algebra. Moreover, due to the correspondence of actions and representations (see Section \ref{actions}), choosing such an action is the same as choosing a well-behaved generator of $\Rep(\G)$. Hence if $\beta$ and $\tilde{\beta}$ are different choices of actions of the same quantum group $\G$, the resulting compact \emph{matrix} quantum groups $(\mathbb H, u_{\beta \wr_* \al})$ and $(\tilde{\mathbb H}, u_{\tilde{\beta} \wr_* \al})$ will \emph{not} necessarily be isomorphic. However, as compact quantum groups, we have an isomorphism $\mathbb H \cong \tilde{\mathbb H}$, since the concrete rigid $C^*$-tensor categories generated by $u_{\beta \wr_* \al}$ and $u_{\tilde{\beta} \wr_* \al}$ are the same. In fact, one can generalize Definition \ref{wreathnewdef} even further. We can drop the assumption that $\Rep(\G)$ is finitely generated. We do so by prescribing a concrete rigid $C^*$-tensor category $\cC$ as we will briefly sketch now. \\
Let $\Obj(\cC) = \langle \Rep(\G) \rangle$ be the set of words with letters in $\Rep(\G)$. The tensor operation will be the concatenation of words and therefore we will denote such a word by $a(x_1) \ot a(x_2) \dots \ot a(x_s), \ x_1, \dots, x_m \in \Rep(\G)$. To describe the morphism space 
\[ \Mor( a(x_1) \ot a(x_2) \dots \ot a(x_m), a(y_1) \ot a(y_2) \dots \ot a(y_n) ), \]
consider a planar tangle of degree $k_0 = m+n$ with inner disks $D_i$ of degree $k_i$ and elements $v_i \in \cP(\al)_{k_i}, \ i = 1, \dots, s$. By definition, we have $ \eta := Z_T(v_1, \dots, v_s) \in \cP(\al)_{k_0}$ and therefore we can consider the special $(m,n)$-tangle $T_{\eta}$ labelled by $\eta$. If $\cH$ is the Hilbert space associated to the action $\al$ of $\F$, this special tangle induces a map $Z(T,v_1,\dots, v_s) := Z_{T_{\eta}} : \cH^{\ot m} \to \cH^{\ot n}$. Now, mark the boundary intervals of the outer disk of $T$ which are adjacent to shaded regions of the tangle clockwise by $y_1,\dots,y_n,x_m,\dots,x_1$. Every shaded region $r$ of the tangle is now marked on its outer boundary by elements $x_{i_1}, \dots x_{i_v}, y_{j_1}, \dots, y_{j_w}, i_1 < i_2 \dots < i_v, \ j_1 < j_2 \dots < j_w$. Associate with any such region $r$ a morphisms 
\[ S_r \ \in \ \Mor(x_{i_1} \ot \dots \ot x_{i_v}, y_{j_1} \ot \dots \ot y_{i_w}) \]
with the usual convention that we replace $x_{i_1} \ot \dots \ot x_{i_v}$ by $\1_{\G}$ if no boundary interval of $r$ is marked with a representation in $\{ x_1,\dots, x_m \}$ (with a similar convention for the $y$ part). By ordering the morphisms $S_r$ appropriately, we obtain a morphism 
\[ S = \bigotimes_{r \text{ shaded region}} S_r: \ \cH_{x_1} \ot \dots \ot \cH_{x_m} \ \to \ \cH_{y_1} \ot \dots \ot \cH_{y_n},   \]
and therefore a map
\[ Z(T,v_1,\dots, v_s) \ot S: \cH^{\ot m} \ot \cH_{x_1} \ot \dots \ot \cH_{x_m} \ \to \ \cH^{\ot n} \ot \cH_{y_1} \ot \dots \ot \cH_{y_n}. \]
If we reorder as in the beginning of this section, we obtain maps
\[ Z(T,v_1,\dots, v_s; S): \bigotimes_{i=1}^m \cH \ot \cH_{x_i} \ \to \ \bigotimes_{j=1}^n \cH \ot \cH_{y_j}.  \]
 We set $\Mor( a(x_1) \ot a(x_2) \dots \ot a(x_m), a(y_1) \ot a(y_2) \dots \ot a(y_n) )$ to be the span of all maps obtained in this way. One can check that this defines a concrete rigid $C^*$-tensor category and hence a compact quantum group $\G \wr_* (\F, \al)$. If $\Rep(\G)$ is finitely generated and we choose a faithful, centrally ergodic, $\tr$-preserving action $\beta$ of $\G$, the compact matrix quantum group $(\G \wr_* (\F, \al), a(u_{\beta}))$ is isomorphic to the free wreath product of $(\G, \beta)$ and $(\F, \al)$ as in Definition \ref{wreathnewdef}, since building the reduced free pair $T * T'$ of tangles $T$ and $T'$ corresponds to labelling shaded regions of $T$ with morphisms in the annular category of $\cP(\beta)$. 
\end{rem}

\section{Approximation properties} \label{approximation}

Central approximation properties of discrete quantum groups (i.e. duals of compact quantum groups) were defined in \cite{dCFY14} and were explained more conceptually in \cite{PV15} (see also \cite{NY15}, \cite{GJ16}) in terms of representations of rigid $C^*$-tensor categories. A functional $\omega: \Pol(\G) \to \C$ is called central if
 \[ (\omega \ot \psi) \circ \Delta = (\psi \ot \omega) \circ \Delta \]
for all $\psi \in \Pol(\G)^*$. It is called positive if $\omega(x^* x) \geq 0$ for all $x \in \Pol(\G)$. Central functionals on $\G$ are in one-to-one correspondence with functions $\varphi: \Irr(\G) \to \C$ through the map 
\[ \varphi \to \omega_{\varphi}, \quad \omega_{\varphi}(u_{ij}^x) = \varphi(x) \delta_{ij}, \]
where $x \in \Irr(\G)$ and $u^x = (u_{ij}^x)$ is a representative of $x$ as a unitary matrix. In addition, $\omega_{\varphi}$ is positive if and only if $\varphi$ is positive in the sense of \cite{PV15}. In \cite{PV15}, the authors also define what it means for a function $\varphi: \Irr(\G) \to \C$ (and hence for a central functional on $\Pol(\G)$) to be completely bounded. Characterizations of this property are given in \cite{AV16} and \cite{AdLW16}.

\begin{defi}[\cite{dCFY14}, \cite{PV15}]
The discrete dual of $\G$ satisfies
\begin{itemize}
\item the central Haagerup property if there exists a net of positive functions $\varphi_i: \Irr(\G) \to \C$ in $c_0(\Irr(\G))$ such that $\varphi_i \to 1$ pointwise as $i \to \infty$;
\item the  central almost completely positive approximation property (ACPAP) if there exists a net of positive functions $\varphi_i: \Irr(\G) \to \C$ that converges to $1$ pointwise and such that for every fixed $i$, there exists a net of finitely supported completely bounded functions $\psi_i^k: \Irr(\G) \to \C$ such that $\lim_{k} \| \psi_i^k - \varphi_i \|_{cb} = 0$.
\end{itemize}
\end{defi}
Clearly, the central ACPAP implies both the central Haagerup property and the central complete metric approximation property (see \cite{dCFY14}). Also, note that for Kac type quantum groups, the central version of the Haagerup property is equivalent to the original non-central one. 
The following result mentioned in the introduction (Corollary \ref{corC}) is an immediate consequence of \cite[Proposition 9.3]{PV15}.

\begin{cor}
Let $(\F,\al)$, $(\G,\beta)$ be compact quantum groups with faithful, centrally ergodic, $\tr$- preserving actions $\al, \beta$ on finite dimensional $C^*$-algebras and let $(\G \wr_* \F, \beta \wr_* \al)$ be their free wreath product.
\begin{itemize}
\item If both $\F$ and $\G$ have the central Haagerup property, so does $\G \wr_* \F$.
\item If both $\F$ and $\G$ have the central ACPAP, so does $\G \wr_* \F$. 
\end{itemize}
\end{cor}

\section{Basis of a free product of planar algebras} \label{basissection}
In this section we will describe a basis of a free product of planar algebras using methods of free probability theory. As a corollary, this yields a basis for the intertwiner spaces of a free wreath product quantum group.
\subsection{Boolean decomposition of a planar algebra}\label{BooleanDecompoPlanarAlgebra}
We will start out by introducing a way to decompose a subfactor planar algebra into tensor products of vector subspaces. This should be compared to the introduction of the graded algebra of a planar algebra in \cite{GJS10}. In analogy with Section \ref{reconstruction}, given a subfactor planar algebra $\cP$, we consider the maps $U_{k,l}: \cP_k \ot \cP_l \to \cP_{k+l}$ induced by the tangle depicted in Figure \ref{concatenation} for $k,l \geq 1$. This is nothing but the graded multiplication defined in \cite{GJS10}.  For an interval partition $I=(i_{1},\dots,i_{r})\in \cI(n)$ (this is short for $I = (\llbracket 1, i_1 \rrbracket, \dots, \llbracket i_r, n \rrbracket)$), let us define the map $U_{I}$ by induction on the length of $I$ through the relation
$$U_{I}=U_{i_{1},n-i_{1}}\circ( \id_{i_{1}}\otimes U_{I'}),$$
where $I'=(i_{2},\dots,i_{r})\in \cI(n-i_{1})$.
We say that a sequence of vector subspaces $(L_{n})_{n\geq 1}$ is a subspace of a planar algebra $\cP$ if $L_{n}\subset\cP_{n}$ for all $n\geq 1$. We simply write $(L_{n})_{n\geq 1}\subset \cP$ when no confusion with the inclusion of planar algebras can arise. If $(L_{n})_{n\geq 1}\subset \cP$ and $I=(i_{1},\dots,i_{r})$ is an interval partition of $n$, we denote by $L(I)$ the subspace of $\cP_{n}$ given by 
$$L(I)=U_{I}\left(L(i_{1})\otimes \dots\otimes L(i_{r})\right).$$
\begin{defi}
A subspace $(L_{n})_{n\geq 1}$ of $\cP$ is called a Boolean decomposition of $\cP$ if we have the direct sum decomposition
$$\cP_{n}=\bigoplus_{I\in\mathcal{I}(n)}L(I)$$
for all $n\geq 1$.
\end{defi}
Let $\mu_{\cP}$ be the unique probability measure on $\R$ with $n$-th moment $\dim \cP_n$. The fact that this measure indeed exists follows from observations on the principal graph of $\cP$ in \cite{J99}.
Note that if $(L_{n})_{n\geq 1}$ is a Boolean decomposition of $\cP$, then necessarily $\dim L(n)=b_{\mu_{\cP}}(n)$ by the definition of $\mu_\cP$ and by the definition of the Boolean cumulants (refer to Section \ref{nonCrossFreeProba} for the definition of Boolean cumulants of a probability measure).\\
Let us show that there always exists a Boolean decomposition of a planar algebra $\cP$. Recall that each vector space $\cP_{n}$ is a Hilbert space with the hermitian product given by the map $(x,y)\mapsto \Tr_{n}(y^*x)$.
\begin{defi}
The $n$-th Boolean orthogonal subspace of $\cP$, denoted by $\cB_{\cP}(n)$, is the orthogonal complement of $\bigcup\limits_{\substack{k,m\geq 1\\k+m=n}}U_{k,m}(\cP_{k}\otimes \cP_{m})$ in $\cP_{n}$.
\end{defi}
These subspaces have been studied by Guionnet, Jones and Shlyaktenko in \cite{GJS10}. We have the following decomposition of $\mathcal{P}_{n}$ (see \cite[Theorem 1]{GJS10}).
\begin{equation}\label{resultGJS}
\mathcal{P}_{n}=\bigoplus_{I\in\mathcal{I}(n)}\cB_{\cP}(I),
\end{equation}
where the direct sum is an orthogonal direct sum. Therefore, the collection of subspaces $(\cB_{\cP}(n))_{n\geq 1}$ gives a Boolean decomposition of $\cP$, which we call the orthogonal Boolean decomposition of $\cP$.
\subsection{Surgery on non-crossing partitions}\label{surgeryNoncrossPart}
In this subsection, we describe the behavior of the Kreweras complement of a non-crossing partition when we merge or split blocks. \\
Let $p = \{ B_1, \dots, B_r \}$ be a non-crossing partition. We define a partial order $\preceq$ on $p$ by saying that a block $B=\lbrace i_{1}<\dots<i_{t}\rbrace$ is larger than a block $B'=\lbrace j_{1}<\dots<j_{t'}\rbrace$ if and only if $i_{1}<j_{1}<j_{t'}<i_{t}$: in this case, we write $B'\preceq B$.  We obtain a function $d_{p}:p\longrightarrow \mathbb{N}$, called depth function and defined as follows. If $B\in p$, $d_{p}(B)$ is the maximal length of a chain having $B$ as lower element with respect to $\preceq$.\\
Two blocks $B$ and $B'$ are called adjacent if they can be merged in such a way that the partition is still non-crossing. The following facts are some straightforward deductions of the definition of the order $\preceq$.
\begin{enumerate}
\item If $B$ and $B'$ are adjacent, then either $B \preceq B'$ or $B'\preceq B$ or $d_{p}(B)=d_{p}(B')$.
\item If we merge a block $B$ of $p$ with an adjacent block $B'$ of larger depth, we get a new partition $p' = \lbrace \tilde{B} \rbrace_{\substack{ \tilde{B}\in p\\ \tilde{B} \not=B,B'}}\cup \lbrace B\cup B'\rbrace$. If $B''$ is a different block with a depth smaller than $B$ and $B'$, then $d_{p'}(B'')\leq d_{p}(B'')$. Moreover, $d_{p'}(B\cup B')=d_{p}(B)$. The new partition $p'$ is denoted by $p_{B,B'}$.
\item If we split a block $B$ of $p$ in two interval blocks $B_{1},B_{2}$ (with $B_{1}$ sitting left of $B_{2}$),  we get a new partition $p' = \lbrace \tilde{B} \rbrace_{\substack{\tilde{B} \in p\\ \tilde{B} \not=B}}\cup \lbrace B_{1}, B_{2}\rbrace$. If $B'$ is different from $B$ in $p$, then $d_{p'}(B')\leq d_{p}(B')$. Moreover, $d_{p'}(B_{1})=d_{p'}(B_{2})=d_{p}(B)$. The new partition $p'$ is denoted by $p_{B,i}$, where $i$ is the rank in $B$ of the last element of $B_{1}$ (with respect to the natural order on integers). 
\end{enumerate}
We write $\sigma$ for the subset $\lbrace 1,3,\dots, 2n-1\rbrace$ of $\llbracket 1,2n\rrbracket$. Let $p$ be a partition of $n$; the Kreweras complement $K(p)$ of $p$ is the unique partition of $n$ such that $(K(p),\sigma^{c})$ is the Kreweras complement of $(p,\sigma)$ (see Section \ref{Krewerassec}). This yields a bijective map $K:NC(n)\longrightarrow NC(n)$. The inverse of this map is denoted by $K^{-1}$. In the sequel, the blocks of a partition $p$ are denoted by $B$, whereas the blocks of its Kreweras complement are denoted by $C$. The following result from \cite{NS06} will be useful.
\begin{lem}[\cite{NS06}]\label{KrewerasNumberBlocks}
Let $p$ be a non-crossing partition of $n$. Then, $K(p)$ is the unique partition such that $(p,\sigma)\vee(K(p),\sigma^{c})$ is non-crossing and
$$\vert p\vert +\vert K(p)\vert =n+1.$$
\end{lem}
If $B$ is a block of $p$, then Lemma \ref{characKreweras} yields that any adjacent block of $B$ in $(p,\sigma)\vee(K(p),\sigma^{c})$ belongs to $K(p)$. Moreover, there is a unique adjacent block of $B$ in $(p,\sigma)\vee(K(p),\sigma^{c})$ of depth lower or equal than the one of $B$. Thus, we can define $C^{B}$ without any ambiguity as the unique block of $K(p)$ which is adjacent to $B$ in $(p,\sigma)\vee(K(p),\sigma^{c})$ and has a depth smaller or equal to the one of $B$. The block $C^{B}$ is called the upper enveloping block of $B$. Similarly, we also define $C^{B}_{1},\dots,C^{B}_{r}$ (with $r$ equal to $\vert B\vert-1$) the $r$ adjacent blocks of $B$ in $(p,\sigma)\vee(K(p),\sigma^{c})$ having a larger depth. These blocks are called the lower enveloping blocks of $B$. Examples of upper and lower enveloping blocks of a given block are drawn in Figure \ref{DefinitionBlocks}.  
We define as well the upper enveloping block $B^{C}$ and the lower enveloping blocks $B^{C}_{1},\dots,B^{C}_{r'}$ of a block $C$ of $K(p)$ of size $r'+1$.
\begin{figure}
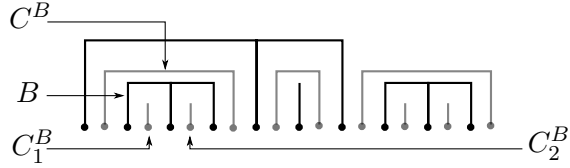
\caption{\label{DefinitionBlocks}Upper and lower enveloping blocks of a given block $B$.}
\end{figure}
\begin{lem}\label{linkKrewerasDecompo}
Let $B$ be a block of $p$ of cardinality $r$ and let $i\in \llbracket 2,r\rrbracket$. Then, 
$$K(p_{B,i})=K(p)_{C^{B},C^{B}_{i}}$$
and 
$$K(p)_{C,i}=K(p_{B^{C},B^{C}_{i}}).$$
\end{lem}
\begin{proof}
By symmetry, it suffices to prove the first equality. By Lemma \ref{KrewerasNumberBlocks}, we only have to prove that $\mathfrak{p}:=(p_{B,i},\sigma)\vee (K(p)_{C^{B},C^{B}_{i}},\sigma^{c})$ is noncrossing and that $\vert p_{B,i}\vert+\vert K(p)_{C^{B},C^{B}_{i}}\vert =n+1$. 

The partition $\mathfrak{p}$ can be constructed as follows. First, consider the partition 
$$p_{1}=\left((p,\sigma)\vee(K(p),\sigma^{c})\right)_{B,i},$$ 
which is non-crossing. Then, $p_{1}$ is also equal to $\left((p_{B,i},\sigma)\vee(K(p),\sigma^{c})\right)$. By definition of the upper and lower enveloping blocks of $B$, $C^{B}$ and $C^{B}_{i}$ are adjacent in $p_{1}$. Thus, $(p_{1})_{C^{B},C^{B}_{i}}$ is non-crossing and $(p_{1})_{C^{B},C^{B}_{i}}=(p_{B,i},\sigma)\vee(K(p)_{C^{B},C^{B}_{i}},\sigma^{c})=\mathfrak{p}$. Therefore, $\mathfrak{p}$ is non-crossing.

Since $\vert p_{B,i}\vert =\vert p\vert +1$ and $\vert K(p)_{C^{B},C^{B}_{i}}\vert=\vert K(p)\vert -1$, we have 
$$\vert p_{B,i}\vert+\vert K(p)_{C^{B},C^{B}_{i}}\vert=\vert p\vert+1+\vert K(p)\vert -1=n+1.$$
Hence, Lemma \ref{KrewerasNumberBlocks} yields that $K(p_{B,i})=K(p)_{C^{B},C^{B}_{i}}$.
 \end{proof}
 \subsection{Boolean decomposition of a free product of planar algebras}
The purpose of this subsection is to improve Proposition \ref{spinReducedPairs} by giving a Boolean decomposition of a free product of planar algebras.\\
Recall from Section \ref{Krewerassec} that $\pi_{0}$ (resp.~$\pi_{1}$) is the non-crossing partition of $2n$ defined by the relation $2i\sim_{\pi_{0}}2i+1$ (resp.~$2i\sim_{\pi_{1}} 2i-1$) for $1\leq i\leq n$.  A partition $\pi$ of $2n$ such that $\pi\geq \pi_{1}$ yields a partition $G(\pi)$ of $n$ defined by the relation $i\sim_{G(\pi)}j$ if and only if $2i\sim_{\pi }2j$. A partition $\pi$ of $2n$ such that $\pi\geq \pi_{0}$ yields a partition $F(\pi)$ of $n$ defined by the relation $i\sim_{F(\pi)}j$ if and only if $2i-1\sim_{\pi}2j-1$. Note that the maps $F$ and $G$ are clearly bijective. If $\pi\geq\pi_{1}$ (resp. $\pi\geq \pi_{0}$) and $B$ is a block of $G(\pi)$ (resp. $F(\pi)$), $\tilde{B}$ denotes the block of $\pi$ containing $2i-1$  and $2i$ (resp. $2i-2$ and $2i-1$) for each $i\in B$ . The maps $F$ and $G$ are depicted in Figure 
\begin{figure}
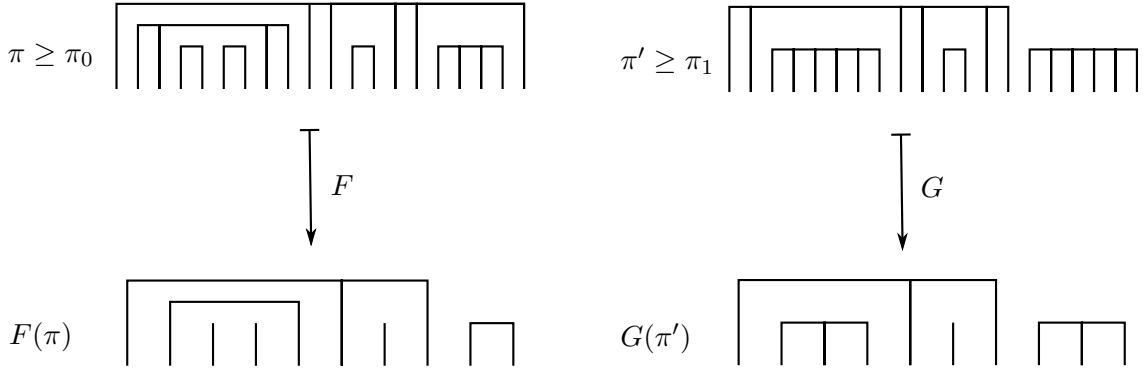
\caption{\label{mapsFG}Maps $F$ and $G$.}
\end{figure}
\begin{lem}\label{divideKreweras}
If $\pi\geq \pi_{0}$, $K(F(\pi))=G(\kr'(\pi))$.
\end{lem}
\begin{proof}
Suppose that $i\sim_{K(F(\pi))} j$. By Lemma \ref{characKreweras}, for all $i< k\leq j,j< l\leq i$, we have $k\not\sim_{F(\pi)}l$; hence by definition of $F$, we have also $2k-1\not\sim_{\pi}2l-1$ for all $i< k\leq j,j< l\leq i$. Since $\pi\geq \pi_{0}$, it follows that $k\not\sim_{\pi} l$ for all $2i-1<k\leq 2j-1, 2j-1<l\leq 2i-1$. \\
Thus, for all $4(i-1)+1<k\leq 4(j-1)+1$, $4(j-1)+1<l\leq 4(i-1)$ with $k,l\in S$, we get $k\not\sim_{(f(\pi),S)}l$, where $f$ is the map $f(i)=2i-\delta(i)$ defined in Section \ref{noncrossNotation}. This implies that $4(i-1)+2\sim_{kr(f(\pi),S)} 4(j-1)+2$ and thus $2(i-1)+1\sim_{\kr'(\pi)}2(j-1)+1$. Since $\kr'(\pi)\geq\pi_{1}$, $2i\sim_{\kr'(\pi)}2j$ and $i\sim_{G(\kr'(\pi))}j$. Thus $K(F(\pi))\leq G(\kr'(\pi))$.
The same proof yields the converse inequality.
\end{proof}

Let $\cP, \cQ$ be subfactor planar algebras.
For $p\in NC(n)$, denote by $\hat{p}$ the map $Z_{T_{F^{-1}(p)}}$ and by $\widehat{\cB}_{\cP}(p)$ the vector subspace of $\cP_{n}$ defined by 
$$\widehat{\cB}_{\cP}(p)=\hat{p}(\cB_{\cP}(i_{1})\otimes\dots\otimes \cB_{\cP}(i_{r})),$$
where $i_{1},\dots,i_{r}$ are the respective cardinalities of the blocks $B_{1},\dots,B_{r}$ of $p$. Similarly, denote by $\tilde{p}$ the map $Z_{T_{G^{-1}(p)}}$ and by $\widetilde{\cB}_{\cQ}(p)$ the vector subspace of $\cQ_{n}$ defined by 
$$\widetilde{\cB}_{\cQ}(p)=\tilde{p}(\cB_{\cQ}(i_{1})\otimes\dots\otimes \cB_{\cQ}(i_{r})),$$
where $i_{1},\dots,i_{r}$ are the respective cardinalities of the blocks $B_{1},\dots,B_{r}$ of $p$. Note that the definition of the vector space $\widetilde{\cB}_{\cP}(I)$ given for the interval partition $I$ in Section \ref{BooleanDecompoPlanarAlgebra} coincides with the one given here.

\begin{thm}\label{main5}
Set $L_{n}=\bigoplus_{p\in NC(n)}\widehat{\cB}_{\cP}(p)\otimes \widetilde{\cB}_{\cQ}(K(p))$. The collection of vector spaces $(L_{n})_{n\geq 1}$ gives a Boolean decomposition of $\cP*\cQ$.
\end{thm}
Note that it is not clear yet that the sum in the definition of $L_{n}$ is direct. The proof of this theorem needs some preliminary results. If $p\in NC(n)$, we denote by $B_{1},\dots,B_{r}$ its blocks and by $C_{1},\dots,C_{r'}$ the blocks of $K(p)$. Moreover, we denote by $p^{\circ}$ the partition $(p,\sigma)\vee(K(p),\sigma^{c})$. 

By Proposition \ref{spinReducedPairs} (and the following discussion), $(\cP*\cQ)_{n}$ is spanned by elements of the form $\hat{p}(v)\otimes \widetilde{K(p)}(w)$ for $p\in NC(n)$, $v=v_{1}\otimes\dots\otimes v_{r}\in \cP_{\vert B_{1}\vert }\otimes\dots\otimes \cP_{\vert B_{r}\vert}$ and $w=w_{1}\otimes \dots\otimes w_{r'}\in\cQ_{\vert C_{1}\vert}\otimes \dots\otimes \cQ_{\vert C_{r'}\vert}$. The goal of the proof will be to reduced the domain of definition of $\hat{p}$ and $\widetilde{K(p)}$ until getting the desired result. Let us define the set 
$$\Sigma \ = \ \lbrace (p,v,w) \ ; \  p\in NC(n),\ v\in \cP_{\vert B_{1}\vert }\otimes\dots\otimes \cP_{\vert B_{r}\vert}, \ w\in\cQ_{\vert C_{1}\vert}\otimes \dots\otimes \cQ_{\vert C_{r'}\vert}\}$$
and the map $\Theta:\Sigma\longrightarrow (\cP*\cQ)_{n}$ with $\Theta(p,v,w)=\hat{p}(v)\otimes \widetilde{K(p)}(w)$. Thus, by Proposition \ref{spinReducedPairs}, $\Theta$ is surjective.

For $(p,v,w)\in\Sigma$, we denote by $S(p,v,w)$ the set of blocks $B_{i}$ of $p$ different from $B_{1}$ such that $v_{i}\not \in \cB_{\cP}(\vert B_{i}\vert)$ and of blocks $C_{i}$ of $K(p)$ such that $w_{i}\not\in \cB_{\cP}(\vert C_{i}\vert)$. The goal of the first part of the proof of Theorem \ref{main5} is to prove that $\Theta$ is still surjective when its domain of definition is restricted to the set of elements $(p,v,w)\Sigma$ such that $S(p,v,w)=\emptyset$. We will prove it by recurrence, and to this end, we define the degree $\vec{d}(p,v,w)=(d_{1},d_{2})\in\N^{2}$ of $(p,v,w)\in\Sigma$ as $\vec{d}(p,v,w)=(0,0)$ if $S(p,v,w)=\emptyset$ and otherwise: 
\begin{itemize}
\item $d_{1}=\max_{X\in S(p,v,w)}d_{p^{\circ}}(X)$.
\item $d_{2}=\sum_{X\in S(p,v,w),d_{p^{\circ}(x)}=d_{1}}\vert X\vert$.
\end{itemize}

For $\vec{d}\in\N^{2}$, we denote by $\Sigma_{\vec{d}}$ the subset of $\Sigma$ consisting of pairs $(p,v,w)$ such that $\vec{d}(p,v,w)=\vec{d}$. Note that when $d_{1}\geq 1$, necessarily $d_{2}>1$, since $\cP_{1}=\cB_{\cP}(1)$ and $\cQ_{1}=\cB_{\cQ}(1)$. In  the following lemma, the set $\N^{2}$ is considered with the lexicographical order. This lemma is the main step of the first part of the proof of Theorem \ref{main5}.
\begin{lem}\label{escape}
Let $\vec{d}\in\N^{2},\vec{d}>(1,0)$. Then,
$$\Theta(\Sigma_{\vec{d}})\subset \underset{\vec{d}'< \vec{d}}{\spann}(\Theta(\Sigma_{\vec{d}'})).$$
\end{lem}
\begin{proof}
Let $\vec{d}\in\N^{2}$ with $\vec{d}>(1,0)$, and let $z\in \Sigma_{\vec{d}}$.  Let $X$ be the first block of $S(p,v,w)$ of depth $d_{1}$ in the lexicographical order. We can suppose without loss of generality that $X=C_{i}$ for some $1\leq i\leq r'$, the case $X=B_{i}$ for $2\leq i\leq r$ being similar. Let $t$ be the cardinality of $C_{i}$. Since $w_{i}\not\in\cB_{\cQ}(t)$, by equation \eqref{resultGJS}, we can write
$$w_{i}=\sum_{I=(I_{1},\dots,I_{l})\in\mathcal{I}(t)} U_{I}(w^{I}_{1}\otimes \dots\otimes w^{I}_{l}),$$
with each $w^{I}_{j}$ in $\cB_{\cQ}(I_{j})$. Using the definition of the operators $U_{I}$ and regrouping terms yields
$$w_{i}=\sum_{j=1}^{t}U_{j,t-j}(w^{j}_{1}\otimes w_{2}^{j}),$$
with $w_{1}^{j}\in \cB_{\cQ}(j)$ and $w_{2}^{j}\in \cQ_{t-j}$. For $1\leq j\leq t$, denote by $w^{\uparrow j}$ the tensor product $w_{1}\otimes \dots\otimes w_{i-1}\otimes U_{j,t-j}(w^{j}_{1}\otimes w_{2}^{j})\otimes w_{i+1}\otimes\dots\otimes w_{r'}$. Then, $w=\sum_{j=1}^{t}w^{\uparrow j}$, and $\Theta(p,v,w)=\sum_{j=1}^{t}\Theta(p,v,w^{\uparrow j})$.\\
Note that $w_{i}^{\uparrow t}\in \cB_{\cQ}(t)$, which implies $C_{i}\not\in S(p,v,w^{\uparrow t})$. Since $S(p,v,w^{\uparrow t})\setminus \lbrace C_{i}\rbrace=S(p,v,w)\setminus \lbrace C_{i}\rbrace$, it follows that $d_{1}(p,v,w^{\uparrow t})\leq d_{1}(p,v,w)$, and when the equality holds we have $d_{2}(p,v,w^{\uparrow t})=d_{2}(p,v,w)-t$. In any case, $\vec{d}(p,v,w^{\uparrow t})<\vec{d}(p,v,w)$.\\
Let $1\leq j<t$. With the definition of  $w^{\uparrow j}$ and $\widetilde{K(p)}$, we have
\begin{align*}
\widetilde{K(p)}(w^{\uparrow j})=&Z_{T_{G^{-1}(K(p))}}(w_{1}\otimes \dots\otimes U_{j,t-j}(w^{j}_{1}\otimes w_{2}^{j})\otimes\dots\otimes w_{r'})\\
=&Z_{T_{G^{-1}(K(p))}\circ_{D_{i}}U_{j,t-j}}(w_{1}\otimes\dots\otimes w^{j}_{1}\otimes w^{j}_{2}\otimes \dots\otimes w_{r'}).
\end{align*}
Inserting the planar tangle $U_{j,t-j}$ in the disc $D_{i}$ of $T_{G^{-1}(K(p))}$ splits the disc corresponding to $\widetilde{C_{i}}$ in two after the distinguished point $2j$. Hence, $T_{G^{-1}(K(p))}\circ_{D_{i}}U_{j,t-j}=T_{G^{-1}(K(p))_{\widetilde{C_{i}},2j}}$, and thus 
\begin{align*}
\widetilde{K(p)}(w^{\uparrow j})=&Z_{T_{G^{-1}(K(p))_{\widetilde{C_{i}},2j}}}(w_{1}\otimes\dots\otimes w^{j}_{1}\otimes \dots\otimes w^{j}_{2}\otimes \dots\otimes w_{r'})\\
=&Z_{T_{G^{-1}(K(p)_{C_{i},j})}}(w_{1}\otimes\dots\otimes w^{j}_{1}\otimes \dots\otimes w^{j}_{2}\otimes \dots\otimes w_{r'})\\
=&\widetilde{K(p)}_{C_{i},j}(w_{1}\otimes\dots\otimes w^{j}_{1}\otimes \dots\otimes w^{j}_{2}\otimes \dots\otimes w_{r'}),
\end{align*}
where the position of $w_{1}^{j}$ and $w_{2}^{j}$ in the above tensor products is given by the position of the two parts of the block $C_{i}$ in the partition $K(p)_{C_{i},j}$.  We denote by $w^{\downarrow j}$ the vector $w_{1}\otimes\dots\otimes w^{j}_{1}\otimes\dots\otimes  w^{j}_{2}\otimes \dots\otimes w_{r'}$. 
Let $k_{1}$ be such that $B_{k_{1}}=B^{C_{i}}$ is the upper enveloping block of $C_{i}$ in $\pi^{\circ}$ and let $k_{2}$ be such that $B_{k_{2}}=B^{C_{i}}_{j}$ is the lower enveloping block of $C_{i}$ between the $j$-th and $j+1$-th element of $C_{i}$. 
For $a,b,c\geq 1$, let $\tilde{M}_{a,b,c}$ be the tangle displayed in Figure \ref{tensorbis}.
\begin{figure}[h!]
\begin{center}
\scalebox{0.7}{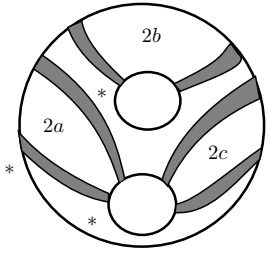}
\end{center}
\caption{\label{tensorbis}: Planar tangle $\tilde{M}_{a,b,c}$.}
\end{figure} 

Since $B_{k_{1}}$ and $B_{k_{2}}$ are adjacent in $p$, then by a reasoning similar to the one giving the expression of $\widetilde{K(p)}(w^{\uparrow j})$ we have
$$\hat{p}(v)=\hat{p}_{B_{k_{1}},B_{k_{2}}}(v_{1}\otimes \dots\otimes \tilde{v}_{k_{1}}\otimes \dots\otimes \cancel{v_{k_{2}}}\otimes \dots\otimes v_{r}),$$
where $\tilde{v}_{k_{1}}=Z_{\tilde{M}_{s,\vert B_{k_{2}}\vert,\vert B_{k_{1}}\vert-s}}(v_{k_{1}},v_{k_{2}})$ and $s$ is the number of elements in $B_{k_{1}}$ before $B_{k_{2}}$ in $p$. We denote by $v^{\downarrow j}$ the vector $v_{1}\otimes \dots\otimes \tilde{v}_{k_{1}}\otimes \dots\otimes \cancel{v_{k_{2}}}\otimes \dots\otimes v_{r}$.\\
By Lemma \ref{linkKrewerasDecompo}, $K(p_{B_{k_{1}},B_{k_{2}}})=K(p)_{C_{i},j}$, hence 
$$\Theta(p,v,w^{\uparrow j})=\Theta(p_{B_{k_{1}},B_{k_{2}}},v^{\downarrow j},w^{\downarrow j}).$$ 
Denote by $C_{i}^{1}$ (resp. $C_{i}^{2}$) the block of $K(p)_{C_{i},j}$ containing the first (resp. last) elements of $C_{i}$. Then, we observe that 
$$S(p,v,w)\setminus \lbrace C_{i},B_{k_{1}},B_{k_{2}}\rbrace=S(p_{B_{k_{1}},B_{k_{2}}},v^{\downarrow j},w^{\downarrow j})\setminus \lbrace B_{k_{1}}\cup B_{k_{2}},C_{i}^{1},C_{i}^{2}).$$ Since $C_{i}\preceq B_{k_{1}}$ in $p^{\circ},$ we have $d_{(p_{B_{k_{1}},B_{k_{2}}})^{\circ}}(B_{k_{1}}\cup B_{k_{2}})=d_{p^{\circ}}(B_{k_{1}})< d_{p^{\circ}}(C_{i})$. Moreover, $d_{(p_{B_{k_{1}},B_{k_{2}}})^{\circ}}(C_{i}^{1})=d_{(p_{B_{k_{1}},B_{k_{2}}})^{\circ}}(C_{i}^{2})=d_{p^{\circ}}(C_{i})$. Therefore, we conclude that $d_{1}(p,v,w)\geq d_{1}(p_{B_{k_{1}},B_{k_{2}}},v^{\downarrow j},w^{\downarrow j})$; whenever both sides are equal, the fact that $C_{i}^{1}$ does not belong to $S(p_{B_{k_{1}},B_{k_{2}}},v^{\downarrow j},w^{\downarrow j})$ yields that
$$d_{2}(p,v,w)\geq d_{2}(p_{B_{k_{1}},B_{k_{2}}},v^{\downarrow j},w^{\downarrow j})+j.$$ 
Therefore, $\vec{d}(p_{B_{k_{1}},B_{k_{2}}},v^{\downarrow j},w^{\downarrow j})<\vec{d}(p,v,w)$, and we have shown that $\Theta(p,v,w^{\uparrow j})\subset \bigcup_{\vec{d}'< \vec{d}}\Theta(\Sigma_{\vec{d}'})$.\\
Summing over all $1\leq j\leq t$ yields that $\Theta(p,v,w)\subset\underset{\vec{d}'< \vec{d}}{\spann}\Theta(\Sigma_{\vec{d}'})$.
\end{proof}
The second part of the proof of Theorem \ref{main5} will amount to identify the elements in $\Sigma$ which span the $n$-th vector space of the Boolean decomposition of $\cP*\cQ$ that we are looking for. The following lemma is the main step in this identification: it shows that a tensor product decomposition of $\Theta(p,v,w)$ appears when $v_{1}\not\in \cB_{\cP}(\vert B_{1}\vert)$.
\begin{lem}\label{decompoFreeProduct}
Suppose that $(p,v,w)\in \Sigma_{(0,0)}$ is such that $v_{1}=U_{s,\vert B_{1}\vert-s}(x\otimes y)$ for some $1<s<\vert B_{1}\vert$ and $x\in \cB_{\cP}(s),y\in \cP_{\vert B_{1}\vert-s}$. Then, there exists $k_{1},k_{2}>0$ such that 
$$\Theta(p,v,w)=U_{k_{1},k_{2}}(\Theta(p_{1},v^{1},w^{1})\otimes \Theta(p_{2},v^{2},w^{2})),$$
with $(p_{1},v^{1},w^{1}),(p_{2},v^{2},w^{2})\in\Sigma_{(0,0)}$ and $\hat{p}_{1}(v^{1})\in\widehat{\cB}_{\cP}(p_{1})$.
\end{lem}
\begin{proof}
Let $p\in NC(n)$, and let $t$ be the cardinality of $B_{1}$. Suppose that $(p,v,w)\in \Sigma_{(0,0)}$ is such that $v_{1}=U_{s,t-s}(x\otimes y)$ for some $1<s<t$ and $x\in \cB_{\cP}(s),y\in \cP_{t-s}$. Let $k$ be the $s+1$-th element of $B_{1}$.\\
On one hand we have 
\begin{align*}
\hat{p}(v_{1}\otimes \dots\otimes v_{r})=&\hat{p}(U_{s,t-s}(x\otimes y)\otimes \dots\otimes v_{r})\\
=&\hat{p}_{B_{1},s}(x\otimes \dots\otimes v_{i} \otimes y\otimes \dots\otimes v_{r}),
\end{align*}
where $B_{i}$ is the block ending just before $k$.
Since $p$ is non-crossing and $1,k\in B_{1}$, for all $z\geq k$ and $z'\leq k-1$ such that $z\not\in B_{1}$ we have $z\not\sim_{p}z'$. Hence, $p_{B_{1},s}$ is the juxtaposition of two partition $p_{1}$ of order $k-1$ and $p_{2}$ of order $n-k+1$. Therefore,
$$\hat{p}(v_{1}\otimes \dots\otimes v_{r})=U_{k-1,n-k+1}(\hat{p}_{1}(x\otimes \dots\otimes v_{i}),\hat{p_{2}}(y\otimes \dots\otimes v_{r})).$$
On the other hand, since $k\sim_{p}1$, we have $z\not\sim_{K(p)}z'$ for $z\leq k-1,z'\geq k$ by Lemma \ref{characKreweras}. Thus, $K(p)$ is the juxtaposition of two partitions $q_{1}$ of order $k-1$ and $q_{2}$ of order $n-k+1$. Therefore,
$$\widetilde{K(p)}(w)=U_{k-1,n-k+1}(q_{1}(w_{1}\otimes \dots\otimes w_{j})\otimes q_{2}(w_{j+1}\otimes \dots \otimes w_{r'})),$$
where $j$ is the last block of $K(p)$ before the element $k$.\\
Let us write $v^{1}=x\otimes \dots\otimes v_{i}$, $v^{2}=y\otimes \dots\otimes v_{r}$, $w^{1}=w_{1}\otimes \dots\otimes w_{j}$ and $w^{2}=w_{j+1}\otimes \dots \otimes w_{r'}$. Lemma \ref{characKreweras} yields that $q_{1}=K(p_{1})$ and $q_{2}=K(p_{2})$, thus
\begin{align*}
\Theta(p,v,w) \ &= \ U_{k-1,n-k+1}((\hat{p}_{1}(v^{1})\otimes \tilde{q}_{1}(w^{1})) \otimes (\hat{p_{2}}(v^{2})\otimes \tilde{q}_{2}(w^{2}))\\
&= \ U_{k-1,n-k+1}(\Theta(p_{1},v^{1},w^{1}) \otimes \Theta(p_{2},v^{2},w^{2})).
\end{align*}
For all $a\geq 2$ we have $v_{a}\in \cB_{\cP}(\vert B_{a}\vert)$ and for all $b\geq 1$ we have $w_{b}\in\cB_{\cQ}(\vert C_{b}\vert)$, thus $(p_{1},v^{1},w^{1})$ and $(p_{2},v^{2},w^{2})$ are in $\Sigma_{(0,0)}$. Since $x\in \cB_{\cP}(s)$, $\hat{p}_{1}(v^{1})\in\widehat{\cB}_{\cP}(p_{1})$.
\end{proof}
We can now turn to the proof of Theorem \ref{main5}.
\begin{proof}[Proof of Theorem \ref{main5}]
We denote by $L_{n}$ the vector space $\sum_{p\in NC(n)}\widehat{\cB}_{\cP}(p)\otimes\widetilde{\cB}_{\cQ}(K(p))$. We have to prove that the sum in the definition of $L_{n}$ is a direct sum, and that
$$(\cP*\cQ)_{n}=\bigoplus_{I\in\mathcal{I}(n)}L(I).$$
By Proposition \ref{spinReducedPairs}, $(\cP*\cQ)_{n}$ is spanned by the set $\Theta(\Sigma)$. Lemma \ref{escape} yields that the vector space spanned by $\Theta(\Sigma)$ is equal to the vector space spanned by $\Theta(\Sigma_{(0,0)})$. Therefore, 
$$(\cP*\cQ)_{n}=\underset{\substack{p\in NC(n),w_{i}\in \cB_{\cQ}(\vert C_{i}\vert),i\geq 1\\v_{i}\in\cB_{\cP}(\vert B_{i}\vert),i\geq 2}}{\spann}\Theta(p,v,w).$$
Let $(p,v,w)\in \Sigma_{(0,0)}$ with $p\in NC(n)$. By \eqref{resultGJS}, we can write $v=\sum_{I\in\mathcal{I}(\vert B_{1}\vert )} v^{I}$, with $v^{I}\in\cB_{\cP}(I)$. Let $I=(i_{1},\dots,i_{r})\in\mathcal{I}(\vert B_{1}\vert )$. Applying Lemma \ref{decompoFreeProduct} $r-1$ times yields the existence of $K=(k_{1},\dots,k_{r})\in\mathcal{I}(n)$ and $(p_{i},v^{i},w^{i})\in \Sigma_{(0,0)}$ for each $1\leq i\leq r$ such that $p_{i}\in NC(k_{i})$, $\hat{p}_{i}(v^{i})\in\widehat{\cB}_{\cP}(p_{i})$, and such that
$$\Theta(p,v^{I},w)=U_{K}(\Theta(p_{1},v^{1},w^{1})\otimes\dots\otimes\Theta(p_{r},v^{r},w^{r})).$$
Since each $\Theta(p_{i},v^{i},w^{i})\in \widehat{\cB}_{\cP}(p_{i})\otimes \widetilde{\cB}_{\cQ}(K(p_{i}))$, it follows that $\Theta(p,v^{I},w^{I})\in L(K)$. Summing on all $I\in \mathcal{I}(\vert B_{1}\vert )$ yields that $\Theta(p,v,w)\in\sum_{K\in\mathcal{I}(n)}L(K)$. Finally, we have $(\cP*\cQ)_{n}\subset \sum_{K\in\mathcal{I}(n)}L(K)$, and since it is readily seen that $\sum_{K\in\mathcal{I}(n)}L(K)\subset (\cP*\cQ)_{n}$, we obtain
\begin{equation}\label{inclusionVectorSpace}
(\cP*\cQ)_{n}=\sum_{K\in\mathcal{I}(n)}L(K).
\end{equation}
It remains to prove that the sums involved in the theorem are direct sums. This will be proven by a dimension argument. Denote by $\tilde{b}_{\cP*\cQ}(n)$ the dimension of $L(n)$. On the one hand, the definition of $L(n)$ yields that $\tilde{b}_{\cP*\cQ}(n)\leq\sum_{p\in NC(n)}b_{\mu_{\cP}}(p)b_{\mu_{\cQ}}(K(p))$ for $n\geq 1$. On the other hand, by a result of Bisch and Jones in \cite{BJXX} (see also \cite[Chapter 7.3]{Ta15} for a combinatorial proof of this fact, which relies on the computation of the principal graph given in \cite{L02}), we have 
$$(\mu_{\cP*\cQ})_{n}=(\mu_{\cP}\boxtimes \mu_{\cQ})_{n}.$$
The latter result together with \eqref{booleanFreeNB} implies that
\begin{align*}
b_{\mu_{\cP*\cQ}}(n)=&b_{\mu_{\cP}\boxtimes\mu_{\cQ}}(n)\\
=&\sum_{p\in NC(n)}b_{\mu_{\cP}}(p)b_{\mu_{\cQ}}(K(p))\geq\tilde{b}_{\cP*\cQ}(n).
\end{align*}
Hence, 
\begin{align*}
\dim (\cP*\cQ)_{n}=&(\mu_{\cP*\cQ})(n)=\sum_{K\in\mathcal{I}(n)}b_{\mu_{\cP*\cQ}}(K)\\
\geq&\sum_{K\in\mathcal{I}(n)}\tilde{b}_{\mu_{\cP*\cQ}}(K)=\sum_{K\in\mathcal{I}(n)}\dim L(K)\geq\dim (\cP*\cQ)_{n} ,
\end{align*}
the latter equality being due to \eqref{inclusionVectorSpace}. Therefore, each inequality is an equality. This implies that $L_{n}=\bigoplus_{p\in NC(n)}\widehat{\cB}_{\cP}(p)\otimes\widetilde{\cB}_{\cQ}(K(p))$ and that we have the direct sum decomposition
$$(\cP*\cQ)_{n}=\bigoplus_{K\in\mathcal{I}(n)}L(K).$$
\end{proof}

Theorem \ref{main5} yields a natural basis of the fixed point spaces for a free wreath product of  two compact quantum groups $(\F, \al), \ (\G, \beta)$ which act on finite-dimensional $C^*$-algebras in the usual way. Using the notations of Section $7$, let us fix a basis $(\cF_{n})_{n\geq 1}$ (resp. $(\cG_{n})_{n\geq 1}$) for the collection of vector space $(\cB_{\cP(\alpha)}(n))_{n\geq 1}$ (resp. $(\cB_{\cP(\beta)}(n))_{n\geq 1}$). For $p\in NC(k)$, denote by $\widehat{\cF}(p)$ the subset of $\cP(\alpha)_{k}$ defined by 
$$\widehat{\cF}(p)=\hat{p}(\cF_{i_{1}}\otimes\dots\otimes \cF_{i_{r}}),$$
where $i_{1},\dots,i_{r}$ are the respective sizes of the blocks of $p$. Similarly, denote by $\widetilde{\cG}(p)$ the subset of $\cP(\beta)_{k}$ defined by 
$$\widetilde{\cG}(p)=\tilde{p}(\cG_{i_{1}}\otimes\dots\otimes \cG_{i_{r}}).$$
\begin{cor}
A basis of $\Fix((\beta\wr_{*}\alpha)^{n})$ is given by the set 
$$\left\lbrace U_{k_{1}}^{*}\left(\widehat{\cF}(p_{1})\otimes\widetilde{\cG}(K(p_{1}))\right)\otimes \dots \otimes U_{k_{r}}^{*}\left(\widehat{\cF}(p_{r})\otimes \widetilde{\cG}(K(p_{r}))\right)\Big\vert \vec{k}\in \mathcal{I}(n), p_{i}\in NC(k_{i})\right\rbrace.$$
\end{cor}
\begin{proof}
This is a straightforward consequence of Theorem \ref{main3} and Theorem \ref{main5}.
\end{proof}

\end{document}